\documentclass{amsart}
\usepackage{amssymb,amsmath}
\usepackage{stmaryrd}
\usepackage{yhmath}
\usepackage{latexsym}
\usepackage{amsthm}
\usepackage{amscd}
\usepackage{mathrsfs}
\usepackage[all]{xy}
\usepackage{tikz}
\usetikzlibrary{cd}
\usepackage{mathtools}
\usepackage{arydshln}
\usepackage{bbm}
\usepackage{bm}
\usepackage{url}
\usepackage{color}

\newcommand{\sym}{\mathrm{sym}}
\newcommand{\asym}{\mathrm{asym}}
\newcommand{\ur}{\mathrm{ur}}
\newcommand{\ram}{\mathrm{ram}}
\newcommand{\red}{\mathrm{red}}
\newcommand{\ad}{\mathrm{ad}}

\newcommand{\Kal}{\mathrm{Kal}}
\newcommand{\Tam}{\mathrm{Tam}}
\newcommand{\Yu}{\mathrm{Yu}}
\newcommand{\KY}{\mathrm{KY}}
\newcommand{\FKS}{\mathrm{FKS}}

\newcommand{\BH}{\mathrm{BH}}
\newcommand{\Art}{\mathrm{Art}}
\newcommand{\rec}{\mathrm{rec}}
\newcommand{\lan}{\langle}
\newcommand{\ran}{\rangle}
\newcommand{\coh}{\mathrm{coh}}
\newcommand{\sep}{\mathrm{sep}}
\newcommand{\arith}{\mathrm{arith}}
\newcommand{\geom}{\mathrm{geom}}
\newcommand{\pr}{\mathrm{pr}}
\newcommand{\ab}{\mathrm{ab}}
\newcommand{\et}{\mathrm{et}}
\newcommand{\Trd}{\mathrm{Trd}}
\newcommand{\Nrd}{\mathrm{Nrd}}
\newcommand{\op}{\mathrm{op}}
\newcommand{\w}{\mathrm{w}}
\newcommand{\cl}{\mathrm{cl}}

\newcommand{\Jac}[2]{\begin{pmatrix}\frac{#1}{#2}\end{pmatrix}}

\newcommand{\Gm}{\bbG_{\mathrm{m}}}
\newcommand{\eps}[2]{{}_{#1}\epsilon_{#2}}
\newcommand{\vmod}[2]{{}_{#1}\frakV_{#2}}
\newcommand{\bvmod}[2]{{}_{#1}\boldsymbol{\frakV}_{#2}}
\newcommand{\bmu}{\boldsymbol{\mu}}
\newcommand{\bGamma}{\boldsymbol{\Gamma}}
\newcommand{\bfrakU}{\boldsymbol{\frakU}}
\newcommand{\ordx}[2]{{}_{#1}\ord_{x}(#2)}

\newcommand{\bfB}{\mathbf{B}}
\newcommand{\bfe}{\mathbf{e}}
\newcommand{\bfS}{\mathbf{S}}
\newcommand{\bfG}{\mathbf{G}}
\newcommand{\bfJ}{\mathbf{J}}

\newcommand{\bbC}{\mathbb{C}}

\newcommand{\bbF}{\mathbb{F}}
\newcommand{\bbG}{\mathbb{G}}

\newcommand{\bbQ}{\mathbb{Q}}
\newcommand{\bbR}{\mathbb{R}}

\newcommand{\bbZ}{\mathbb{Z}}

\newcommand{\calA}{\mathcal{A}}
\newcommand{\calB}{\mathcal{B}}
\newcommand{\calC}{\mathcal{C}}
\newcommand{\calD}{\mathcal{D}}

\newcommand{\calJ}{\mathcal{J}}

\newcommand{\calO}{\mathcal{O}}

\newcommand{\calT}{\mathcal{T}}

\newcommand{\calX}{\mathcal{X}}

\newcommand{\frakA}{\mathfrak{A}}

\newcommand{\frakg}{\mathfrak{g}}
\newcommand{\frakH}{\mathfrak{H}}

\newcommand{\frakJ}{\mathfrak{J}}
\newcommand{\frakK}{\mathfrak{K}}

\newcommand{\frakP}{\mathfrak{P}}
\newcommand{\frakp}{\mathfrak{p}}

\newcommand{\frakU}{\mathfrak{U}}
\newcommand{\frakV}{\mathfrak{V}}

\DeclareMathOperator{\tr}{tr}
\DeclareMathOperator{\Nr}{Nr}
\DeclareMathOperator{\Tr}{Tr}
\DeclareMathOperator{\depth}{depth}
\DeclareMathOperator{\sgn}{sgn}
\DeclareMathOperator{\ord}{ord}
\DeclareMathOperator{\GL}{GL}

\DeclareMathOperator{\cInd}{c-Ind}
\DeclareMathOperator{\Ind}{Ind}
\DeclareMathOperator{\Hom}{Hom}
\DeclareMathOperator{\End}{End}
\DeclareMathOperator{\Aut}{Aut}
\DeclareMathOperator{\Ker}{Ker}
\DeclareMathOperator{\Coker}{Coker}
\DeclareMathOperator{\Res}{Res}
\DeclareMathOperator{\Gal}{Gal}
\DeclareMathOperator{\Ad}{Ad}

\DeclareMathOperator{\id}{id}
\DeclareMathOperator{\LLC}{LLC}
\DeclareMathOperator{\LJLC}{LJLC}
\DeclareMathOperator{\inv}{inv}
\DeclareMathOperator{\Br}{Br}
\DeclareMathOperator{\ch}{ch}

\theoremstyle{plain}
\newtheorem{thm}{Theorem}[section]
\newtheorem*{thm*}{Theorem}
\newtheorem{prop}[thm]{Proposition}

\newtheorem{cor}[thm]{Corollary}

\theoremstyle{definition}
\newtheorem{defn}[thm]{Definition}

\theoremstyle{remark}
\newtheorem{rem}[thm]{Remark}
\newtheorem*{claim*}{Claim}

\numberwithin{equation}{section}
\title{Local Jacquet--Langlands correspondence for regular supercuspidal representations}

\author{Kazuki Tokimoto}
\address{Institute of Mathematics, Academia Sinica, Astronomy-Mathematics Building, No.\ 1, Sec.\ 4, Roosevelt Road, Taipei 10617, Taiwan.}
\email{tokimoto@gate.sinica.edu.tw}

\begin{document}

\begin{abstract}
We prove that Kaletha's local Langlands correspondence for regular supercuspidal representations gives the classical local Jacquet--Langlands correspondence due to Deligne--Kazhdan--Vign\'{e}ras and Badulescu.
As in a former joint paper with Oi, where a similar result is proved for the local Langlands correspondence for the general linear group, the key ingredients in our proof are the work of Bushnell--Henniart explicitly describing the local Jacquet--Langlands correspondence for essentially tame supercuspidal representations and its reinterpretation due to Tam in terms of Langlands--Shelstad's \(\zeta\)-data.
While, under suitable assumptions, this result follows from more general theorems either in the recent work of Fintzen--Kaletha--Spice or that of Chan--Oi,
our proof does not require any assumptions.

We also complement a few points on the proof in the former paper with Oi.
\end{abstract}

\subjclass[2010]{}
\keywords{}

\maketitle

\section*{Introduction}\label{sec:intro}
For a non-archimedean local field \(F\) and a connected reductive group \(\bfG\) over \(F\), the conjectural local Langlands correspondence associates to each irreducible smooth representation of \(\bfG(F)\) an \(L\)-parameter of \(\bfG\).
Recently Kaletha \cite{MR4013740} defines the notion of regularity for supercuspidal representations for tamely ramified reductive groups \(\bfG\) and constructs the local Langlands correspondence for (most) regular supercuspidal representations.
Notably, his construction of the local Langlands correspondence is explicit and works uniformly for groups.

In a former joint paper \cite{MR4206603} with Masao Oi, we prove that
Kaletha's local Langlands correspondence coincides with Harris--Taylor's local Langlands correspondence for \(\GL_n\).
In this paper we prove a similar theorem for the local Jacquet--Langlands correspondence for inner forms of \(\GL_n\)
due to Deligne--Kazhdan--Vign\'{e}ras \cite{MR771672} and Badulescu \cite{MR1951441}.
Let \(\bfG\) be an inner form of \(\bfG^{\ast}:=\GL_n\) and let us fix a central simple algebra \(A\) over \(F\) such that \(\bfG(F)=A^{\times}\).
Then the local Jacquet--Langlands correspondence is a natural bijection \(\LJLC^{\cl}\) between the irreducible discrete series representations of \(\bfG^{\ast}(F)\) and of \(\bfG(F)\) (we add the superscript ``cl" to indicate it is the classical one).
Let us write \(\LJLC^{\Kal}\) for the local Jacquet--Langlands correspondence naturally induced by Kaletha's local Langlands correspondence.
It is a bijection between the regular supercuspidal representations of \(\bfG^{\ast}(F)\) and of \(\bfG(F)\).

The main theorem of this paper is the following (note that the assumption on the residue characteristic is necessary for Kaletha's theory):
\begin{thm}[Theorem \ref{thm:main}] \label{thm:intro}
Assume that the residue characteristic \(p\) of \(F\) is odd.
Let \(\pi\) be a regular supercuspidal representation of \(\bfG^{\ast}(F)=\GL_n(F)\).
Then we have
\[
\LJLC^{\cl}(\pi)=\LJLC^{\Kal}(\pi).
\]
\end{thm}

We must remark that Theorem \ref{thm:intro} follows from more general results in Fintzen--Kaletha--Spice \cite{2106.09120} or Chan--Oi \cite{2301.09812},
if we assume that the characteristic \(\ch F\) of \(F\) is zero and \(p\) is sufficiently large, or if we make a suitable assumption on a field extension \(E/F\) of degree \(n\) associated to \(\pi\).
Theorem \ref{thm:intro} holds without any such assumptions, essentially because none of the results we use, especially \cite{MR2848585}, makes any assumptions
(see Remark \ref{rem:main} for more details).

Let us also remark on the relation with the more recent work \cite{1912.03274} of Kaletha.
In \cite{1912.03274} the theory in \cite{MR4013740} is extended to a larger class of supercuspidal representations.
In particular, the notion of \(F\)-non-singularity, which is more general than that of regularity, is introduced and the \(L\)-parameters of \(F\)-non-singular supercuspidal representations are constructed.
One may wonder if the same theorem holds for \(F\)-non-singular supercuspidal representations.
However, that would yield nothing new because, for \(\GL_n(F)\), any \(F\)-non-singular supercuspidal representation is regular (see Remark \ref{rem:F-non-sing}).

The proof of Theorem \ref{thm:intro} proceeds similarly to the previous joint paper \cite{MR4206603} and relies on Bushnell--Henniart \cite{MR2848585} and Tam \cite{MR3505131}.
Let us briefly outline it.
Regular supercuspidal representations are parametrized by (the conjugacy classes of) \emph{tame elliptic regular pairs} \((\bfS, \xi)\), which are certain pairs of a tamely ramified elliptic maximal torus \(\bfS\) and a character \(\xi\) of \(\bfS(F)\).
Since tame elliptic regular pairs of \(\bfG^{\ast}\) and \(\bfG\) are naturally identified, for such a pair \((\bfS, \xi)\), we let \({}_{\bfG^{\ast}}\pi_{(\bfS, \xi)}\) and \({}_{\bfG}\pi_{(\bfS, \xi)}\) denote the corresponding regular supercuspidal representations of \(\bfG^{\ast}(F)=\GL_n(F)\) and \(\bfG(F)\).
When working out the definition of \(\LJLC^{\Kal}\), one finds that
\begin{equation} \label{eq:LJLC^Kal}
\LJLC^{\Kal}({}_{\bfG^{\ast}}\pi_{(\bfS, \xi)})={}_{\bfG}\pi_{(\bfS, \xi \cdot \eps{\bfG^{\ast}}{}^{-1} \cdot \eps{\bfG}{})},
\end{equation}
where \(\eps{\bfG^{\ast}}{}\) (resp.\ \(\eps{\bfG}{}\)) is a tamely ramified character of \(\bfS(F)\) associated to the tame elliptic regular pair \((\bfS, \xi)\) of \(\bfG^{\ast}\) (resp.\ of \(\bfG\)) in \cite{MR4013740}.
In fact, for the groups \(\bfG^{\ast}, \bfG\), the parametrization was previously known and the relation with \(\LJLC^{\cl}\) was studied by Bushnell--Henniart \cite{MR2848585} (though with different terminology).
In \cite{MR2848585}, a tamely ramified character \(\nu_{\rec}\) (``rectifier'') depending on \((\bfS, \xi)\) is given such that
\[
\LJLC^{\cl}({}_{\bfG^{\ast}}\pi_{(\bfS, \xi)})={}_{\bfG}\pi_{(\bfS, \xi \cdot \nu_{\rec}^{-1})}.
\]
Moreover, Tam \cite{MR3505131} gives a reinterpretation of \(\nu_{\rec}\) by constructing an explicit set \(\zeta_{\Tam}=\{\zeta_{\Tam, \alpha}\}_{\alpha}\) of \(\zeta\)-data related to \(\nu_{\rec}\).
A set of \(\zeta\)-data introduced by Langlands--Shelstad \cite{MR909227} is a set \(\zeta=\{\zeta_{\alpha}\}_{\alpha}\) of certain characters indexed by the roots \(\alpha\) of \(\bfS\) and Tam's \(\zeta\)-data \(\zeta_{\Tam}\) yields a factorization of \(\nu_{\rec}\) as a product of characters of \(\bfS(F)\) indexed by the Galois orbits of the roots of \(\bfS\).
In fact, the characters \(\eps{\bfG^{\ast}}{}\) and \(\eps{\bfG}{}\) are defined as similar products of characters \(\eps{\bfG^{\ast}}{\alpha}\) and \(\eps{\bfG}{\alpha}\) indexed by certain Galois orbits of the roots of \(\bfS\).
Therefore, in order to prove Theorem \ref{thm:intro}, it suffices to show suitable equalities involving \(\zeta_{\Tam, \alpha}, \eps{\bfG^{\ast}}{\alpha}, \eps{\bfG}{\alpha}\).
However, the conventions or the ``languages'' adopted are considerably different; in \cite{MR2848585}, \cite{MR3505131} certain orders in the central simple algebra \(A\) play an important role in constructing and studying the representations \({}_{\bfG^{\ast}}\pi_{(\bfS, \xi)}\) and eventually defining the characters \(\nu_{\rec}, \zeta_{\Tam, \alpha}\), whereas in \cite{MR4013740} the Bruhat--Tits theory and the root-theoretic techniques are extensively used throughout.
Thus, proving the necessary equalities is not as straightforward as it may seem at first sight.

In addition to proving Theorem \ref{thm:intro}, we complement the following two points in the proof of the main theorem in our previous joint paper \cite{MR4206603}
(see Remark \ref{rem:complement} for more details):
\begin{itemize}
\item As is well-known, there are two common normalizations for the Artin reciprocity map and the local Langlands correspondence.
Although in \cite{MR4206603} we do not pay much attention to the different normalizations used by Kaletha, Langlands--Shelstad, Bushnell--Henniart and Tam,
the proof can be justified as follows:
\begin{itemize}
\item fortunately, Bushnell--Henniart's theorem (Theorem \ref{thm:etLLC}) and Tam's proposition (Proposition \ref{prop:Tam}), which are the two key facts for the passage from the automorphic side to the Galois side, hold verbatim also in the other normalization;
\item the main computation in \cite{MR4206603} is entirely done on the automorphic side and thus the normalization is not relevant
\end{itemize}
 and hence the result is still correct.

\item While we assume the characteristic of \(F\) to be zero in \cite{MR4206603}, the proof works even when the characteristic is positive, essentially because all the results we cite are valid regardless of the characteristic.
\end{itemize}

This paper is organized as follows.
In Section \ref{sec:notation} we summarize some of the standard notation and conventions used in this paper.
In Section \ref{sec:review-of-BH-Tam} we review Bushnell--Henniart's description of \(\LJLC^{\cl}\) for regular supercuspidal representations (or \emph{essentially tame supercuspidal representations} in their terminology) of \(\bfG^{\ast}(F)\) in terms of the rectifier
and Tam's factorization result of the rectifier.
As we need to discuss the previous paper \cite{MR4206603} later, we also explain similar results for the local Langlands correspondence \(\LLC_{\GL_n}^{\cl}\) for \(\bfG^{\ast}=\GL_n\).

In Section \ref{sec:review-of-Kal} we review the relevant parts of Kaletha's theory from \cite{MR4013740}.
More specifically, we give an overview of the parametrization of regular supercuspidal representations and how it is applied to give a partial explicit construction of the local Langlands correspondence for general tamely ramified connected reductive groups, and specialize to the situation where the group is an inner form \(\bfG\) of \(\bfG^{\ast}=\GL_n\).
Our goal in this section is to explain the equality \eqref{eq:LJLC^Kal} and a similar expression for Kaletha's local Langlands correspondence \(\LLC_{\GL_n}^{\Kal}\) for \(\bfG^{\ast}=\GL_n\).

In Section \ref{sec:main-thm} we state Theorem \ref{thm:main} (=Theorem \ref{thm:intro} above). Also, having recalled the necessary concepts and terminology up to this point, we complement the above two points in the argument of the previous joint paper \cite{MR4206603}.
In Section \ref{sec:prelim1} we recall the definition of the characters \(\eps{\bfG^{\ast}}{\alpha}, \eps{\bfG}{\alpha}\) in \cite{MR4013740}
and also compute them in one case by essentially computing the toral invariant, a sign invariant introduced in \cite{MR3402796}, for our groups.
In Section \ref{sec:prelim2} we recall the ingredients necessary for the definition of the characters \(\zeta_{\Tam, \alpha}\).
In Section \ref{sec:pf-main-thm} we finally prove Theorem \ref{thm:main} by showing the equalities involving the characters \(\eps{\bfG^{\ast}}{\alpha}, \eps{\bfG}{\alpha}, \zeta_{\Tam, \alpha}\).

In Appendix \ref{sec:app} we take up a fact which has been implicit so far.
When explicitly describing \(\LJLC^{\Kal}\), Bushnell--Henniart and Tam use the construction of supercuspidal representations of \(\bfG(F)\) from tame elliptic regular pairs (or \emph{admissible pairs} in their terminology) given in \cite{MR2848585}.
It is probably well-known to experts but not entirely trivial that this construction coincides with the one given by Kaletha \cite{MR4013740} (which builds on Yu \cite{MR1824988}).
We give a sketch proof of this coincidence in Appendix \ref{sec:app}, which justifies writing \({}_{\bfG}\pi_{(\bfS, \xi)}\) unambiguously.

\medbreak
\noindent{\bfseries Acknowledgment.}\quad
The author spent an unusually long time on writing this paper and received much encouragement along the way from many people including Yao Cheng, Naoki Imai, Yoichi Mieda, Masao Oi, Nobuo Sato, Takeshi Tsuji and Chia-Fu Yu.
He wishes to thank them all heartily, especially Tsuji and Yu.
He would also like to thank Wen-Wei Li, Alexander Bertoloni Meli and Geo Kam-Fai Tam for their valuable comments on the previous paper \cite{MR4206603}.
He is grateful to Tasho Kaletha for providing helpful answers to questions.
He also thanks Institute of Mathematics, Academia Sinica for their support and excellent research environment and acknowledges the support by NSTC grant 111-2811-M-001-024.

\setcounter{tocdepth}{2}
\tableofcontents

\section{Notation and conventions} \label{sec:notation}
Let us summarize some of the notation and conventions used throughout the paper.

\begin{description}
\item[Local field] Let \(F\) be a non-archimedean local field.
We let \(\calO_F, \frakp_F, k_F:=\calO_F/\frakp_F\) denote the valuation ring of \(F\), its maximal ideal and the residue field.
We write \(p, q\) for the characteristic and the cardinality of \(k_F\).
We impose no assumption on the characteristic \(\ch F\) of \(F\); it is either \(0\) or \(p\).
Let \(\mu_F\) be the group of roots of unity of \(F\) of order prime to \(p\).
We write \(v_F\) for the normalized valuation of \(F\).

We also often use the following common notation for some subgroups of \(F^{\times}\):
\[
U_F:=U_F^0:=\calO_F^{\times}, 
\quad U_F^i:=1+\frakp_F^i \text{ for \(i \in \bbZ_{>0}\)}.
\]

For a finite finite separable extension \(E/F\) of \(F\), we write \(e(E/F), f(E/F)\) for the ramification index and the residue degree of the extension \(E/F\).

We use similar notation for any finite separable extension \(E\) of \(F\).
We fix a separable closure \(F^{\sep}\) and write \(\Gamma_F:=\Gal(F^{\sep}/F)\) and \(W_F\) for the absolute Galois group and the (absolute) Weil group of \(F\).
Whenever we work with a separable extension of \(F\) we tacitly take it inside \(F^{\sep}\).

\item[Artin reciprocity map]
We identify characters of \(F^{\times}\) and those of \(W_F\) via the Artin reciprocity map \(\Art_F\colon F^{\times} \xrightarrow{\sim} W_F^{\ab}\);
if \(\chi\) is a character of \(F^{\times}\), then we usually simply write \(\chi\) for the character \(\chi \circ \Art_F^{-1}\) of \(W_F\).

In this paper, we usually normalize \(\Art_F\)
so that a uniformizer is sent to an \emph{arithmetic} Frobenius element 
(i.e., an element which acts as the \(q\)-th power map on \(k_F\)),
rather than a geometric Frobenius element
(i.e., an element whose inverse is an arithmetic Frobenius element).
When we need to discuss the normalizations,
we refer to them as the arithmetic and geometric normalizations
and write \({}_{\arith}\Art_F\) and \({}_{\geom}\Art_F\) for the maps normalized accordingly.
Thus, we have \({}_{\geom}\Art_F(x)={}_{\arith}\Art_F(x)^{-1}\) for \(x\in F^{\times}\).

As clearly explained in \cite[\S 4]{1201.5658}, the choice of a normalization of the Artin reciprocity map determines a compatible  normalization of the local Langlands correspondence for connected reductive groups, in particular for \(\GL_n\) and tori.

\item[Inner forms and inner twists]
While this paper mainly deals with \(\GL_n\) and its inner forms, in some parts we discuss general inner forms.
Let us recall some standard notions.

Let \(\bfG, \bfG'\) be connected reductive groups over \(F\).
We say that \(\bfG'\) is an \emph{inner form of} \(\bfG\)
if there exists 
an isomorphism \(\psi \colon \bfG \to \bfG'\) over \(F^{\sep}\)
such that, for any \(\sigma \in \Gamma_F\), the automorphism \(\psi^{-1} \circ \sigma(\psi)\) of \(\bfG\) is inner.
We refer to such an isomorphism \(\psi\) or such a pair \((\bfG', \psi)\) as an \emph{inner twist of} \(\bfG\).
For an inner twist \((\bfG', \psi)\) of \(\bfG\), the map \(\sigma \mapsto \psi^{-1} \circ \sigma(\psi)\) is a 1-cocycle valued in the adjoint quotient \(\bfG_{\ad}\) of \(\bfG\).
Two inner twists of \(\bfG\) are said to be \emph{isomorphic} if they define the same cohomology class in \(H^1(F, \bfG_{\ad})\).
Note that this implies that the underlying inner forms are isomorphic, but the converse does not hold in general.

Let \(A\) be a central simple algebra over \(F\) such that \(\dim_F A=n^2\) and let \(\underline{A}^{\times}\) be the algebraic group such that
\[
\underline{A}^{\times}(R)=(A\otimes_{F} R)^{\times}
\]
for any commutative \(F\)-algebra \(R\).
Then \(\underline{A}^{\times}\) is an inner form of \(\GL_n\)
and every inner form of \(\GL_n\) arises this way.
(Note, however, that we always have \(\underline{A}^{\times} \simeq \underline{A}^{\op, \times}\) while \(A\) is not isomorphic to \(A^{\op}\) in general.)
In this situation, an isomorphism \(M_n(F^{\sep}) \xrightarrow{\sim} A\otimes_{F} F^{\sep}\) of \(F^{\sep}\)-algebras induces
an inner twist \(\GL_n \to \underline{A}^{\times}\) of \(\GL_n\), whose isomorphism class is independent of the choice of an \(F^{\sep}\)-algebra isomorphism by Skolem--Noether theorem.
In other words, an inner form of \(\GL_n\) is canonically endowed with an inner twist (up to isomorphism) once the expression in terms of a central simple algebra is given.
For most parts of the paper (namely, except for the parts where we discuss the relevant general theory), we fix \(A\) and put \(\bfG^{\ast}=\GL_n, \bfG=\underline{A}^{\times}\).
Whenever we need an inner twist we use the canonical one induced by \(A\), sometimes even implicitly.

\item[Brauer group]
We write \(\Br(F):=H^2(F, \Gm)\) for the Brauer group of \(F\).
There is a canonical isomorphism
\(\inv_F \colon \Br(F) \xrightarrow{\sim} \bbQ/\bbZ\).

Let \(A, \bfG^{\ast}=\GL_n, \bfG=\underline{A}^{\times}\) be as above and let \(t\in H^1(F, \bfG^{\ast}_{\ad})\) denote the cohomology class defined by the canonical inner twist \(\bfG^{\ast} \to \bfG\).
Then the class \([A]\in \Br(F)\) of \(A\) is equal to the image of \(t\) 
by the coboundary map \(H^1(F, \bfG^{\ast}_{\ad})\to H^2(F, \Gm)=\Br(F)\) associated to the short exact sequence
\(1\to \Gm \to \bfG^{\ast} \to \bfG^{\ast}_{\ad} \to 1\).

\item[Dual groups and \(L\)-groups]
For a connected reductive group \(\bfG\) over \(F\),
we write \(\widehat{\bfG}\) for the dual group of \(\bfG\) taken over \(\bbC\) and \({}^{L}\bfG:=\widehat{\bfG} \rtimes W_F\) for the \(L\)-group.
See \cite[\S 1]{MR757954} for the precise definition and convention.
\end{description}

\section{Description of the local Jacquet-Langlands correspondence due to Bushnell--Henniart and Tam} \label{sec:review-of-BH-Tam}
In this section we first review Bushnell--Henniart's description of the local Jacquet--Langlands correspondence for essentially tame supercuspidal representations of \(\GL_n(F)\).
Then, after introducing \(\chi\)-data and \(\zeta\)-data of Langlands--Shelstad, we recall Tam's reinterpretation of the above description using \(\zeta\)-data.
Along the way, we also review similar results for the local Langlands correspondence for \(\GL_n\).

\subsection{Bushnell--Henniart's essentially tame local Jacquet--Langlands correspondence} \label{sec:review-of-BH}
Let \(A^{\ast}:=M_n(F)\) be the matrix algebra
and
let \(A\) be a central simple algebra over \(F\) such that \(\dim_F A=n^2\).
Then there exist positive integers \(m, d\) and a central division algebra 
\(D\) over \(F\) of dimension \(d^2\)
such that \(n=md\)
and \(A \simeq M_m(D)\)
(we fix such an isomorphism).
We put \(\bfG^{\ast}:=\GL_n\) and \(\bfG:=\underline{A}^{\times}\).

The local Jacquet--Langlands correspondence yields a natural bijection, which we write as \(\LJLC^{\cl}\), between the sets of isomorphism classes of irreducible discrete series representations of \(\bfG^{\ast}(F)=\GL_n(F)\)
and of \(\bfG(F)=\GL_m(D)\):
\[
\LJLC^{\cl} \colon 
\{\text{discrete series rep's of \(\GL_n(F)\)}\}/{\sim} 
\stackrel{1:1}{\longrightarrow}
\{\text{discrete series rep's of \(\GL_m(D)\)}\}/{\sim} 
\]
characterized by a character relation.
It is 
due to Deligne--Kazhdan--Vign\'{e}ras \cite{MR771672} and Rogawski \cite{MR700135} in characteristic zero and Badulescu \cite{MR1951441} in positive characteristic.

In \cite{MR2848585}, Bushnell--Henniart made the bijection \(\LJLC^{\cl}\) explicit for certain sets of supercuspidal representations denoted by \(\calA _n^{\et}(F)\) and \(\calA _m^{\et}(D)\).
This class of representations is defined in terms of two invariants of discrete series representations \(\pi\) of \(\GL_m(D)\).

One is the torsion number \(t(\pi)\), the number of unramified characters \(\chi\) of \(F^{\times}\) such that the twisted representation \(\chi \pi\) is isomorphic to \(\pi\).
As is easily seen, \(t(\pi)\) is a divisor of \(n\).

The other is the parametric degree \(\delta (\pi)\), a positive integer defined in \cite{MR2848585}.
It follows from the definition of \(\delta(\pi)\) (which we do not recall here) that \(t(\pi)\) divides \(\delta(\pi)\) and \(\delta(\pi)\) divides \(n\).
A discrete series representation \(\pi\) is said to be \emph{essentially tame} if \(p\) does not divide \(\frac{\delta (\pi)}{t(\pi)}\).
Moreover, if \(\delta (\pi)=n\), then \(\pi\) is supercuspidal. The converse is also true for \(\GL_n(F)\), but not in general.

Let \(\calA _m^{\et}(D)\) (resp.\ \(\calA _n^{\et}(F)\)) denote the set of isomorphism classes of discrete series representations \(\pi\) of \(\GL_m(D)\) (resp.\ \(\GL_n(F)\)) for which \(\delta (\pi)=n\) and \(\pi\) is essentially tame.
By the condition \(\delta (\pi)=n\) any such \(\pi\) is supercuspidal.
An important property of these two invariants is that they are preserved under \(\LJLC^{\cl}\).
In particular, we see that a discrete series representation \(\pi\) of \(\GL_m(D)\) corresponds to a supercuspidal representation of \(\GL_n(F)\) by \(\LJLC^{\cl}\) if and only if \(\delta (\pi)=n\), 
and that \(\LJLC^{\cl}\) induces a bijection between \(\calA _m^{\et}(D)\) and \(\calA _n^{\et}(F)\).

\[
\begin{tikzcd}
\{\text{discrete series rep's of \(\GL_n(F)\)}\}/{\sim} 
\arrow[r, "\LJLC^{\cl}", "1:1"']
& \{\text{discrete series rep's of \(\GL_m(D)\)}\}/{\sim} \\
\{\text{supercuspidal rep's of \(\GL_n(F)\)}\}/{\sim} \arrow[u,hookrightarrow] \arrow[r]
& \{\text{supercuspidal rep's of \(\GL_m(D)\)}\}/{\sim} \arrow[u, hookrightarrow] \\
\calA_n^{\et}(F) \arrow[u, hookrightarrow] \arrow[r, "1:1"] & \calA_m^{\et}(D) \arrow[u, hookrightarrow] 
\end{tikzcd}
\]

Bushnell--Henniart's explicit description of \(\LJLC^{\cl}\) involves a parametrization of elements of \(\calA_n^{\et}(F)\) and \(\calA_m^{\et}(D)\) by \emph{admissible pairs}.
An admissible pair (of degree \(n\)) is a pair \((E/F, \xi)\), where
\begin{itemize}
\item \(E/F\) is a tamely ramified extension of degree \(n\) and
\item \(\xi\) is a smooth character of \(E^{\times}\) satisfying some regularity condition.
\end{itemize}
Two admissible pairs \((E/F, \xi)\), \((E'/F, \xi')\) are said to be \(F\)-isomorphic if there is an \(F\)-isomorphism \(i\colon E \xrightarrow{\sim} E'\) 
such that \(\xi=\xi' \circ i\).
Bushnell--Henniart established a bijection between the set of \(F\)-isomorphism classes of admissible pairs and \(\calA_m^{\et}(D)\).
It is an extension of a similar bijection for \(\calA_n^{\et}(F)\) in \cite{MR2138141}, which in turn is based on the pioneering work \cite{MR0492087} of Howe (where the term admissible character is used instead).
We write the supercuspidal representation of \(\GL_m(D)\) (resp.\ \(\GL_n(F)\)) parametrized by an admissible pair \((E/F, \xi)\) of degree \(n\) as \({}_{\bfG}\pi^{\BH}_{(E/F, \xi)}\) (resp.\ \({}_{\bfG^{\ast}}\pi^{\BH}_{(E/F, \xi)}\)).
\[
\{\text{admissible pairs of degree \(n\)}\}/\text{\(F\)-isom.} \xrightarrow{1:1} \calA_m^{\et}(D); \quad (E/F, \xi) \mapsto {}_{\bfG}\pi^{\BH}_{(E/F, \xi)}.
\]
We will briefly review the construction of \({}_{\bfG}\pi^{\BH}_{(E/F, \xi)}\) in the beginning of Appendix \ref{sec:app}.

One of the main results of \cite{MR2848585} describes \(\LJLC^{\cl}\) for \(\calA_n^{\et}(F)\) (i.e., essentially tame supercuspidal representations of \(\GL_n(F)\))
in terms of admissible pairs and a character twist:
\begin{thm}[\cite{MR2848585}] \label{thm:etLJLC}
Let \((E/F, \xi)\) be an admissible pair.
Then there exists an explicit tamely ramified character \(\nu_{\rec}\)
of \(E^{\times}\) (which depends on \(\xi\), but only through \(\xi|_{U_E^1}\)) such that 
\[
\LJLC^{\cl} ({}_{\bfG^{\ast}}\pi^{\BH}_{(E/F, \nu_{\rec} \cdot \xi)})={}_{\bfG}\pi^{\BH}_{(E/F, \xi)},
\]
or equivalently,
\[
\LJLC^{\cl} ({}_{\bfG^{\ast}}\pi^{\BH}_{(E/F, \xi)})={}_{\bfG}\pi^{\BH}_{(E/F, \nu_{\rec}^{-1} \cdot \xi)}.
\]
\end{thm}
Here, the superscript ``\(\cl\)" stands for ``classical" 
and the subscript ``\(\rec\)" stands for the term ``rectifier" in \cite{MR2679700} (and also \cite{MR3505131}).

\begin{rem} \label{rem:char-etLJLC}
As emphasized in the second paragraph of page 470 of \cite{MR2848585}, this result is valid regardless of the characteristic \(\ch F\).
\end{rem}

Let us also consider the local Langlands correspondence for \(\GL_n\).
It is a bijection \(\LLC_{\GL_n}^{\cl}\) from the set of isomorphism classes of irreducible smooth representations of \(\GL_n(F)\) to the set of \(\GL_n(\bbC)\)-conjugacy classes of \(L\)-parameters of \(\GL_n\) (or equivalently, isomorphism classes of \(n\)-dimensional Frobenius-semisimple Weil--Deligne representations),
which is due to Laumon--Rapoport--Stuhler \cite{MR1228127} in positive characteristic
and Harris--Taylor \cite{MR1876802} in characteristic zero.
In their earlier papers \cite{MR2138141}, \cite{MR2148193}, \cite{MR2679700}, 
Bushnell--Henniart also made \(\LLC_{\GL_n}^{\cl}\) explicit for \(\calA_n^{\et}(F)\):
\begin{thm}[\cite{MR2138141}, \cite{MR2148193}, \cite{MR2679700}] \label{thm:etLLC}
Let \((E/F, \xi)\) be an admissible pair.
Then there exists an explicit tamely ramified character \(\mu_{\rec}\)
of \(E^{\times}\) (which depends on \(\xi\), but only through \(\xi|_{U_E^1}\)) such that 
\[
\LLC_{\GL_n}^{\cl}({}_{\bfG^{\ast}}\pi^{\BH}_{(E/F, \mu_{\rec} \cdot \xi)})=\Ind_{W_E}^{W_F} \xi,
\]
or equivalently,
\[
\LLC_{\GL_n}^{\cl}({}_{\bfG^{\ast}}\pi^{\BH}_{(E/F, \xi)})=\Ind_{W_E}^{W_F} (\mu_{\rec}^{-1} \cdot \xi).
\]
\end{thm}

\begin{rem} \label{rem:normalization-of-BH} 
Similarly to the Artin reciprocity map (Section \ref{sec:notation}),
we write \({}_{\arith}\LLC_{\GL_n}^{\cl}\) (resp.\ \({}_{\geom}\LLC_{\GL_n}^{\cl}\)) for the local Langlands correspondence for \(\GL_n\) according to the arithmetic (resp.\ geometric) normalization.
Then, as explained in \cite[\S 4.4]{1201.5658}, for any irreducible smooth representation \(\pi\) of \(\GL_n(F)\), 
the \(n\)-dimensional Weil--Deligne representation \({}_{\arith}\LLC_{\GL_n}^{\cl}(\pi)\)
is isomorphic to the contragredient \(({}_{\geom}\LLC_{\GL_n}^{\cl}(\pi))^{\vee}\) of its geometrically normalized counterpart.

As we usually use the arithmetic normalization in this paper, if we wish to be explicit about the normalization, the (second) statement of Theorem \ref{thm:etLLC} reads:
\[
{}_{\arith}\LLC_{\GL_n}^{\cl}({}_{\bfG^{\ast}}\pi^{\BH}_{(E/F, \xi)})
=\Ind_{W_E}^{W_F} \big((\mu_{\rec}^{-1} \cdot \xi) \circ {}_{\arith}\Art_E^{-1}\big).
\]
Since Bushnell--Henniart adopt the \emph{geometric} normalization,
the statement immediately found in \cite{MR2138141}, \cite{MR2148193}, \cite{MR2679700}
is the above equality with the two symbols ``\(\arith\)'' replaced by ``\(\geom\)''.
However, it is not difficult to deduce Theorem \ref{thm:etLLC} from the latter equality as follows:
\begin{align*}
{}_{\arith}\LLC_{\GL_n}^{\cl}({}_{\bfG^{\ast}}\pi^{\BH}_{(E/F, \xi)})
&=\big({}_{\geom}\LLC_{\GL_n}^{\cl}({}_{\bfG^{\ast}}\pi^{\BH}_{(E/F, \xi)})\big)^{\vee} \\
&=\big( \Ind_{W_E}^{W_F} \big((\mu_{\rec}^{-1} \cdot \xi) \circ {}_{\geom}\Art_E^{-1}\big)\big)^{\vee} \\
&=\Ind_{W_E}^{W_F} \big((\mu_{\rec}^{-1} \cdot \xi) \circ {}_{\geom}\Art_E^{-1}\big)^{-1} \\
&=\Ind_{W_E}^{W_F} \big((\mu_{\rec}^{-1} \cdot \xi) \circ {}_{\arith}\Art_E^{-1}\big).
\end{align*}
In other words, the character \(\mu_{\rec}\) exactly the same as the one in \cite{MR2138141}, \cite{MR2148193}, \cite{MR2679700} works for the arithmetically normalized statement.
\end{rem}
\begin{rem} \label{rem:char-etLLC}
Again Theorem \ref{thm:etLLC} is valid regardless of the characteristic \(\ch F\).
Indeed, according to \cite[page 8, Note on characteristic]{MR3236840},
although \(\ch F\) was originally assumed to be zero in \cite{MR2138141}, \cite{MR2148193}, \cite{MR2679700}, now the result applies in positive characteristic as well thanks to \cite{MR2673422}.
\end{rem}

\subsection{Langlands--Shelstad's \(\chi\)-data and the associated \(L\)-embedding} \label{sec:chi-data-L-emb}
We recall the notion of \(\chi\)-data and the associated \(L\)-embedding introduced by Langlands--Shelstad \cite{MR909227}.
In this section let \(\bfG\) be a general connected reductive group over \(F\)
and \(\bfS\) be a maximal torus of \(\bfG\) defined over \(F\).

The Galois group \(\Gamma_F\) acts on the set \(\Phi(\bfS, \bfG)\) of (absolute) roots of \(\bfG\).
Given a root \(\alpha \in \Phi(\bfS, \bfG)\), we write \(\Gamma_{\alpha}\) and \(\Gamma_{\pm \alpha}\) for the stabilizers of \(\{\alpha\}\) and \(\{\pm \alpha\}\) respectively.
Note that we have \(\Gamma_{\alpha} \subseteq \Gamma_{\pm \alpha}\) and the index is either \(1\) or \(2\).
We also write \(F_{\alpha}\) and \(F_{\pm \alpha}\) for the fixed field of \(\Gamma_{\alpha}\) and \(\Gamma_{\pm \alpha}\) in \(F^{\sep}\) respectively.
Following Langlands--Shelstad \cite{MR909227} and Adler--Spice \cite{MR2543925}, we say that
\begin{itemize}
	\item a root \(\alpha \in \Phi(\bfS, \bfG)\) is \emph{asymmetric} (resp.\ \emph{symmetric}) if \(\Gamma_{\alpha}=\Gamma_{\pm \alpha}\) (resp. \(\Gamma_{\alpha} \subsetneq \Gamma_{\pm \alpha}\)) and
	\item a symmetric root \(\alpha \in \Phi(\bfS, \bfG)\) is \emph{unramified} (resp.\ \emph{ramified}) if the quadratic extension \(F_{\alpha}/F_{\pm \alpha}\) is unramified (resp.\ ramified).
\end{itemize}
We write \(\Phi(\bfS, \bfG)^{\sym}\), \(\Phi(\bfS, \bfG)_{\sym}\), \(\Phi(\bfS, \bfG)_{\sym, \ur}\) and \(\Phi(\bfS, \bfG)_{\sym, \ram}\) for the set of asymmetric roots, symmetric roots, symmetric unramified roots and symmetric ramified roots, respectively.
\begin{defn}[\cite{MR909227}]
	\emph{A set} \(\chi=\{\chi_{\alpha}\}_{\alpha \in \Phi(\bfS, \bfG)}\) \emph{of \(\chi\)-data} (for \(\Phi(\bfS, \bfG)\)) consists of characters \(\chi_{\alpha} \colon F_{\alpha}^{\times} \to \bbC ^{\times}\)
	satisfying the following properties:
	\begin{enumerate}
		\item for any \(\alpha \in \Phi(\bfS, \bfG)\) and \(\sigma \in \Gamma_F\), we have
		\[
		\chi_{-\alpha}=\chi_{\alpha}^{-1}, \quad \chi_{\sigma \cdot \alpha}=\chi_{\alpha}\circ \sigma^{-1}
		\]
		and
		\item for a symmetric root \(\alpha \in \Phi(\bfS, \bfG)\), the restriction \(\chi_{\alpha}|_{F_{\pm \alpha}^{\times}}\) agrees with the quadratic character associated to the quadratic extension \(F_{\alpha}/F_{\pm \alpha}\).
	\end{enumerate}
\end{defn}
Given a set \(\chi\) of \(\chi\)-data, we can construct an embedding of \({}^{L}\bfS\) into \({}^{L}\bfG\) (up to \(\widehat{\bfG}\)-conjugacy) in the following way.
Let \(j \colon \bfS \hookrightarrow \bfG\) denote the inclusion map.
We fix a Borel subgroup \(\bfB \subset \bfG\) containing \(\bfS\) (possibly not defined over \(F\))
and a \(\Gamma_F\)-fixed splitting \((\calB, \calT, \{\calX\})\) of \(\widehat{\bfG}\) for now.
We then have an isomorphism \(\calT \xrightarrow{\sim} \widehat{\bfS}\) associated to the pairs \((\bfB, \bfS)\) and \((\calB, \calT)\), and hence also an embedding \(\widehat{j} \colon \widehat{\bfS} \hookrightarrow \widehat{\bfG}\) by composing the inverse of the former isomorphism with the inclusion \(\calT \hookrightarrow \widehat{\bfG}\).
By further using \(\{\calX\}\) and \(\chi\), we can construct an embedding \({}^{L}j_{\chi} \colon {}^{L}\bfS \hookrightarrow {}^{L}\bfG\) extending \(\widehat{j}\) and preserving the projections to \(W_F\)
(see	\cite[(2.6)]{MR909227}).
The \(\widehat{\bfG}\)-conjugacy class of \({}^{L}j_{\chi}\) depends on the choice of \(\chi\)-data, but not on the choice of \(\bfB\) and \((\calB, \calT, \{\calX\})\) (see \cite[(2.6.1) and (2.6.2)]{MR909227}).
\[
\begin{tikzcd}
\widehat{\bfS} \arrow[rr,"\widehat{j}", hookrightarrow] \arrow[d, hookrightarrow]& & \widehat{\bfG} \arrow[d, hookrightarrow] \\
{}^{L}\bfS \arrow[rr, "{}^{L}j_{\chi}", hookrightarrow] \arrow[dr, twoheadrightarrow] & & {}^{L}\bfG \arrow[dl, twoheadrightarrow] \\
& W_F &
\end{tikzcd}
\]
\begin{rem} \label{rem:normalization-for-emb}
We do not recall the construction of \({}^{L}j_{\chi}\), 
but only remark that in the construction each \(\chi_{\alpha}\) in the \(\chi\)-data is identified with a character of \(W_{F_{\alpha}}\) via the Artin reciprocity map \({}_{\arith}\Art_{F_{\alpha}}\) (see Section \ref{sec:notation} for \({}_{\arith}\Art_{F_{\alpha}}\)).
To stress the arithmetic normalization used here, we sometimes write \({}_{\arith}^{L}j_{\chi}\) for \({}^{L}j_{\chi}\) in this paper.
As is easily checked, if \(\chi=\{\chi_{\alpha}\}\) is a set of \(\chi\)-data, 
then so is \(\chi^{-1}=\{\chi_{\alpha}^{-1}\}\).
Therefore, we can naturally construct an \(L\)-embedding \({}_{\geom}^{L}j_{\chi}\) with the geometric normalization, namely,
\[
{}_{\geom}^{L}j_{\chi}:={}_{\arith}^{L}j_{\chi^{-1}}.
\]
\end{rem}
\begin{rem} \label{rem:char-LS}
Throughout \cite{MR909227}, \(\ch F\) is assumed to be zero.
However, as explained in \cite[\S 6.2.2]{2008.04472}, the construction of an \(L\)-embedding from a set of \(\chi\)-data works without any problems even in positive characteristic.
\end{rem}

Later we also need a similar notion called \(\zeta\)-data:
\begin{defn}[\cite{MR909227}]
\emph{A set} \(\zeta=\{\zeta_{\alpha}\}_{\alpha \in \Phi(\bfS, \bfG)}\) \emph{of \(\zeta\)-data} (for \(\Phi(\bfS, \bfG)\)) consists of characters \(\zeta_{\alpha} \colon F_{\alpha}^{\times} \to \bbC ^{\times}\)
satisfying the following properties:
\begin{enumerate}
\item for any \(\alpha \in \Phi(\bfS, \bfG)\) and \(\sigma \in \Gamma_F\), we have
\[
\zeta_{-\alpha}=\zeta_{\alpha}^{-1}, \quad \zeta_{\sigma \cdot \alpha}=\zeta_{\alpha}\circ \sigma^{-1}
\]
and
\item[(\(2'\))] for a symmetric root \(\alpha \in \Phi(\bfS, \bfG)\), the restriction \(\zeta_{\alpha}|_{F_{\pm \alpha}^{\times}}\) is trivial.
\end{enumerate}
\end{defn}

Before proceeding, let us summarize some standard facts.
\begin{rem} \label{rem:GL_n-and-ell-tori}
\begin{enumerate}
\item \label{item:roots-for-inner-form-of-GL_n}
Let \(\bfG\) and \(\bfS\) be as above.
Moreover, let \(\psi \colon \bfG \to \bfG'\) be an inner twist of \(\bfG\)
and suppose that the restriction \(\psi|_{\bfS} \colon \bfS \to \bfG'\) is defined over \(F\).
Then the image of \(\bfS\) under \(\psi\) is also a maximal torus in \(\bfG'\), which we continue to write simply as \(\bfS\).
In this situation, we have an equality
\[
\Phi(\bfS, \bfG) =\Phi(\bfS, \bfG')
\]
of subsets of the (absolute) character group \(X^{\ast}(\bfS)\) of \(\bfS\).

Now let \(\bfG^{\ast}=\GL_n\) and \(\bfG=\underline{A}^{\times}\) be as in Section \ref{sec:review-of-BH}
and suppose that a maximal torus \(\bfS \subset \bfG^{\ast}\) is elliptic.
Recall that we have an inner twist \(\psi \colon \bfG^{\ast} \to \bfG\), canonically up to isomorphism (see Section \ref{sec:notation}).
By \cite[\S 10]{MR858284} and \cite[Lemma 3.2.1]{MR4013740}, there exists \(g \in \bfG^{\ast}(F^{\sep})\) such that \(\psi \circ \Ad (g)|_{\bfS} \colon \bfS \to \bfG\)  is defined over \(F\).
Thus we are in the above situation and have
\(\Phi(\bfS, \bfG^{\ast})=\Phi(\bfS, \bfG)\).

\item \label{item:field-extension}
Let \(\bfG^{\ast}=\GL_n=\underline{A}^{\ast, \times}\) and \(\bfG=\underline{A}^{\times}\) be as in the previous paragraph.

It is well-known and easily seen that there are natural bijections between the following sets:
\begin{enumerate}
\item the set of \(F\)-isomorphism classes of separable extension \(E/F\) of degree \(n\),
\item[(b1)] the set of \(\bfG^{\ast}(F)\)-conjugacy classes of elliptic maximal tori \(\bfS \subset \bfG^{\ast}\),
\item[(c1)] the set of equivalence classes of \(F\)-embeddings \(j\colon \bfS \hookrightarrow \bfG^{\ast}\) of tori \(\bfS\) such that \(j(\bfS) \subset \bfG^{\ast}\) is an elliptic maximal torus,
where \(j_1 \colon \bfS_1 \hookrightarrow \bfG^{\ast}\) and 
\(j_2 \colon \bfS_2 \hookrightarrow \bfG^{\ast}\) are equivalent if
there exists an \(F\)-isomorphism \(\iota \colon \bfS_1 \stackrel{\sim}{\to} \bfS_2\) such that \(j_1\) and \(j_2 \circ \iota\) are \(\bfG^{\ast}(F)\)-conjugate,
\item[(b2)] the set of \(\bfG(F)\)-conjugacy classes of elliptic maximal tori \(\bfS \subset \bfG\), and
\item[(c2)] the set of equivalence classes of \(F\)-embeddings \(j\colon \bfS \hookrightarrow \bfG\) of tori \(\bfS\) such that \(j(\bfS) \subset \bfG\) is an elliptic maximal torus, where the equivalence relation is similar to the one in (c1).
\end{enumerate}
We only sketch the bijections.
The bijection from (c1) to (b1) is induced by taking the image of \(j\) and 
that from (b1) to (a) is induced by taking an extension \(E/F\) such that \(\bfS\) is isomorphic to \(\Res_{E/F}\Gm\),
while that from (a) to (c1) is induced by taking an \(F\)-algebra embedding \(E\hookrightarrow A^{\ast}\) to obtain \(\bfS:=\Res_{E/F} \Gm \hookrightarrow \underline{A}^{\ast, \times}=\bfG^{\ast}\).
The bijection between (b1) and (b2) (and that between (c1) and (c2)) can be induced in a similar way to the previous remark.

We frequently use these bijections throughout this paper.
\end{enumerate}
\end{rem}

\subsection{\(L\)-embedding and induction} \label{sec:L-emb-and-ind}
Let \(\bfG^{\ast}=\GL_n\) and \(\bfG=\underline{A}^{\times}\) be as in Section \ref{sec:review-of-BH}.
Moreover, let \(\bfS \subset \bfG^{\ast}\) be an elliptic maximal torus.
Then \(\bfS \simeq \Res_{E/F}\Gm\) for some separable extension \(E/F\) of degree \(n\).

First we recall an explicit description of the set \(\Phi(\bfS, \bfG^{\ast})=\Phi(\bfS, \bfG)\)  (cf.\ Remark \ref{rem:GL_n-and-ell-tori} \eqref{item:roots-for-inner-form-of-GL_n}) following \cite[\S 3.1]{MR3509939} (also reviewed in \cite[\S 2.1]{MR3505131}).
The canonical isomorphism
\[
E\otimes_{F} F^{\sep} \stackrel{\sim}{\to} (F^{\sep})^{\Gamma_F/\Gamma_E};
\quad 
x\otimes a\mapsto (\sigma(x)a)_{\sigma \Gamma_E}
\]
induces 
an isomorphism
\(\bfS_{F^{\sep}}\simeq \Gm^{\Gamma_F/\Gamma_E}\).
For any \(g\in \Gamma_F\) we write \(\delta_{g}\) for the composite
\[
\delta_{g} \colon \bfS_{F^{\sep}}\simeq \Gm^{\Gamma_F/\Gamma_E}\to \Gm
\]
with the projection to the \(g\Gamma_E\)-component.
Then for any set \(\{g_1, \dots, g_n\}\) of representatives of \(\Gamma_F/\Gamma_E\)
the characters \(\delta_{g_1}, \dots, \delta_{g_n}\)
form a \(\bbZ\)-basis of the (absolute) character group \(X^{\ast}(\bfS)\).

Put \(\begin{bmatrix}
g_1 \\ g_2
\end{bmatrix}
=\delta_{g_1}-\delta_{g_2} \in X^{\ast}(\bfS)\)
for \(g_1, g_2 \in \Gamma_F\).
Then the set \(\Phi(\bfS, \bfG^{\ast})\) 
is described as
\[
\Phi(\bfS, \bfG^{\ast})=\Big\{
\begin{bmatrix}
g_1 \\ g_2
\end{bmatrix}
\mid 
g_1\Gamma_E \neq g_2\Gamma_E
\Big\}.
\]
Note that \(\Phi(\bfS, \bfG^{\ast})\) is stable under the action of \(\Gamma_F\)
and the action is given by
\[
g\cdot \begin{bmatrix}
g_1 \\ g_2
\end{bmatrix}
=\begin{bmatrix}
gg_1 \\ gg_2
\end{bmatrix}
\]
for \(g\in \Gamma_F\) and 
\(\begin{bmatrix}
g_1 \\
g_2
\end{bmatrix} 
\in \Phi(\bfS, \bfG^{\ast})\).
Moreover, we have the following bijection:
\begin{equation} \label{eq:Galois-orb-as-coset}
(\Gamma_E\backslash \Gamma_F/\Gamma_E)' \xrightarrow{1:1}\Gamma_F\backslash \Phi(\bfS, \bfG^{\ast});
\quad 
\Gamma_Eg\Gamma_E \mapsto 
\Gamma_F
\begin{bmatrix}
1 \\
g
\end{bmatrix},
\end{equation}
where \((\Gamma_E\backslash \Gamma_F/\Gamma_E)':=\Gamma_E\backslash \Gamma_F/\Gamma_E-\{\Gamma_E\}\) denotes the set of non-trivial double cosets.

Now let \(\xi\) be a character of \(\bfS(F)\)
and \(\chi=\{\chi_{\alpha}\}_{\alpha \in \Phi(\bfS, \bfG^{\ast})}\) be a set of \(\chi\)-data.
The local Langlands correspondence for the torus \(\bfS\) associates to \(\xi\) an \(L\)-parameter \(\phi_{\xi} \colon W_F \to {}^{L}\bfS\) of \(\bfS\).
On the other hand, the \(\chi\)-data \(\chi\) yields an \(L\)-embedding 
\({}^{L}j_{\chi}\colon {}^{L}\bfS \hookrightarrow {}^{L}\bfG^{\ast}\) as in Section \ref{sec:chi-data-L-emb}.
Thus, by composition we have an \(L\)-parameter of \(\bfG^{\ast}\): 
\[
{}^{L}j_{\chi} \circ \phi_{\xi} \colon W_F \to {}^{L}\bfG^{\ast}.
\]
Note that by post-composing the projection \(\pr_{\widehat{\bfG^{\ast}}} \colon {}^{L}\bfG^{\ast} \to \widehat{\bfG^{\ast}}\) we may equally regard this \(L\)-parameter
as an \(n\)-dimensional representation of \(W_F\).
The following proposition of Tam gives a useful expression of this representation.

\begin{prop}[{\cite[Proposition 6.5]{MR3509939}}] \label{prop:Tam}
Let \(\xi\) and \(\chi\) be as above.
Put 
\[
\mu_{\chi}:=\prod_{[\alpha]\in \Gamma_F\backslash \Phi(\bfS, \bfG^{\ast})} \chi_{\alpha}|_{E^{\times}},
\]
where, for each orbit \([\alpha]\in \Gamma_F\backslash \Phi(\bfS, \bfG^{\ast})\),
we take a representative \(\alpha \in \Phi(\bfS, \bfG^{\ast})\) of the form \(\alpha=
\begin{bmatrix}
1 \\
g
\end{bmatrix}\)
for some \(g\in \Gamma_F\).

Then the \(L\)-parameter \({}^{L}j_{\chi} \circ \phi_{\xi}\), regarded as an \(n\)-dimensional representation of \(W_F\), is isomorphic to \(\Ind_{W_E}^{W_F} (\xi \cdot \mu_{\chi})\).
\end{prop}
Note that 
the fixed field \(F_{\alpha}\) 
for a root \(\alpha=\begin{bmatrix}
g_1 \\
g_2
\end{bmatrix}
\in \Phi(\bfS, \bfG^{\ast})\) is nothing but the composite field \(g_1(E)\cdot g_2(E)\),
and thus \(F_{\alpha}^{\times}\) contains \(E^{\times}\) for 
\(\alpha=\begin{bmatrix}
1 \\
g
\end{bmatrix}\)
as in the proposition 
and the restriction \(\chi_{\alpha}|_{E^{\times}}\) makes sense.
Also, the restriction \(\chi_{\alpha}|_{E^{\times}}\) only depends on the orbit \([\alpha]\) and not on the choice of \(\alpha=\begin{bmatrix}
1 \\
g
\end{bmatrix}\).
Indeed, in view of \eqref{eq:Galois-orb-as-coset}, another choice \(\beta\) is
of the form \(\beta=\begin{bmatrix}
1 \\
h_1gh_2
\end{bmatrix}\)
for some \(h_1, h_2 \in \Gamma_E\).
Then we have 
\[
\beta=\begin{bmatrix}
1 \\
h_1g
\end{bmatrix}
=h_1 \cdot
\begin{bmatrix}
h_1^{-1} \\
g
\end{bmatrix}
=h_1 \cdot
\begin{bmatrix}
1 \\
g
\end{bmatrix}
=h_1 \cdot \alpha,
\]
which implies 
\[
\chi_{\beta}(x)
=\chi_{h_1 \cdot \alpha}(x)
=\chi_{\alpha}(h_1^{-1}(x))
=\chi_{\alpha}(x)
\]
for any \(x\in E^{\times}\).

\begin{rem} \label{rem:normalization-of-Tam-prop}
Let us discuss the arithmetic and geometric normalizations.
The normalization is relevant to Proposition \ref{prop:Tam} through three objects:
\begin{itemize}
\item the \(L\)-embedding \({}^{L}j_{\chi}\) (Remark \ref{rem:normalization-for-emb}), 
\item the \(L\)-parameter \(\phi_{\xi}\) of \(\bfS\), and
\item the character \(\xi \cdot \mu_{\chi}\) of \(W_E\) (Section \ref{sec:notation}).
\end{itemize}
We write \({}_{\arith}\phi_{\xi}\) (resp.\ \({}_{\geom}\phi_{\xi}\)) for the \(L\)-parameter of \(\bfS\) associated to the character \(\xi\) of \(\bfS(F)\) by the local Langlands correspondence for \(\bfS\) according to the arithmetic (resp.\ geometric) normalization.
We have \({}_{\arith}\phi_{\xi}={}_{\geom}\phi_{\xi ^{-1}}\).
Thus, in precise but cumbersome notation, the statement of Proposition \ref{prop:Tam} is expressed as follows:
\[
\pr_{\widehat{\bfG^{\ast}}} \circ {}_{\arith}^{L}j_{\chi} \circ {}_{\arith}\phi_{\xi} \simeq \Ind_{W_E}^{W_F} \big((\xi \cdot \mu_{\chi}) \circ {}_{\arith}\Art_E^{-1}\big).
\]
In fact, \cite[Proposition 6.5]{MR3509939} adopts the \emph{geometric} normalization. In other words, it asserts
\[
\pr_{\widehat{\bfG^{\ast}}} \circ {}_{\geom}^{L}j_{\chi} \circ {}_{\geom}\phi_{\xi} \simeq \Ind_{W_E}^{W_F} \big((\xi \cdot \mu_{\chi}) \circ {}_{\geom}\Art_E^{-1}\big),
\]
which is not exactly the same as Proposition \ref{prop:Tam}.
However, similarly to Remark \ref{rem:normalization-of-BH}, it is fairly easy to deduce one from the other as follows:
\begin{align*}
\pr_{\widehat{\bfG^{\ast}}} \circ {}_{\arith}^{L}j_{\chi} \circ {}_{\arith}\phi_{\xi} 
& \simeq \pr_{\widehat{\bfG^{\ast}}} \circ {}_{\geom}^{L}j_{\chi^{-1}} \circ {}_{\geom}\phi_{\xi^{-1}} \\
& \simeq \Ind_{W_E}^{W_F} \big((\xi^{-1} \cdot \mu_{\chi^{-1}}) \circ {}_{\geom}\Art_E^{-1}\big) \\
& \simeq \Ind_{W_E}^{W_F} \big((\xi^{-1} \cdot \mu_{\chi}^{-1}) \circ {}_{\geom}\Art_E^{-1}\big) \\
& \simeq \Ind_{W_E}^{W_F} \big((\xi \cdot \mu_{\chi}) \circ {}_{\arith}\Art_E^{-1}\big).
\end{align*}
Here we used the equality \(\mu_{\chi^{-1}}=\mu_{\chi}^{-1}\), which is immediate from the definition of \(\mu_{\chi}\).
\end{rem}
Motivated by Proposition \ref{prop:Tam}, Tam defined a variant of \(\mu_{\chi}\) for \(\zeta\)-data:
\begin{defn}[{\cite[page 379]{MR3505131}}] \label{def:nu_zeta}
Let \(\zeta=\{\zeta_{\alpha}\}_{\alpha \in \Phi(\bfS, \bfG^{\ast})}\)
be a set of \(\zeta\)-data.
Put 
\[
\nu_{\zeta}:=\prod_{[\alpha]\in \Gamma_F\backslash \Phi(\bfS, \bfG^{\ast})} \zeta_{\alpha}|_{E^{\times}},
\]
where, for each orbit \([\alpha]\in \Gamma_F\backslash \Phi(\bfS, \bfG^{\ast})\),
we take a representative \(\alpha \in \Phi(\bfS, \bfG^{\ast})\) as in Proposition \ref{prop:Tam}.
\end{defn}

\subsection{Tam's reinterpretation of rectifiers} \label{sec:Tam-thm}
Let \(\bfG^{\ast}, \bfG, \bfS \simeq \Res_{E/F}\Gm\) be as in the previous section and let \((E/F, \xi)\) be an admissible pair.
Then by Theorems \ref{thm:etLLC} and \ref{thm:etLJLC} we have characters \(\mu_{\rec}, \nu_{\rec}	\) of \(E^{\times}\).

In \cite{MR3509939} (resp.\ \cite{MR3505131}), Tam constructed an explicit 
set of \(\chi\)-data (resp.\ \(\zeta\)-data), which we write as
\[
\chi_{\Tam}=\{\chi_{\Tam, \alpha}\}_{\alpha \in \Phi(\bfS, \bfG^{\ast})}, \quad \zeta_{\Tam}=\{\zeta_{\Tam, \alpha}\}_{\alpha \in \Phi(\bfS, \bfG^{\ast})}
\]
respectively.
\begin{thm}[{\cite[Theorem 7.1 (ii)]{MR3509939}, \cite[Theorem 5.5 (ii)]{MR3505131}}] \label{thm:Tam}
Let the notation be as above.
Then we have
\[
\mu_{\rec}=\mu_{\chi_{\Tam}},
\quad
\nu_{\rec}=\nu_{\zeta_{\Tam}}.
\]
\end{thm}
\begin{rem} \label{rem:char-Tam}
This theorem as well as Proposition \ref{prop:Tam} holds without any restriction on \(\ch F\).
Indeed, according to \cite[Remark 1.3]{MR3509939} and \cite[Remark 1.3]{MR3505131}, the assumption \(\ch F=0\) is made only because the papers \cite{MR909227}, \cite{MR1687096} are referred,
but the relevant part of the theory is available equally in positive characteristic (see Remark \ref{rem:char-LS}).
\end{rem}

Combining Theorem \ref{thm:Tam} with Theorems \ref{thm:etLLC} and \ref{thm:etLJLC}, we gain the following:
\begin{cor} \label{cor:Tam}
Let the notation be as above.
Then we have 
\[
\LLC_{\bfG^{\ast}}^{\cl}({}_{\bfG^{\ast}}\pi_{(E/F, \xi)}^{\BH})
=\Ind_{W_E}^{W_F} (\mu_{\chi_{\Tam}}^{-1} \cdot \xi),
\quad
\LJLC^{\cl}({}_{\bfG^{\ast}}\pi_{(E/F, \xi)}^{\BH})={}_{\bfG}\pi_{(E/F, \nu_{\zeta_{\Tam}}^{-1} \cdot \xi)}^{\BH}.
\]
\end{cor}

\begin{rem} \label{rem:exp-for-cl}
By Remarks \ref{rem:char-Tam}, \ref{rem:normalization-of-BH}, \ref{rem:char-etLLC}, the expression for \(\LLC_{\bfG^{\ast}}^{\cl}({}_{\bfG^{\ast}}\pi_{(E/F, \xi)}^{\BH})\) in Theorem \ref{cor:Tam} is valid with the arithmetic normalization and regardless of \(\ch F\).
Moreover, by Remarks \ref{rem:char-Tam}, \ref{rem:char-etLJLC}, the expression for \(\LJLC^{\cl}({}_{\bfG^{\ast}}\pi_{(E/F, \xi)}^{\BH})\) in Theorem \ref{cor:Tam} is valid regardless of \(\ch F\).
\end{rem}

\section{Kaletha's local Langlands correspondence} \label{sec:review-of-Kal}
Here we review Kaletha's local Langlands correspondence for regular supercuspidal representations defined in \cite{MR4013740},
which works for quite general tamely ramified connected reductive groups over \(F\),
and give a minor rephrasing that is suitable for our purpose.
We also introduce some notation when the theory is specialized to the groups \(\bfG^{\ast}=\GL_n, \bfG=\underline{A}^{\times}\).

Strictly speaking, a few assumptions on \(p\) are imposed in \cite{MR4013740} (cf.\ \cite[\S 2.1]{MR4013740}).
We will not recall them, but remark that for the groups \(\bfG^{\ast}=\GL_n, \bfG=\underline{A}^{\times}\) one only needs to assume that \(p\neq 2\). 

\subsection{Regular supercuspidal representations and their parametrization} \label{sec:rsc-rep}
Let \(\bfG\) be a tamely ramified connected reductive group over \(F\).
In \cite{MR1824988}, Yu introduced a construction producing an irreducible supercuspidal representation \(\pi_{\Psi}^{\Yu}\) of \(\bfG(F)\) from a concrete datum \(\Psi=((\bfG^0\subsetneq \bfG^1 \subsetneq \cdots \subsetneq \bfG^t), \pi_{-1}, (\phi_0, \phi_1, \dots, \phi_t))\), where
\begin{itemize}
	\item each \(\bfG^i\) is a tame twisted Levi subgroup of \(\bfG\) (i.e., a connected reductive subgroup defined over \(F\) that becomes a Levi subgroup over some tamely ramified extension),
	\item \(\pi_{-1}\) is a depth-zero supercuspidal representation of \(\bfG^0(F)\), and
	\item each \(\phi_i\) is a character of \(\bfG^i(F)\)
\end{itemize}
subject to various conditions.
In a later paper \cite{MR2431732} such a datum \(\Psi\) was called a reduced generic cuspidal \(\bfG\)-datum.
Following \cite{MR4013740}, we simply call \(\Psi\) a Yu-datum in this paper.

In their study \cite{MR2431732} of Yu's construction, Hakim--Murnaghan defined an equivalence relation, called \(\bfG\)-equivalence, on the set of Yu-data and proved that
two Yu-data \(\Psi\) and \(\Psi'\) give rise to isomorphic supercuspidal representations \(\pi_{\Psi}^{\Yu}\) and \(\pi_{\Psi'}^{\Yu}\) if and only if \(\Psi\) and \(\Psi'\) are \(\bfG\)-equivalent.
This result was proved in \cite[Theorem 6.6]{MR2431732} under a hypothesis \(C(\vec{\bfG})\), but the theorem was shown to hold without this hypothesis in \cite[\S 3.5]{MR4013740}.

Building on these results, Kaletha \cite[Definition 3.7.3]{MR4013740} defined the notion of regularity for (the \(\bfG\)-equivalence class of) a Yu-datum by imposing a certain condition on \(\pi_{-1}\).
An irreducible supercuspidal representation of \(\bfG(F)\) obtained from a regular Yu-datum by Yu's construction is called a \emph{regular supercuspidal representation}.
He then showed that the \(\bfG\)-equivalence classes of regular Yu-data are parametrized by simpler objects called \emph{tame elliptic regular pairs} (or tame elliptic regular pairs of \(\bfG\), when we want to stress the group \(\bfG\)).
More specifically, a tame elliptic regular pair \((\bfS, \xi)\) consists of 
\begin{itemize}
	\item a tamely ramified elliptic maximal torus \(\bfS\) of \(\bfG\) and
	\item a character \(\xi\) of \(\bfS(F)\) satisfying some regularity condition
\end{itemize}
(\cite[Definition 3.7.5]{MR4013740}),
and there exists an explicit bijection between the set of \(\bfG(F)\)-conjugacy classes of tame elliptic regular pairs and the set of \(\bfG\)-equivalence classes of regular Yu-data (\cite[Proposition 3.7.8]{MR4013740}).
We write \(\pi_{(\bfS, \xi)}^{\KY}\) for \(\pi_{\Psi}^{\Yu}\) if the \(\bfG\)-equivalence class of a regular Yu-datum \(\Psi\) corresponds to the \(\bfG(F)\)-conjugacy class of a tame elliptic regular pair \((\bfS, \xi)\) by this bijection.
Thus, the association \((\bfS, \xi) \mapsto \pi_{(\bfS, \xi)}^{\KY}\) induces a bijection between the set of \(\bfG(F)\)-conjugacy classes of tame elliptic regular pairs of \(\bfG\) and the set of isomorphism classes of regular supercuspidal representations of \(\bfG(F)\).
\[
\begin{tikzcd}
& \{\text{supercuspidal rep's of \(\bfG(F)\)}\}/{\sim} \\
\{\text{Yu-data}\}/\text{\(\bfG\)-eq.} \arrow[r,"1:1"] & \{\text{Yu's sc.\ rep's}\}/{\sim} \arrow[u, hookrightarrow] \\
\{\text{regular Yu-data}\}/\text{\(\bfG\)-eq.} \arrow[u, hookrightarrow] \arrow[r, "1:1"] & \{\text{regular sc.\ rep's}\}/{\sim} \arrow[u, hookrightarrow] \\
\{\text{tame elliptic regular pairs}\}/\text{\(\bfG(F)\)-conj.} \arrow[u, "1:1"] \arrow[ur, "1:1", "{(S, \xi) \mapsto\pi_{(\bfS, \xi)}^{\KY}}"']
\end{tikzcd}
\]

\begin{rem} \label{rem:Yu-error}
After the paper \cite{MR4013740} was published,
it was pointed out that some statements in Yu's paper (\cite[Proposition 14.1, Theorem 14.2]{MR1824988}) are false and consequently the proof of the irreducibility and supercuspidality of the representation \(\pi_{\Psi}^{\Yu}\) in \cite{MR1824988} is not correct as it stands.
	
Fortunately, there are already two ways to remedy this situation.
One is to resort to a result of Fintzen \cite{MR4357723}, 
which establishes the above irreducibility and supercuspidality 
without using \cite[Proposition 14.1, Theorem 14.2]{MR1824988}.
	
The other way due to Fintzen--Kaletha--Spice \cite{2106.09120} is to ``twist" each step in Yu's construction by subtle quadratic characters so that the original proof in \cite{MR1824988} works for this modified construction (see \cite[\S 4.1]{2106.09120}).
It is useful when one needs to apply \cite[Proposition 14.1, Theorem 14.2]{MR1824988} to auxiliary representations in the intermediate steps for other purposes.
What is important for this paper is the following (see \cite[\S 4.3, \S 4.4]{2106.09120}):
\begin{itemize}
\item If we write \(\pi_{\Psi}^{\Yu, \FKS}\) for the irreducible supercuspidal representation obtained from \(\Psi\) by the modified construction in \cite{2106.09120}, 
then \(\pi_{\Psi}^{\Yu, \FKS} \simeq \pi_{\Psi'}^{\Yu}\) for some \(\Psi'\).
The procedure \(\Psi \mapsto \pi_{\Psi}^{\Yu, \FKS}\) induces a modified parametrization of Yu's supercuspidal representations by Yu-data which preserves the regularity,
and hence also a modified parametrization of regular supercuspidal representations by tame elliptic regular pairs.
\item As reviewed in Section \ref{sec:Kal-LLC} below, 
Kaletha's local Langlands correspondence \(\LLC_{(\bfG, \psi)}^{\Kal}\) for regular supercuspidal representations is constructed in terms of the parametrization of these representations by tame elliptic regular pairs.
Even though the modified parametrization is used in later papers such as \cite{2106.09120}, \cite{1912.03274} to study (an extension of) the correspondence, the map \(\LLC_{(\bfG, \psi)}^{\Kal}\) (from certain regular supercuspidal representations to \(L\)-parameters) remains the same.
In other words, Kaletha's local Langlands correspondence in this paper is the same as the one discussed in \cite{2106.09120}, \cite{1912.03274} (restricted to those regular supercuspidal  representations).
\end{itemize}
\end{rem}

\subsubsection{Tame elliptic regular pairs and admissible pairs} \label{sec:ter-and-adm-pairs}
Now let \(\bfG=\underline{A}^{\times}\) be an inner form of \(\bfG^{\ast}=\GL_n\) as in Section \ref{sec:review-of-BH}.
(As remarked in the beginning of Section \ref{sec:review-of-Kal},
we assume \(p \neq 2\).)
Then the admissible pairs of degree \(n\) are essentially the same as the tame elliptic regular pairs of \(\bfG\) and \(\bfG^{\ast}\) in the following sense.
The bijections in Remark \ref{rem:GL_n-and-ell-tori} \eqref{item:field-extension} yields an identification of
\begin{itemize}
\item \(F\)-isomorphism classes of pairs \((E/F, \xi)\) of tamely ramified extension \(E/F\) of degree \(n\) and a character \(\xi\) of \(E^{\times}\) and 
\item \(\bfG^{\ast}(F)\)-conjugacy classes (resp.\ \(\bfG(F)\)-conjugacy classes) of pairs \((\bfS, \xi)\) of a tamely ramified elliptic maximal torus \(\bfS\) of \(\bfG^{\ast}\) (resp.\ of \(\bfG\)) and a character \(\xi\) of \(\bfS(F)\).
\end{itemize}
This in turn induces the following bijections: 
\begin{align*}
\{\text{adm.\ pairs of degree \(n\)}\}/{\text{\(F\)-isom.}} &\xrightarrow{1:1} 
\{\text{tame ell.\ reg.\ pairs of \(\bfG^{\ast}\)}\}/{\text{\(\bfG^	{\ast}(F)\)-conj.}} \\
&\xrightarrow{1:1}
\{\text{tame ell.\ reg.\ pairs of \(\bfG\)}\}/{\text{\(\bfG(F)\)-conj.}}
\end{align*}
The first bijection is shown in \cite[Lemma 3.7.7]{MR4013740} (note that the proof works even if \(p\) divides \(n\)).
The second bijection follows easily from the definition of the tame elliptic regular pair (which we omitted).

We identify tame elliptic regular pairs \((\bfS, \xi)\) of \(\bfG^{\ast}\) and of \(\bfG\) by the second bijection above and 
write \({}_{\bfG^{\ast}}\pi_{(\bfS, \xi)}^{\KY}\) (resp.\ \({}_{\bfG}\pi_{(\bfS, \xi)}^{\KY}\)) for the associated regular supercuspidal representation of \(\bfG^{\ast}(F)\) (resp.\ of \(\bfG(F)\)).
\begin{prop} \label{prop:BH-KY}
Let \((\bfS, \xi)\) be a tame elliptic regular pair of \(\bfG\) (or \(\bfG^{\ast}\)) and \((E/F, \xi)\) be an admissible pair of degree \(n\) such that the \(\bfG(F)\)-conjugacy class (or the \(\bfG^{\ast}(F)\)-conjugacy class) of \((\bfS, \xi)\) and the \(F\)-isomorphism class of \((E/F, \xi)\) correspond under the above bijections.
Then we have the following isomorphisms:
\[
{}_{\bfG}\pi_{(E/F, \xi)}^{\BH} \simeq {}_{\bfG}\pi_{(\bfS, \xi)}^{\KY},
\quad 
{}_{\bfG^{\ast}}\pi_{(E/F, \xi)}^{\BH} \simeq {}_{\bfG^{\ast}}\pi_{(\bfS, \xi)}^{\KY}.
\]
\end{prop}
This is probably well-known and in any case a proof for \(\bfG^{\ast}\) is outlined in \cite[Appendix A]{MR4206603}, but we include a sketch of a proof for \(\bfG\) in Appendix \ref{sec:app} of the present paper.
Henceforth, we simply write these representations as \({}_{\bfG}\pi_{(\bfS, \xi)}\) and \({}_{\bfG^{\ast}}\pi_{(\bfS, \xi)}\) respectively.

\begin{rem} \label{rem:F-non-sing}
In \cite{1912.03274} Kaletha extended his local Langlands correspondence in \cite{MR4013740} to a larger class of supercuspidal representations,
\emph{\(F\)-non-singular} supercuspidal representations.
They are again based on Yu's construction, but more general than regular supercuspidal representations,
and are similarly parametrized by 
tame elliptic \emph{\(F\)-non-singular} pairs.
However, for inner forms \(\bfG\) of \(\bfG^{\ast}=\GL_n\), these notions are the same.
To see this, we only need to check that any tame elliptic \(k_F\)-non-singular pair of \(\bfG\) comes from an admissible pair.
This can easily be done by noting that it is shown for \(\bfG^{\ast}=\GL_n\) in \cite[Remark 3.4.4]{1912.03274} and the \(F\)-non-singularity is preserved under the stable conjugacy by \cite[Remark 3.1.4]{1912.03274}.
\end{rem}

\subsection{Kaletha's LLC for regular supercuspidal representations} \label{sec:Kal-LLC}
Let \(\bfG^{\ast}\) be a tamely ramified connected reductive group over \(F\) that is quasi-split.

In this section we first review the construction of Kaletha's local Langlands correspondence for regular supercuspidal representations (\cite[\S 5]{MR4013740}).
In doing so, we also recall some necessary ingredients, in particular characters \(\epsilon\) and an explicit set \(\chi_{\Kal}\) of \(\chi\)-data, in some detail.
Then we give a (rather obvious) reformulation of it as maps \(\LLC_{(\bfG, \psi)}^{\Kal}\) from certain regular supercuspidal representations of inner twists \((\bfG, \psi)\) of \(\bfG^{\ast}\) to \(L\)-parameters of \(\bfG^{\ast}\) (cf.\ the diagram \eqref{eq:LLC^Kal}).
Finally, we specialize to the case where \(\bfG^{\ast}=\GL_n\).

While, strictly speaking, the domain of the map \(\LLC_{(\bfG, \psi)}^{\Kal}\) may not be the set of all regular supercuspidal representations of \(\bfG(F)\) in general,
we often call it Kaletha's local Langlands correspondence for regular supercuspidal representations in this paper (as we often did so far)
because that is the case for \(\bfG^{\ast}=\GL_n\) (cf.\ Remark \ref{rem:FKS-extrareg} below).

As noted in the beginning of Section \ref{sec:review-of-Kal},
a few assumptions on \(p\) have to be made,
but we omit the details in the general case and only remark that the condition \(p\neq 2\) suffices in our case.
In \cite[\S 5]{MR4013740}, the characteristic \(\ch F\) of \(F\) is also assumed to be zero.
This assumption can now be removed, thanks to the work of Dillery \cite{2008.04472} extending the formalism of rigid inner twists in \cite{MR3548533} to general characteristic
(see also Section 2.1 and the beginning of Section 5.2 in \cite{MR4013740}).

Kaletha's local Langlands correspondence for regular supercuspidal representations is defined in terms of two kinds of data called \emph{regular supercuspidal \(L\)-packet data} and \emph{regular supercuspidal data}.
The former are related to certain \(L\)-parameters of \(\bfG^{\ast}\) and the latter to certain regular supercuspidal representations of inner forms of \(\bfG^{\ast}\).

\subsubsection{Regular supercuspidal \(L\)-packet data and \(L\)-parameters} \label{sec:rsc-L-data}
A \emph{regular supercuspidal \(L\)-packet datum} \((\bfS, \widehat{j}, \chi, \xi)\) consists of
\begin{itemize}
\item a tamely ramified torus \(\bfS\) over \(F\) whose dimension is equal to the absolute rank of \(\bfG^{\ast}\),
\item an embedding \(\widehat{j} \colon \widehat{\bfS} \hookrightarrow \widehat{\bfG^{\ast}}\) of the dual torus \(\widehat{\bfS}\) into the dual group \(\widehat{\bfG^{\ast}}\) whose \(\widehat{\bfG^{\ast}}\)-conjugacy class is \(\Gamma_F\)-stable,
\item a set \(\chi\) of \(\chi\)-data, and
\item a character \(\xi\) of \(\bfS(F)\)
\end{itemize}
satisfying various conditions.
We do not recall the precise conditions imposed, but only note a few points.

As discussed in \cite[\S 5.1]{MR4013740},
there is a natural bijection between
the set of \(\Gamma_F\)-stable \(\bfG^{\ast}(F^{\sep})\)-conjugacy classes of embeddings \(\bfS \hookrightarrow \bfG^{\ast}\) over \(F^{\sep}\)
and 
the set of \(\Gamma_F\)-stable \(\widehat{\bfG^{\ast}}\)-conjugacy classes of embeddings \(\widehat{\bfS} \hookrightarrow \widehat{\bfG^{\ast}}\).
Thus, given an embedding \(\widehat{j}\) as above, its \(\widehat{\bfG^{\ast}}\)-conjugacy class
corresponds to a \(\Gamma_F\)-stable \(\bfG^{\ast}(F^{\sep})\)-conjugacy class \(J\) of embeddings \(\bfS \hookrightarrow \bfG^{\ast}\) over \(F^{\sep}\).
Since \(\bfG^{\ast}\) is quasi-split, \(J\) has a \(\Gamma_F\)-fixed element \(j\colon \bfS \hookrightarrow \bfG^{\ast}\) by \cite[Corollary 2.2]{MR683003} and \cite[Lemma 7.6]{2008.04472}.
Although the torus \(\bfS\) above is given abstractly,
many conditions imposed are formulated via this embedding \(j\).
For instance, the subtorus \(j\bfS \subset \bfG^{\ast}\) is assumed to be elliptic;
this condition is independent of the choice of \(j\in J\).
Also, pulling back \(\Phi(j\bfS, \bfG^{\ast})\) via \(j\) gives a \(\Gamma_F\)-invariant subset \(\Phi(\bfS, \bfG^{\ast}) \subset X^{\ast}(\bfS)\) and \(\chi\) is assumed to be a set ot \(\chi\)-data for \(\Phi(\bfS, \bfG^{\ast})\) implicitly; the subset \(\Phi(\bfS, \bfG^{\ast}) \subset X^{\ast}(\bfS)\) is independent of the choice of \(j\in J\).

Similarly to the construction before Proposition \ref{prop:Tam}, 
to a regular supercuspidal \(L\)-packet datum \((\bfS, \widehat{j}, \chi, \xi)\), we can associate an \(L\)-parameter of \(\bfG^{\ast}\) in the following way.
First by the local Langlands correspondence for the torus \(\bfS\),
one has an \(L\)-parameter \(\phi_{\xi} \colon W_F\to {}^{L}\bfS\) of \(\bfS\) associated to \(\xi\).
From \(\widehat{j}\) and \(\chi\), Langlands--Shelstad's construction yields an \(L\)-embedding \({}^{L}j_{\chi} \colon \widehat{\bfS} \to \widehat{\bfG^{\ast}}\).
Then by composing them we get an \(L\)-parameter
\[
{}^{L}j_{\chi} \circ \phi_{\xi} \colon W_F\to {}^{L}\bfG^{\ast}
\]
of \(\bfG^{\ast}\).
In fact, in \cite[Definition 5.2.3]{MR4013740} a class of \(L\)-parameters called \emph{regular supercuspidal \(L\)-parameters} is introduced and it is proved in \cite[Proposition 5.2.7]{MR4013740}
that the above procedure induces a bijection 
between the set of isomorphism classes of regular supercuspidal \(L\)-packet data 
and the set of \(\widehat{\bfG^{\ast}}\)-conjugacy classes of regular supercuspidal \(L\)-parameters.
We do not recall the definition of a regular supercuspidal \(L\)-parameter
nor of the isomorphism of regular supercuspidal \(L\)-packet data,
but will only use this bijection in an exposition in this section.

\subsubsection{Regular supercuspidal data and representations} 
\label{sec:rsc-data-and-rep}
A \emph{regular supercuspidal datum} is a tuple \((\bfS, \widehat{j}, \chi, \xi, (\bfG, \psi, z), j)\), where
\begin{itemize}
	\item \((\bfS, \widehat{j}, \chi, \xi)\) is a regular supercuspidal \(L\)-packet datum,
	\item \((\bfG, \psi, z)\) is a rigid inner twist of \(\bfG^{\ast}\) in the sense of \cite[\S 5.1]{MR3548533} (if \(\ch F=0\)) and \cite[Definition 7.1]{2008.04472} (if \(\ch F=p\)), and 
	\item \(j \colon \bfS \hookrightarrow \bfG\) is an embedding over \(F\)
\end{itemize}
satisfying various conditions.
In particular, it is assumed that the \(\bfG^{\ast}(F^{\sep})\)-conjugacy class of \(\psi^{-1} \circ j\) corresponds to the \(\widehat{\bfG^{\ast}}\)-conjugacy class of \(\widehat{j}\) under the natural bijection in \cite[\S 5.1]{MR4013740} recalled in the review of regular supercuspidal \(L\)-packet data above.
Rigid inner forms are crucial when one considers finer properties of the local Langlands correspondence,
but for the purpose of this paper
they merely play a minor role.
Thus we only recall the following basic facts:
\begin{itemize}
\item a rigid inner twist \((\bfG, \psi, z)\) of \(\bfG^{\ast}\)
consists of an inner twist \((\bfG, \psi)\) of \(\bfG^{\ast}\)
and an additional structure \(z\) (whose subtle definition 
we omit here),
\item the automorphism group \(\Aut(\bfG, \psi, z)\) is isomorphic to \(\bfG(F)\) acting on \(\bfG\) as the usual inner automorphism, and
\item any inner twist \((\bfG, \psi)\) admits a rigid inner twist \((\bfG, \psi, z)\) for some \(z\).
\end{itemize}
For more on rigid inner twists and their applications to the local Langlands correspondence, see \cite{MR3675168} as well as \cite{MR3548533} and \cite{2008.04472}.

A regular supercuspidal datum \((\bfS, \widehat{j}, \chi, \xi, (\bfG, \psi, z), j)\) gives rise to a regular supercuspidal representation of \(\bfG(F)\) as follows.
First we replace the datum with another datum 
\((\bfS, \widehat{j}, \chi_{\Kal}, \xi \cdot \zeta_{\bfS}^{-1}, (\bfG, \psi, z), j)\), where
\begin{itemize}
\item \(\chi_{\Kal}\) is an explicit set of \(\chi\)-data for \(\Phi(\bfS, \bfG^{\ast})\) in Step 2 in \cite[\S 5.3]{MR4013740} (denoted by \(\chi'\) in loc.\ cit.) and
\item \(\zeta_{\bfS}\) is the character of \(\bfS(F)\) defined by the \(\zeta\)-data \(\zeta=\{\chi_{\Kal, \alpha} \cdot \chi_{\alpha}^{-1}\}_{\alpha \in \Phi(\bfS, \bfG^{\ast})}\) as in \cite[Definition 4.6.5]{MR4013740}.
\end{itemize}
Then we consider the subtorus \(j\bfS \subset \bfG\) and 
define a character \(j\xi'\) of \(j\bfS(F)\) by
\[
\big( (\xi \cdot \zeta_{\bfS}^{-1}) \circ j^{-1}\big) \cdot
\epsilon_{f, \ram} \cdot \epsilon^{\ram},
\]
where \(\epsilon_{f, \ram}\) and \(\epsilon^{\ram}\) are quadratic characters of \(j\bfS(F)\) in Step 3 in \cite[\S 5.3]{MR4013740}.
The pair \((j\bfS, j\xi')\) forms a tame elliptic regular pair of \(\bfG\).
Finally, \(\pi_{(j\bfS, j\xi')}\) is the regular supercuspidal representation we associate to the given regular supercuspidal datum.
We also simply write \(\pi_{j}=\pi_{(j\bfS, j\xi')}\) when the given datum is clear from the context.

Let us elaborate a little on the objects \(\zeta_\bfS, \chi_{\Kal}, \epsilon_{f, \ram}, \epsilon^{\ram}\)
used to define the character \(j\xi'\) above.
The character \(\zeta_{\bfS}\) is trivial if \(\chi_{\Kal}\) and \(\chi\) are equal.
The set \(\chi_{\Kal}\) of \(\chi\)-data is associated to
\(\xi\) (and \(\Phi(\bfS, \bfG^{\ast})\), which only depends on \(\bfS\) and \(\widehat{j}\)) and in particular is independent of the rational embedding \(j\colon \bfS \hookrightarrow \bfG\) and the rigid inner twist \((\bfG, \psi, z)\).
By contrast, the characters \(\epsilon_{f, \ram}\) and \(\epsilon^{\ram}\) are associated to the tame elliptic regular pair \((j\bfS, j\xi)\) of \(\bfG\)
and are recalled later in Section \ref{sec:epsilon-character}.
To ease the notation, we put
\begin{equation} \label{eq:def-of-epsilon}
\epsilon:=\epsilon_{f, \ram} \cdot \epsilon^{\ram}.
\end{equation}

Again there is a natural notion of isomorphism of regular supercuspidal data (see \cite[Definition 5.3.3]{MR4013740}) and 
the regular supercuspidal representation \(\pi_{j}\) depends only on the isomorphism class of the regular supercuspidal datum \((\bfS, \widehat{j}, \chi, \xi, (\bfG, \psi, z), j)\).
In fact,  once a rigid inner twist \((\bfG, \psi, z)\) is fixed,
the above procedure induces an injection
(cf.\ the map (\(\star\)) in the commutative diagram \eqref{eq:rsc-data} below)
from the set of isomorphism classes of regular supercuspidal data whose underlying rigid inner twist is \((\bfG, \psi, z)\) to
the set of isomorphism classes of regular supercuspidal representations of \(\bfG(F)\).
Provisionally in this paper, we refer to those regular supercuspidal representations belonging to the image of this map as \emph{FKS-extraregular supercuspidal representations}.
\begin{rem} \label{rem:FKS-extrareg}
We use the term FKS-extraregular supercuspidal representations only for the exposition in this section.
What is important for us is the following:
\begin{itemize}
\item A regular supercuspidal representation \(\pi=\pi_{(\bfS, \xi)}\) is FKS-extraregular if and only if the tame elliptic regular pair \((\bfS, \xi \cdot \epsilon^{-1})\) (cf.\ \eqref{eq:explicit-rsc-dat} below) satisfies the extraregularity (\cite[Definition 3.7.5]{MR4013740}), which is a condition stronger than the regularity.
Or equivalently, this amounts to \(\pi\) corresponding to a tame elliptic extraregular pair under the modified parametrization of Fintzen--Kaletha--Spice \cite{2106.09120} (cf.\ Remark \ref{rem:Yu-error}).
\item If \(\bfG\) is an inner form of \(\bfG^{\ast}=\GL_n\), then any regular supercuspidal representation \(\pi\) of \(\bfG(F)\) is FKS-extraregular.
This can be seen by noting that any tame elliptic regular pair of \(\bfG\) is extraregular (cf.\ \cite[Lemma 3.7.7]{MR4013740}).
\end{itemize}
\end{rem}
For later purpose, we describe the inverse map from the set of FKS-extraregular supercuspidal representations.
Let \(\pi\) be an FKS-extraregular supercuspidal representation of \(\bfG(F)\).
Then there exists a tame elliptic regular pair \((\bfS, \xi)\) of \(\bfG\) such that \(\pi=\pi_{(\bfS, \xi)}\).
Note that here \(\bfS\) is given as a subtorus of \(\bfG\).
We write \(j\) for the inclusion \(\bfS \hookrightarrow \bfG\)
and \(\widehat{\psi^{-1} \circ j} \colon \widehat{\bfS} \hookrightarrow \widehat{\bfG^{\ast}}\) for the embedding induced by \(\psi^{-1} \circ j\).
Then the inverse map sends (the isomorphism class of)
\(\pi\) to
(the isomorphism class of) a regular supercuspidal datum
\begin{equation} \label{eq:explicit-rsc-dat}
(\bfS, \widehat{\psi^{-1} \circ j}, \chi_{\Kal}, (\xi \cdot \epsilon^{-1}) \circ j, (\bfG, \psi, z), j),
\end{equation}
where 
\begin{itemize}
\item \(\chi_{\Kal}\) is the set of \(\chi\)-data associated to \(\xi\)
and 
\item \(\epsilon=\epsilon_{f, \ram} \cdot \epsilon^{\ram}\) is the character of \(\bfS(F)\) associated to the tame elliptic regular pair \((\bfS, \xi)\).
\end{itemize}

Moreover, if we treat \(\pi_{j}\) together with the rigid inner twist \((\bfG, \psi, z)\),
namely regard the tuple \((\bfG, \psi, z, \pi_j)\) as a representation of a rigid inner twist of \(\bfG^{\ast}\) (in the sense of \cite[page 594]{MR3548533} and \cite[Definition 7.9]{2008.04472}),
essentially the same procedure even induces a bijection between 
the set of isomorphism classes of regular supercuspidal data
and the set of isomorphism classes of FKS-extraregular supercuspidal representations of rigid inner twists of \(\bfG^{\ast}\).
As the subset of the latter set consisting of those with the underlying rigid inner twist \((\bfG, \psi, z)\) 
is naturally identified with the set of isomorphism classes of FKS-extraregular supercuspidal representations of \(\bfG(F)\), the bijection in the last paragraph is induced by this bijection essentially via the restriction.

The situation is summarized by the following commutative diagram,
where the two sets with the subscript \((\bfG, \psi, z)\) indicate the subset defined by requiring the underlying rigid inner twist to be \((\bfG, \psi, z)\): 
\begin{equation} \label{eq:rsc-data}
\begin{tikzcd}
	\{\text{reg.\ sc.\ data}\}/{\sim} \arrow[r, "1:1"] &[2.8cm] \{\substack{\text{FKS-extrareg.\ sc.\ rep's of}\\ \text{rigid inner twists of \(\bfG^{\ast}\)}}\}/{\sim} \\
	\{\text{reg.\ sc.\ data}\}_{(\bfG, \psi, z)}/{\sim} \arrow[r, "1:1"', "{(\bfS, \widehat{j}, \chi, \xi, (\bfG, \psi, z), j)\mapsto (\bfG, \psi, z, \pi_j)}"] \arrow[u, hookrightarrow]  \arrow[dr, "1:1", "(\star) \colon {(\bfS, \widehat{j}, \chi, \xi, (\bfG, \psi, z), j)\mapsto \pi_j}"'] & \{\substack{\text{FKS-extrareg.\ sc.\ rep's of}\\ \text{rigid inner twists of \(\bfG^{\ast}\)}}\}_{(\bfG, \psi, z)}/{\sim} \arrow[u, hookrightarrow] \\
	&  
	\{\text{FKS-extrareg.\ sc.\ rep's of \(\bfG(F)\)}\}/{\sim} \arrow[u, "{\pi \mapsto (\bfG, \psi, z, \pi)}"', "1:1"]
\end{tikzcd}
\end{equation}

\subsubsection{Construction of \(L\)-packets}
Let \(\phi\) be a regular supercuspidal \(L\)-parameter of \(\bfG^{\ast}\) obtained from a regular supercuspidal \(L\)-packet datum
\((\bfS, \widehat{j}, \chi, \xi)\).
In \cite[page 1154]{MR4013740} Kaletha defined 
\[
\Pi_{\phi}=\{(\bfG, \psi, z, \pi_j) \mid (\bfS, \widehat{j}, \chi, \xi, (\bfG, \psi, z), j) \mapsto (\bfS, \widehat{j}, \chi, \xi)\},
\]
where \((\bfG, \psi, z, \pi_j)\) are considered up to isomorphism.
In other words, one collects \((\bfG, \psi, z, \pi_j)\) for each regular supercuspidal datum whose underlying regular supercuspidal \(L\)-packet datum is \((\bfS, \widehat{j}, \chi, \xi)\)
(see also the commutative diagram \eqref{eq:LLC^Kal} below and its explanation).
He calls \(\Pi_{\phi}\) the compound \(L\)-packet;
it should be understood as the disjoint union of \(L\)-packets of rigid inner twists of \(\bfG^{\ast}\):
\[
\Pi_{\phi}=\coprod_{(\bfG, \psi, z)} \Pi_{\phi}(\bfG, \psi, z), 
\]
where \(\Pi_{\phi}(\bfG, \psi, z)\) denotes the subset consisting of representations of the rigid inner twist \((\bfG, \psi, z)\).

\subsubsection{\(L\)-parameter associated to an FKS-extraregular supercuspidal representation} \label{sec:LLC^Kal}
Now let \((\bfG, \psi)\) be an inner twist of \(\bfG^{\ast}\).
Let us interpret Kaletha's local Langlands correspondence 
as a map \(\LLC_{(\bfG, \psi)}^{\Kal}\) from the set of isomorphism classes of FKS-extraregular supercuspidal representations of \(\bfG(F)\) 
to the set of \(\widehat{\bfG^{\ast}}\)-conjugacy classes of \(L\)-parameters of \(\bfG^{\ast}\);
this is more useful for the purpose of this paper.

Take some \(z\) so that \((\bfG, \psi, z)\) is a rigid inner twist and
let \(\pi\) be an FKS-extraregular supercuspidal representation of \(\bfG(F)\).
By the bijection discussed before,
we have the regular supercuspidal data
\((\bfS, \widehat{j}, \chi, \xi, (\bfG, \psi, z), j)\)
corresponding to \(\pi\):
if \(\pi=\pi_{(\bfS, \xi)}\), then \((\bfS, \widehat{j}, \chi, \xi, (\bfG, \psi, z), j)\) is expressed as \eqref{eq:explicit-rsc-dat}.
Then we define \(\LLC_{(\bfG, \psi)}^{\Kal}(\pi)\) to be the \(L\)-parameter of \(\bfG^{\ast}\) obtained from the regular supercuspidal \(L\)-packet datum \((\bfS, \widehat{j}, \chi, \xi)\).

The following commutative diagram describes the map \(\LLC_{(\bfG, \psi)}^{\Kal}\) as well as some other related maps:
\begin{equation} \label{eq:LLC^Kal}
\begin{tikzcd}
\{\text{reg.\ sc.\ \(L\)-packet data}\}/{\sim} \arrow[r,"1:1"] & \{\text{reg.\ sc.\ \(L\)-par's of \(\bfG^{\ast}\)}\}/{\text{\(\widehat{\bfG^{\ast}}\)-conj.}} \\
\{\text{reg.\ sc.\ data}\}/{\sim} \arrow[u, "{(\ast)}"] \arrow[r, "1:1"] & \{\substack{\text{FKS-extrareg.\ sc.\ rep's of}\\ \text{rigid inner twists of \(\bfG^{\ast}\)}}\}/{\sim} \arrow[u, dashed, "{(\ast \ast)}"] \\
\{\text{reg.\ sc.\ data}\}_{(\bfG, \psi, z)}/{\sim} \arrow[r, "1:1"] \arrow[u, hookrightarrow]  & \{\text{FKS-extrareg.\ sc.\ rep's of \(\bfG(F)\)}\}/{\sim} \arrow[u, hookrightarrow] \arrow[uu, bend right=80, "\LLC_{(\bfG, \psi)}^{\Kal}"']
\end{tikzcd}
\end{equation}
Here, the arrow \((\ast)\) sends (the isomorphism class of) a regular supercuspidal datum to (the isomorphism class of) its underlying regular supercuspidal \(L\)-packet datum and the dashed arrow \((\ast \ast)\) is the induced one making the diagram commutative.
The compound \(L\)-packet \(\Pi_{\phi}\) is nothing but the fiber of the map \((\ast \ast)\) at \(\phi\).

It is clear on unwinding the definition that \(\LLC_{(\bfG, \psi)}^{\Kal}(\pi)\) does not depend on the choice of \(z\) (as should be the case).
On the other hand, \(\LLC_{(\bfG, \psi)}^{\Kal}(\pi)\) does depend on the choice of the inner twist \(\psi\).
This is reasonable because we implicitly use the identification of \(L\)-parameters of \(\bfG\) with those of \(\bfG^{\ast}\) via the isomorphism \({}^{L}\psi \colon {}^{L}\bfG \xrightarrow{\sim} {}^{L}\bfG^{\ast}\) induced by \(\psi\).
The \(L\)-parameter \({}^{L}\psi^{-1} \circ \LLC_{(\bfG, \psi)}^{\Kal}(\pi)\) of \(\bfG\) is independent of the choice of \(\psi\).

Let us briefly discuss the fibers of the map \(\LLC_{(\bfG, \psi)}^{\Kal}\), i.e., the sets \(\Pi_{\phi}(\bfG, \psi, z)\).
Let \(\phi\) be a regular supercuspidal \(L\)-parameter of \(\bfG^{\ast}\) and \((\bfS, \widehat{j}, \chi, \xi)\) the corresponding regular supercuspidal \(L\)-packet datum.
As explained in Section \ref{sec:rsc-L-data}, \(\widehat{j}\) defines a \(\Gamma_F\)-stable \(\bfG^{\ast}(F^{\sep})\)-conjugacy class \(J\) of embeddings \(\bfS \hookrightarrow \bfG^{\ast}\) over \(F^{\sep}\).
Then, in view of the diagram \eqref{eq:LLC^Kal}, the fiber at \(\phi\) of the map \(\LLC_{(\bfG, \psi)}^{\Kal}\) is parametrized by the set \(\calJ\) of embeddings \(j \colon \bfS \hookrightarrow \bfG\) over \(F\) such that \(\psi^{-1} \circ j\) lies in \(J\).
It follows from \cite[\S 10]{MR858284} and \cite[Lemma 3.2.1]{MR4013740} that \(\calJ\) is non-empty because of the ellipticity assumption on \(\bfS\).
As \(\phi\) is an arbitrary regular supercuspidal \(L\)-parameter, this shows that \(\LLC_{(\bfG, \psi)}^{\Kal}\) is surjective.
(Here the target of the map \(\LLC_{(\bfG, \psi)}^{\Kal}\) is the set of \(\widehat{\bfG^{\ast}}\)-conjugacy classes of regular supercuspidal \(L\)-parameters, and not arbitrary \(L\)-parameters.)
Let \(\phi\) and \(\calJ\) be as above and fix some \(j \in \calJ\).
Then \(\calJ\) is naturally parametrized by the set
\begin{equation} \label{eq:calJ-as-Ker}
\Ker(H^1(F, \bfS) \xrightarrow{j} H^1(F, \bfG)),
\end{equation}
where the map between the two cohomology sets is the one induced by \(j\).
This can be proved either directly without difficulty or by comparing with the parametrization of the compound \(L\)-packet \(\Pi_{\phi}\) found in \cite[(5.1.3)]{MR4013740} and \cite[page 594]{MR3548533}.

\subsubsection{Case of inner forms of \(\GL_n\)} \label{sec:LLC^Kal-for-GL_n}
Now let us specialize and return to the setting of Section \ref{sec:review-of-BH}:
\(\bfG=\underline{A}^{\times}\) is an inner form of \(\bfG^{\ast}=\GL_n\).
We also have a canonical inner twist \(\psi \colon \bfG^{\ast} \to \bfG\) (see Section \ref{sec:notation}).
Thus we assume \(p \neq 2\), as explained in the beginning of Section \ref{sec:review-of-Kal}.
As any regular supercuspidal representation of \(\bfG(F)\) is FKS-extraregular by Remark \ref{rem:FKS-extrareg}, we no longer distinguish the two notions.

In this case, the maps \(\LLC_{(\bfG, \psi)}^{\Kal}\) and \(\LLC_{(\bfG^{\ast}, \id)}^{\Kal}\) are both bijective.
Indeed, the surjectivity generally holds as discussed above.
The injectivity follows from 
the triviality of the set \eqref{eq:calJ-as-Ker};
if \(\phi\) is a regular supercuspidal \(L\)-parameter of \(\bfG^{\ast}\) and \((\bfS, \widehat{j}, \chi, \xi)\) is the corresponding regular supercuspidal \(L\)-packet datum, then 
\(\bfS\) is an induced torus in this case
and thus \(H^1(F, \bfS)\) is trivial by Shapiro's lemma and Hilbert's theorem 90.

We put
\[
\LLC_{\bfG^{\ast}}^{\Kal}:=\LLC_{(\bfG^{\ast}, \id)}^{\Kal},
\quad
\LJLC^{\Kal}:=
(\LLC_{(\bfG, \psi)}^{\Kal})^{-1} \circ \LLC_{\bfG^{\ast}}^{\Kal}.
\]
To make these correspondences explicit,
let \((\bfS, \xi)\) be a tame elliptic regular pair of \(\bfG^{\ast}\).
As reviewed in Section \ref{sec:rsc-data-and-rep}, a set \(\chi_{\Kal}\) of \(\chi\)-data is associated to \(\xi\).
According to Section \ref{sec:ter-and-adm-pairs}, \((\bfS, \xi)\) can also be regarded as a tame elliptic regular pair of \(\bfG\).
We write \(\eps{\bfG^{\ast}}{}\)
(resp.\ \(\eps{\bfG}{}\)) for the characters 
\(\epsilon\) of \(\bfS(F)\) in Section \ref{sec:rsc-data-and-rep} associated to the tame elliptic regular pair \((\bfS, \xi)\) of \(\bfG^{\ast}\) (resp.\ of \(\bfG\)).
\begin{prop} \label{prop:LLC^Kal-explicit}
Let \((\bfS, \xi), \chi_{\Kal}, \eps{\bfG^{\ast}}{}, \eps{\bfG}{}\) be
as above.
Let \({}_{\bfG^{\ast}}\pi_{(\bfS, \xi)}\) be the regular supercuspidal representation of \(\bfG^{\ast}(F)\) corresponding to \((\bfS, \xi)\).
We put
\[
\xi':=\xi \cdot \eps{\bfG^{\ast}}{}^{-1} \cdot \eps{\bfG}{}.
\]
Then we have 
\[
\LLC_{\bfG^{\ast}}^{\Kal}({}_{\bfG^{\ast}}\pi_{(\bfS, \xi)})
=\Ind_{W_E}^{W_F} (\mu_{\chi_{\Kal}} \cdot \xi \cdot \eps{\bfG^{\ast}}{}^{-1}),
\quad
\LJLC^{\Kal}({}_{\bfG^{\ast}}\pi_{(\bfS, \xi)})={}_{\bfG}\pi_{(\bfS, \xi')}.
\]
\end{prop}
\begin{proof}
The description of \(\LLC_{\bfG^{\ast}}^{\Kal}({}_{\bfG^{\ast}}\pi_{(\bfS, \xi)})\) is given in \cite[Proposition 4.6]{MR4206603} 
and it follows from the definition of \(\LLC_{(\bfG^{\ast}, \id)}^{\Kal}=\LLC_{\bfG^{\ast}}^{\Kal}\) in Section \ref{sec:LLC^Kal} and Proposition \ref{prop:Tam};
the \(L\)-parameter \(\LLC_{\bfG^{\ast}}^{\Kal}({}_{\bfG^{\ast}}\pi_{(\bfS, \xi)})\) is \({}^{L}j_{\chi_{\Kal}} \circ \phi_{\xi \cdot \eps{\bfG^{\ast}}{}^{-1}}\),
which is indeed \(\Ind_{W_E}^{W_F} (\mu_{\chi_{\Kal}} \cdot \xi \cdot \eps{\bfG^{\ast}}{}^{-1})\),
when regarded as an \(n\)-dimensional representation of \(W_F\).

Similarly, we compute
\begin{align*}
\LLC_{(\bfG, \psi)}^{\Kal}({}_{\bfG}\pi_{(\bfS, \xi')})
&=\Ind_{W_E}^{W_F} (\mu_{\chi_{\Kal}} \cdot \xi' \cdot \eps{\bfG}{}^{-1}) \\
&=\Ind_{W_E}^{W_F} (\mu_{\chi_{\Kal}} \cdot \xi \cdot \eps{\bfG^{\ast}}{}^{-1})
=\LLC_{\bfG^{\ast}}^{\Kal}({}_{\bfG^{\ast}}\pi_{(\bfS, \xi)}),
\end{align*}
which shows the desired description of \(\LJLC^{\Kal}({}_{\bfG^{\ast}}\pi_{(\bfS, \xi)})\).
Note that here we implicitly used the fact that \(\chi_{\Kal}\) and \(\eps{\bfG}{}\) associated to \((\bfS, \xi')\) and those associated to \((\bfS, \xi)\) are the same.
Although this will follow easily if we recall the definitions, we omit the details.
(It suffices to check that \(\eps{\bfG^{\ast}}{}, \eps{\bfG}{}\) are both tamely ramified and that the dependence of \(\chi_{\Kal}, \eps{\bfG}{}\) on \(\xi\) is only through its restriction to \(\bfS(F)_{0+}=1+\frakp_E\).)
\end{proof}
\begin{rem} \label{rem:exp-for-Kal}
By Remark \ref{rem:normalization-of-Tam-prop}, the expression for \(\LLC_{\bfG^{\ast}}^{\Kal}({}_{\bfG^{\ast}}\pi_{(\bfS, \xi)})\) in Proposition \ref{prop:LLC^Kal-explicit} is valid with the arithmetic normalization.
\end{rem}

\section{Main theorem} \label{sec:main-thm}
Let \(\bfG^{\ast}=\GL_n\) and \(\bfG=\underline{A}^{\times}\) be as in Section \ref{sec:review-of-BH}.
\subsection{Main theorem}
\begin{thm} \label{thm:main}
Assume that \(p \neq 2\).
Then we have
\[
\LJLC^{\cl}(\pi)=\LJLC^{\Kal}(\pi)
\]
for any regular supercuspidal representation \(\pi\) of \(\bfG^{\ast}(F)=\GL_n(F)	\).
\end{thm}
\begin{rem} \label{rem:main}
In fact, as explained below, this theorem follows from general results in \cite{2106.09120}, \cite{2301.09812} if we make suitable assumptions on \(n, p, q\) and so on.
Our proof works without any assumptions (except for \(p\neq 2\), which is necessary for the right-hand side to make sense),
especially thanks to the work \cite{MR2848585}.
\begin{itemize}
\item If \(\ch F=0\) and \(p \geq (e_F+2)n\) (where \(e_F\) is the ramification index of \(F/\bbQ_p\)),
then Theorem \ref{thm:main} follows from \cite[Theorem 4.4.4.1]{2106.09120}, which establishes the endoscopic character relation between inner forms for Kaletha's \(L\)-packets in \cite{MR4013740} and \cite{1912.03274}.
Indeed, when applied to \(\bfG^{\ast}\) and \(\bfG\), the statement of the theorem specializes to the character relation of the local Jacquet--Langlands correspondence.
Note that the assumptions on \(\ch F\) and \(p\) ensure that the character formula \cite[Theorem 4.3.5]{2106.09120} is applicable and appear in the paragraphs right after Proposition 4.3.2.

\item Let \(\pi\) be as in Theorem \ref{thm:main} and express it as \(\pi={}_{\bfG^{\ast}}\pi_{(\bfS, \xi)}\) with \(\bfS \simeq \Res_{E/F}\Gm\) for some tamely ramified extension \(E/F\) of degree \(n\).
Under the assumption on \(\bfS\) below,
\cite[Theorem 9.10]{2301.09812} describes \(\LJLC^{\cl} (\pi)\) in terms of a tame elliptic regular pair of \(\bfG\)
and thus one can check that it agrees with \(\LJLC^{\Kal}(\pi)\)
to deduce Theorem \ref{thm:main}
(note that the modified parametrization of \cite{2106.09120}, which we mention in Remark \ref{rem:Yu-error}, is adopted in \cite{2301.09812}).
The basic assumption in \cite{2301.09812} reads in our setting that \(p\) does not divide \(n!\).
In addition to this, the assumption of \cite[Theorem 9.10]{2301.09812} requires that either of the following should be satisfied
(here we write \(e=e(E/F), f=f(E/F)\) and \(\phi\) is Euler's totient function):
\begin{itemize}
\item \(E/F\) is unramified, \(n\) is a prime such that \((n, q) \neq (2, 2), (2, 3)\);
\item \(E\) is totally ramified;
\item \(f\) is a prime satisfying \(\frac{e}{e-\phi(e)} >2f\) and \(q\) is sufficiently large.
\end{itemize}
Let us briefly comment on this assumption and also on the approach.
In \cite{2301.09812}, 
it is proved among other things
that a regular supercuspidal representation can be characterized by the restriction of the character to ``elliptic very regular elements", where the character values are easier to compute, provided that there exist sufficiently many such elements (\cite[Theorem 9.7]{2301.09812}).
As a sample application of the theorem,
\(\LJLC^{\cl}(\pi)\) is determined in \cite[Theorem 9.10]{2301.09812} by studying the restriction of the characters.
The above assumption ensures that ``Henniart inequality" holds, which quantifies the sufficiency of elliptic very regular elements.
\end{itemize}
Our proof of Theorem \ref{thm:main} may be close in spirit to that of \cite{2301.09812}, as the former relies on Bushnell--Henniart \cite{MR2848585}, where \(\LJLC^{\cl} (\pi)\) is investigated mainly through the evaluation of the character at certain special elements.
\end{rem}

We will prove Theorem \ref{thm:main} in Section \ref{sec:pf-main-thm}.

\subsection{Some complements for \(\GL_n\)}
The following theorem is basically proved in a joint paper \cite{MR4206603} with Masao Oi.
\begin{thm}[{\cite[Theorem 6.1]{MR4206603}}] \label{thm:GL_n}
Assume that \(p \neq 2\).
Then we have
\[
\LLC_{\bfG^{\ast}}^{\cl}(\pi)=\LLC_{\bfG^{\ast}}^{\Kal}(\pi)
\]
for any regular supercuspidal representation \(\pi\) of \(\bfG^{\ast}(F)=\GL_n(F)	\).
\end{thm}
\begin{rem} \label{rem:complement}
Let us complement some points in the proof in \cite{MR4206603}.

First, we recall the outline of the proof.
There exists a tame elliptic regular pair \((\bfS, \xi)\) of \(\bfG^{\ast}\) such that \(\pi \simeq {}_{\bfG^{\ast}}\pi_{(\bfS, \xi)}\).
Let \(E/F\) be a tamely ramified extension of degree \(n\) such that \(\bfS \simeq \Res_{E/F} \Gm\), so that \((\bfS, \xi)\) is identified with an admissible pair \((E/F, \xi)\).
As in Section \ref{sec:Tam-thm}, we get \(\chi_{\Tam}\) associated to \((E/F, \xi)\).
Similarly, Kaletha's theory reviewed in Section \ref{sec:Kal-LLC} (and Section \ref{sec:LLC^Kal-for-GL_n}) yields \(\chi_{\Kal}\) and \(\eps{\bfG^{\ast}}{}\) associated to \((\bfS, \xi)\).
By Corollary \ref{cor:Tam} and Proposition \ref{prop:LLC^Kal-explicit},
we have
\[
\LLC_{\bfG^{\ast}}^{\cl}(\pi)
=\Ind_{W_E}^{W_F} (\mu_{\chi_{\Tam}}^{-1} \cdot \xi),
\quad
\LLC_{\bfG^{\ast}}^{\Kal}(\pi)
=\Ind_{W_E}^{W_F} (\mu_{\chi_{\Kal}} \cdot \xi \cdot \eps{\bfG^{\ast}}{}^{-1}).
\]
Thus, it suffices to show an equality of characters of \(E^{\times}\):
\begin{equation} \label{eq:key-OT}
\mu_{\chi_{\Tam}}^{-1}=\mu_{\chi_{\Kal}} \cdot \eps{\bfG^{\ast}}{}^{-1}
\end{equation}
to prove the theorem.
Much of \cite{MR4206603} is devoted to establishing \eqref{eq:key-OT}.

We remark that although the paper \cite{MR4206603} does not discuss the difference of the normalizations of the local Langlands correspondence used in \cite{MR2138141}, \cite{MR2148193}, \cite{MR2679700}, \cite{MR3509939}, \cite{MR909227} and \cite{MR4013740}, that does not cause any serious problem.
Indeed, by Remarks \ref{rem:exp-for-cl}, \ref{rem:exp-for-Kal},
the reduction to the equality \eqref{eq:key-OT} is valid.
Now the equality \eqref{eq:key-OT} of characters of \(E^{\times}\) does not involve the normalization and is thus legitimate.

Moreover, the assumption \(\ch F=0\) in \cite{MR4206603} is not necessary.
Indeed, by Remark \ref{rem:exp-for-cl}, 
the reduction to the equality \eqref{eq:key-OT} works in positive characteristic as well
and it is easy to check that none of the computations in \cite{MR4206603} makes use of the assumption.
\end{rem}

\begin{rem}
As we saw in Section \ref{sec:review-of-BH}, neither Bushnell--Henniart \cite{MR2138141}, \cite{MR2148193}, \cite{MR2679700} nor Tam \cite{MR3509939} assume that \(p \neq 2\).
Thus, it may be worthwhile to try to extend some parts of Kaletha's theory reviewed in Section \ref{sec:review-of-Kal} to this case (at least in type \(A\), for which \(p=2\) is not a bad prime).
\end{rem}

To prove Theorem \ref{thm:main} along the lines of the proof of Theorem \ref{thm:GL_n}, we need to take a closer look at the characters \(\eps{\bfG^{\ast}}{}, \eps{\bfG}{}\) in Proposition \ref{prop:LLC^Kal-explicit} and the set \(\zeta_{\Tam}\) of \(\zeta\)-data in Corollary \ref{cor:Tam}.
We will do this in the next two sections.

\section{Preliminaries I} \label{sec:prelim1}
We first review (Kaletha's version of) the Howe factorization and the definition of the characters \(\epsilon^{\ram}, \epsilon_{f, \ram}\) (and thus of \(\eps{\bfG^{\ast}}{}, \eps{\bfG}{}\)).
Then, after recalling some facts about symmetric roots in \(\GL_n\)
from \cite{MR3509939} and \cite{MR4206603},
we prove an equality related to the character \(\epsilon_{f, \ram}\).

\subsection{Howe factorization} \label{sec:Howe-factorization}
Let \(\bfG\) be a connected reductive group over \(F\).
We review the notion of a Howe factorization and related objects in \cite[\S 3.6]{MR4013740}.
Again, we do not recall the precise assumptions on \(p\) in the general case and
only remark that they simply amount to \(p \neq 2\) for 
\(\bfG^{\ast}=\GL_n\) and \(\bfG=\underline{A}^{\times}\),
but are slightly stricter for more general groups. 

The material in this section will be used immediately in Section \ref{sec:epsilon-character} for the definition of the characters \(\epsilon^{\ram}, \epsilon_{f, \ram}\) and also in Section \ref{sec:vmod} and Appendix \ref{sec:app}.

Let \(\bfS \subset \bfG\) be a tamely ramified maximal torus and \(\xi \colon \bfS(F) \to \bbC^{\times}\) a character.
Let \(L\) be the splitting field of \(\bfS\).
For a real number \(r >0\), we put
\[
\Phi_r:=\{ \alpha \in \Phi(\bfS, \bfG) \mid \xi(\Nr_{L/F} (\alpha^{\vee}(L^{\times}_r)))=1\}.
\]
Here, \(\alpha^{\vee}\) is the coroot associated to \(\alpha\) and \(L^{\times}_r\) denotes the \(r\)-th Moy--Prasad subgroup of \(L^{\times}\) (more precisely, of the torus \(\Res_{L/F} \Gm\)),
or explicitly, 
\[
L^{\times}_r=1+\frakp_L^{\frac{r}{e(L/F)}}.
\]
Then \(\{\Phi_r\}_{r\in \bbR_{>0}}\) gives a \(\Gamma_F\)-invariant increasing filtration of \(\Phi(\bfS, \bfG)\) and we have \(\Phi_r=\Phi(\bfS, \bfG)\) for sufficiently large \(r\).
We also put
\[
\Phi_{r+}:=\bigcap_{s>r} \Phi_s
\]
for \(r \geq 0\)
and let \(r_{t-1} > \cdots > r_0 > 0\) be the set of breaks
(i.e., real numbers \(r>0\) for which \(\Phi_r \neq \Phi_{r+}\)).
If \(\Phi_{0+}=\Phi(\bfS, \bfG)\), then we set \(t:=0\).
Moreover we put \(r_{-1}:=0\) and \(r_t:=\depth(\xi)\).
According to \cite[Lemma 3.6.1]{MR4013740}, each subset \(\Phi_r \subset \Phi\) is a Levi subsystem.
For \(0\leq i\leq t\), we define \(\bfG^i\) to be the Levi subgroup of \(\bfG\) with maximal torus \(\bfS\) and root system \(\Phi_{r_{i-1}+}\).
We also define \(\bfG^{-1}:=\bfS\).
\begin{defn}[{\cite[Definition 3.6.2]{MR4013740}}] \label{def:Howe-facto}
A Howe factorization of \((\bfS, \xi)\) is a set of characters
\[
\phi_i \colon \bfG^i(F) \to \bbC^{\times} \quad (i=-1, \dots, t)
\]
such that the following holds:
\begin{enumerate}
\item \label{item:Howe1}
we have
\[
\xi=\prod_{i=-1}^t \phi_i|_{\bfS(F)},
\]
\item \label{item:Howe2}
for \(0 \leq i\leq t\), the character \(\phi_i\) is trivial on \(\bfG^i_{\mathrm{sc}}(F)\)
(the image in \(\bfG^i(F)\) of the set of \(F\)-valued points of the simply connected cover of the derived subgroup of \(\bfG^i\)), 
\item \label{item:Howe3}
for \(0 \leq i \leq t-1\), the character \(\phi_i\) has depth \(r_i\) and is \(\bfG^{i+1}\)-generic, and
\item \label{item:Howe4}
for \(i=t, -1\), the character \(\phi_i\) satisfies the following:
\[
\phi_t
\begin{cases}
\text{is trivial} & \text{if \(r_t=r_{t-1}\)} \\
\text{has depth \(r_t\)} & \text{otherwise,}
\end{cases}
\quad 
\phi_{-1}
\begin{cases}
\text{is trivial} &\text{if \(\bfG^0=\bfS\)} \\
\text{is trivial on \(\bfS(F)_{0+}\)} &\text{otherwise.}
\end{cases}
\]
\end{enumerate}
\end{defn}
We refer the reader to \cite[Definition 3.9]{MR2431732} for the definition of \(\bfG^{i+1}\)-genericity.

By \cite[Proposition 3.6.7]{MR4013740}, a Howe factorization of \((\bfS, \xi)\) always exists.

\subsubsection{Relation with the classical Howe factorization} \label{sec:classical-Howe-factorization}
The notion of Howe factorization of \((\bfS, \xi)\) as above, which works for general tamely ramified reductive groups (with some mild conditions on \(p\)),
is modeled on the classical factorization of an admissible character due to Howe \cite{MR0492087}.
As we will also use the latter in this paper, let us briefly recall what it is like and its relation with the above factorization of \((\bfS, \xi)\)
(see \cite[\S 3.5]{MR2431732} for more details).

A Howe factorization of an admissible pair \((E/F, \xi)\) consists of
\begin{itemize}
\item a sequence of subfields \(F=E_t \subsetneq E_{t-1} \subsetneq \cdots \subsetneq E_1 \subsetneq E_0 \subset E_{-1}=E\), and
\item characters \(\xi_i \colon E_i^{\times} \to \bbC^{\times}\) for \(-1 \leq i \leq t\)
satisfying
\[
\xi=\prod_{i=-1}^{t} (\xi_i \circ \Nr_{E/E_i})
\]
\end{itemize}
and some other conditions (see, e.g., \cite[Definition 3.33]{MR2431732} for the precise definition).

Let \(\bfG=\underline{A}^{\times}\) be as in Section \ref{sec:review-of-BH}
and \((\bfS, \xi)\) be a tame elliptic regular pair of \(\bfG\).
In this case, \((\bfS, \xi)\) corresponds to an admissible pair \((E/F, \xi)\) of degree \(n\) by the bijections in Section \ref{sec:ter-and-adm-pairs}.
A Howe factorization of \((\bfS, \xi)\) gives rise to a Howe factorization of an admissible pair \((E/F, \xi)\) as follows.
The inclusion \(\bfS \subset \bfG\), together with the choice of the isomorphism \(\bfS \simeq \Res_{E/F} \Gm\), yields an embedding \(E \hookrightarrow A\) of \(F\)-algebras, with which we regard \(E\) as an \(F\)-subalgebra of \(A\).
Then one can show that there exists a sequence of subfields 
\(F=E_t \subsetneq E_{t-1} \subsetneq \cdots \subsetneq E_1 \subsetneq E_0 \subset E_{-1}=E\) such that,
if we write \(A_i \subset A\) for the central simple \(E_i\)-algebra given as the centralizer of \(E_i\) in \(A\),
then we have
\[
\bfG^i=\Res_{E_i/F} \underline{A_i}^{\times}
\]
in \(\bfG\) for any \(-1\leq i\leq t\).
The condition \eqref{item:Howe2} implies that \(\phi_i\) factors through the reduced norm map \(\Nrd_{A_i/E_i} \colon A_i^{\times}=\bfG^i(F) \to E_i^{\times}\) for \(0\leq i\leq t\); we take a character \(\xi_i\) of \(E_i^{\times}\) such that 
\[
\phi_i=\xi_i \circ \Nrd_{A_i/E_i}.
\]
We set \(\xi_{-1}:=\phi_{-1}\).
Since \(\Nrd_{A_i/E_i}\) restricts on \(E^{\times}\) to the norm map \(\Nr_{E/E_i}\), the condition \eqref{item:Howe1} translates into
\[
\xi=\prod_{i=-1}^{t} (\xi_i \circ \Nr_{E/E_i}).
\]
Moreover, if we put
\[
a_i:=r_ie(E/F)
\]
for \(-1\leq i\leq t\),
then the conditions \eqref{item:Howe3} and \eqref{item:Howe4} in Definition \ref{def:Howe-facto} show (among other things) that
\begin{itemize}
\item for \(0\leq i\leq t-1\), the character \(\xi_i \circ \Nr_{E/E_i}\) has the \(E\)-level \(a_i\) in the sense that \(\min \{a \in \bbZ_{\geq 0} \mid U_E^{a+1} \subset \Ker (\xi_i \circ \Nr_{E/E_i})\}=a_i\), and
\item for \(i=-1, t\), the character \(\xi_i \circ \Nr_{E/E_i}\) is trivial on \(U_E^{a_i+1}\).
\end{itemize}
(These assertions follow from, e.g., ``Level and depth of a multiplicative character" in \cite[page 2013]{MR4206603} and \cite[Lemma 2.52]{MR2431732}, which are applicable because \(E/F\) is tamely ramified.)
These data \((E_i)_{i=-1}^t\), \((\xi_i)_{i=-1}^t\) form a Howe factorization of \((E/F, \xi)\). Later we will also use the sequence of integers \((a_i)_{i=-1}^t\).

\subsection{The characters \(\epsilon^{\ram}\) and \(\epsilon_{f, \ram}\)} \label{sec:epsilon-character}
We recall the definition of the quadratic characters \(\epsilon^{\ram}\) and \(\epsilon_{f, \ram}\) of \(\bfS(F)\) following \cite[\S 4.3, Lemma 4.7.4]{MR4013740}.

Let \((\bfS, \xi)\) be a tame elliptic regular pair of \(\bfG\).
By the procedure recalled in Section \ref{sec:Howe-factorization}, we have
\begin{itemize}
\item an integer \(t\geq 0\) and real numbers \(r_t \geq r_{t-1}> \cdots > r_0>r_{-1}=0\), and
\item Levi subgroups \(\bfS=\bfG^{-1}\subset \bfG^0 \subsetneq \cdots \subsetneq \bfG^{t-1} \subsetneq \bfG^t=\bfG\).
\end{itemize}

Let us first recall the definition of the quadratic character \(\epsilon^{\ram}\).
It is defined as the product of quadratic characters describing ``the contribution from the roots in \(\Phi(\bfS, \bfG^{i+1})-\Phi(\bfS, \bfG^i)\)" for \(0\leq i\leq t-1\).
We write \(\epsilon^{\ram, \bfG^i \subset \bfG^{i+1}}\) for each quadratic character.
Thus, we will eventually define
\begin{equation} \label{eq:epsilon^ram-as-prod}
\epsilon^{\ram}:=\prod_{i=0}^{t-1} \epsilon^{\ram, \bfG^i \subset \bfG^{i+1}}.
\end{equation}

Now let us fix \(0\leq i\leq t-1\) and recall the definition of \(\epsilon^{\ram, \bfG^i \subset \bfG^{i+1}}\).
For this, we need some notation.

Since the Bruhat--Tits building behaves well with respect to tamely ramified extensions (\cite{MR1871292}),
we can regard the Bruhat--Tits building \(\calB(\bfS, F)\) of \(\bfS\) over \(F\) as a subset of that \(\calB(\bfG, F)\) of \(\bfG\).
As \(\bfS\) is elliptic, the image of \(\calB(\bfS, F)\) in the reduced building \(\calB_{\red}(\bfG, F)\) consists of only one point, which we write as \(x\).

For \(\alpha \in \Phi(\bfS, \bfG)\) and \(r\in \bbR\), we write
\[
\frakg_{\alpha}(F_{\alpha})_{x, r}:=\frakg_{\alpha}(F_{\alpha}) \cap \frakg(F_{\alpha})_{x, r},
\]
where \(\frakg_{\alpha}\) is the root subspace for \(\alpha\) and \(\frakg(F_{\alpha})_{x, r}\) is the \(r\)-th Moy--Prasad filtration of \(\frakg(F_{\alpha})\).
We define
\[
\ord_x(\alpha):=\{r \in \bbR \mid \frakg_{\alpha}(F_{\alpha})_{x, r}\supsetneq \frakg_{\alpha}(F_{\alpha})_{x, r+}\}.
\]

For \(\gamma \in \bfS(F)\), we put
\begin{align*}
\Phi_{\gamma}&:=\{ \alpha \in \Phi(\bfS, \bfG^{i+1})-\Phi(\bfS, \bfG^i) \mid \alpha (\gamma) \neq 1\}, \\
\Phi_{\frac{r_i}{2}}&:=\Big\{ \alpha \in \Phi_{\gamma} \mid \frac{r_i}{2} \in \ord_x(\alpha)\Big\}.
\end{align*}
Note that strictly speaking the symbol \(\Phi_{\frac{r_i}{2}}\) here is in conflict with \(\Phi_r\) in Section \ref{sec:Howe-factorization},
but we will not use the latter symbol any more and thus this should cause no confusion.

The character \(\epsilon^{\ram, \bfG^i \subset \bfG^{i+1}}\) is defined by the following formula:
\begin{equation} \label{eq:epsilon^ram,i}
\epsilon^{\ram, \bfG^i \subset \bfG^{i+1}}(\gamma):=
\prod_{\alpha \in \Gamma_F\times \{\pm 1\}\backslash (\Phi_{\frac{r_i}{2}})^{\sym}}
\Jac{\overline{\alpha (\gamma)}}{k_{F_{\alpha}}^{\times}}
\cdot
\prod_{\alpha \in \Gamma_F\backslash (\Phi_{\frac{r_i}{2}})_{\sym, \ur}}
\Jac{\overline{\alpha (\gamma)}}{k_{F_{\alpha}}^{1}}
\end{equation}
for \(\gamma \in \bfS(F)\),
where we write
\[
(\Phi_{\frac{r_i}{2}})^{\sym}:=\Phi_{\frac{r_i}{2}} \cap \Phi(\bfS, \bfG)^{\sym}, 
\quad
(\Phi_{\frac{r_i}{2}})_{\sym, \ur}:=\Phi_{\frac{r_i}{2}} \cap \Phi(\bfS, \bfG)_{\sym, \ur},
\]
and \(\Jac{\cdot}{k_{F_{\alpha}}^{\times}}\) (resp.\ \(\Jac{\cdot}{k_{F_{\alpha}}^1}\)) is the unique character of order two of the cyclic group 
\(k_{F_{\alpha}}^{\times}\) (resp.\ \(k_{F_{\alpha}}^1:=\Ker (\Nr_{k_{F_{\alpha}}/k_{F_{\pm \alpha}}} \colon k_{F_{\alpha}}^{\times} \to k_{F_{\pm \alpha}}^{\times})\)) of even order
(recall that \(p \neq 2\) by our assumption).
As explained after (4.3.3) of \cite{MR4013740} (see also \cite[page 2024]{MR4206603}), in the first product, \(\alpha (\gamma)\) lies in \(\calO_{F_{\alpha}}^{\times}\) and we take its image \(\overline{\alpha (\gamma)}\) in \(k_{F_{\alpha}}^{\times}\).
Similarly in the second product, the \(F_{\alpha}/F_{\pm \alpha}\)-norm of \(\alpha (\gamma) \in \calO_{F_{\alpha}}^{\times}\) is trivial
and we take its image \(\overline{\alpha (\gamma)}\) in \(k_{F_{\alpha}}^1\).

With \(\epsilon^{\ram, \bfG^i \subset \bfG^{i+1}}\) at hand,
we define \(\epsilon^{\ram}\) by \eqref{eq:epsilon^ram-as-prod}.

Let us turn our attention to \(\epsilon_{f, \ram}\).
The definition involves an explicit set of zeta-data \(\zeta_{\ram}=\{\zeta_{\ram, \alpha}\}_{\alpha \in \Phi(\bfS, \bfG)}\) on \cite[page 1132]{MR4013740} (denoted by \(\{\epsilon_{\alpha}\}_{\alpha \in \Phi(\bfS, \bfG)}\) in loc.\ cit.), which we now recall.
For an asymmetric or symmetric unramified root \(\alpha\), we set \(\zeta_{\ram, \alpha}\) to be the trivial character of \(F_{\alpha}^{\times}\).
For a symmetric ramified root \(\alpha\),
we set
\[
\zeta_{\ram, \alpha}(x):=f_{(\bfG, \bfS)}(\alpha)^{v_{F_{\alpha}}(x)}
\]
for \(x\in F_{\alpha}^{\times}\).
Here \(f_{(\bfG, \bfS)}(\alpha) \in \{\pm 1\}\) is the \textit{toral invariant} defined in \cite[\S 4.1]{MR3402796}.
It is a sign associated to a symmetric root \(\alpha \in \Phi(\bfS, \bfG)_{\sym}\) and a pair \((\bfS, \bfG)\), and the dependence on \(\alpha\) is only through its \(\Gamma_F\)-orbit.
We do not recall the definition of the toral invariant, but will compute an equality involving it using a formula in \cite[Proposition 4.3.1]{MR3402796} later in Proposition \ref{prop:toral-invariant}.

The character \(\epsilon_{f, \ram}\) is obtained by applying to \(\zeta_{\ram}\) a general procedure producing a character of \(\bfS(F)\) from a set of zeta-data for \(\Phi(\bfS, \bfG)\).
It is essentially contained in \cite[Lemma 3.3.A, Lemma 3.3.D]{MR909227} and summarized in \cite[Definition 4.6.5]{MR4013740}.
Here we only recall the resulting character \(\epsilon_{f, \ram}\).
For a symmetric ramified root \(\alpha \in \Phi(\bfS, \bfG)_{\sym, \ram}\), we write \(\zeta_{\Gamma_F \cdot \alpha}\) for the composite
\begin{equation} \label{eq:zeta-to-character}
\bfS(F) \xrightarrow{\alpha} F_{\alpha}^1 \xleftarrow{\sim} F_{\alpha}^{\times}/F_{\pm \alpha}^{\times} \xrightarrow{\zeta_{\ram, \alpha}} \bbC^{\times},
\end{equation} 
where \(\alpha\) maps \(\bfS(F)\) into \(F_{\alpha}^1:=\Ker(\Nr_{F_{\alpha}/F_{\pm \alpha}} \colon F_{\alpha}^{\times} \to F_{\pm \alpha}^{\times})\) since \(\alpha\) is symmetric
and the middle isomorphism sends \(\gamma \in F_{\alpha}^{\times}\) to \(\frac{\gamma}{\tau(\gamma)} \in F_{\alpha}^1\) with \(\tau \in \Gal(F_{\alpha}/F_{\pm \alpha})\) being the generator.
As the notation suggests, \(\zeta_{\Gamma_F \cdot \alpha}\) depends only on the \(\Gamma_F\)-orbit.
Then we define
\begin{equation} \label{eq:epsilon_ram}
\epsilon_{f, \ram}:=\prod_{\alpha \in \Gamma_F\backslash \Phi(\bfS, \bfG)_{\sym, \ram}} \zeta_{\Gamma_F \cdot \alpha}.
\end{equation} 

In view of \eqref{eq:epsilon^ram-as-prod}, \eqref{eq:epsilon^ram,i}, \eqref{eq:epsilon_ram}, we define the quadratic characters \(\epsilon_{\alpha}\) for each \(\alpha \in \Phi(\bfS, \bfG)\) as follows.
First suppose that \(\alpha\) is asymmetric or symmetric unramified, and let \(-1\leq i\leq t-1\) be the unique integer such that \(\alpha \in \Phi(\bfS, \bfG^{i+1})-\Phi(\bfS, \bfG^i)\).
We define
\begin{equation} \label{eq:epsilon-asym}
\epsilon_{\alpha}(\gamma):=
\begin{cases}
\Jac{\overline{\alpha(\gamma)}}{k_{F_{\alpha}}^{\times}} &\text{if \(i\neq -1\) and \(\frac{r_i}{2} \in \ord_x(\alpha)\)} \\
1 & \text{otherwise}
\end{cases}
\end{equation}
for \(\alpha \in \Phi(\bfS, \bfG)^{\sym}\) and
\begin{equation} \label{eq:epsilon-symur}
\epsilon_{\alpha}(\gamma):=
\begin{cases}
\Jac{\overline{\alpha(\gamma)}}{k_{F_{\alpha}}^1} &\text{if \(i\neq -1\) and \(\frac{r_i}{2} \in \ord_x(\alpha)\)} \\
1 & \text{otherwise}
\end{cases}
\end{equation}
for \(\alpha \in \Phi(\bfS, \bfG)_{\sym, \ur}\).
Next suppose that \(\alpha\) is symmetric ramified. Then we define
\begin{equation} \label{eq:epsilon-symram}
\epsilon_{\alpha}:=\zeta_{\Gamma_F \cdot \alpha}.
\end{equation}
By definition, we have
\begin{align}
\label{eq:factorization-epsilon^ram} \epsilon^{\ram}&=\prod_{\alpha \in \Gamma_F\times \{ \pm 1\}\backslash \Phi(\bfS, \bfG)^{\sym}} \epsilon_{\alpha} \cdot
\prod_{\alpha \in \Gamma_F\backslash \Phi(\bfS, \bfG)_{\sym, \ur}} \epsilon_{\alpha}, \\
\label{eq:factorization-epsilon_fram} \epsilon_{f, \ram}&=
\prod_{\alpha \in \Gamma_F\backslash \Phi(\bfS, \bfG)_{\sym, \ram}} \epsilon_{\alpha}.
\end{align}

Now we specialize to the case where \(\bfG^{\ast}=\GL_n\) and \(\bfG=\underline{A}^{\times}\) are as in Section \ref{sec:review-of-BH} and introduce some notation adapted to the situation.
In this case, 
tame elliptic regular pairs of \(\bfG^{\ast}\) and \(\bfG\) are naturally identified (see Section \ref{sec:ter-and-adm-pairs}) and thus 
we wrote the associated characters \(\epsilon=\epsilon_{f, \ram} \cdot \epsilon^{\ram}\) as \(\eps{\bfG^{\ast}}{}\) and \(\eps{\bfG}{}\) in Section \ref{sec:LLC^Kal-for-GL_n}.
Similarly, for the characters \(\epsilon_{\alpha}, \epsilon^{\ram}, \epsilon_{f, \ram}\), we add \(\bfG\) or \(\bfG^{\ast}\) as a subscript on the left to indicate with respect to which group the characters are formed.
While we have \(\Phi(\bfS, \bfG^{\ast})=\Phi(\bfS, \bfG)\) (see Remark \ref{rem:GL_n-and-ell-tori} \eqref{item:roots-for-inner-form-of-GL_n}),
the set \(\ord_{x}(\alpha) \subset \bbR\) and the toral invariant depend on the group, and hence so do the various characters.
Thus, we write \(\ordx{\bfG}{\alpha}\) (resp.\ \(\ordx{\bfG^{\ast}}{\alpha}\)) for the set \(\ord_{x}(\alpha)\) defined with respect to \(\bfG\) (resp.\ \(\bfG^{\ast}\)).
For the toral invariant, we continue to use the same symbols \(f_{(\bfG^{\ast}, \bfS)}(\alpha), f_{(\bfG, \bfS)}(\alpha)\) as they already indicate the dependence on the group.

\subsection{Roots of an elliptic maximal torus of \(\GL_n\)}\label{sec:sigma-phi}
We continue to let \(\bfG^{\ast}=\GL_n\) and \(\bfG=\underline{A}^{\times}\).
Let \(\bfS \subset \bfG^{\ast}\) be a tamely ramified elliptic maximal torus.
Then \(\bfS \simeq \Res_{E/F}\Gm\) for a tamely ramified extension \(E/F\) of degree \(n\).
Let \(e=e(E/F)\) and \(f=f(E/F)\) be the ramification index and the residue degree of \(E/F\).
Recall the bijection \eqref{eq:Galois-orb-as-coset}
\[
(\Gamma_E\backslash \Gamma_F/\Gamma_E)' \xrightarrow{1:1}\Gamma_F\backslash \Phi(\bfS, \bfG^{\ast});
\quad 
\Gamma_Eg\Gamma_E \mapsto 
\Gamma_F
\begin{bmatrix}
1 \\
g
\end{bmatrix}.
\]
In this section, we review some of the contents in \cite[\S 3.2]{MR3509939} and \cite[\S 5.1]{MR4206603} (also reviewed in \cite[\S 2.2]{MR3505131}),
where one introduces convenient representatives of \(\Gamma_F/\Gamma_E\) 
to be able to discuss roots and their orbits more explicitly.

It is well-known that one can take uniformizers \(\varpi_E \in E\)
and \(\varpi_F \in F\) such that 
\[
\varpi_E^e=z_{E/F} \varpi_F
\]
for a root \(z_{E/F} \in \mu_E\) of unity.
We also fix an \(e\)-th root \(z_{E/F, e}\) of \(z_{E/F}\) 
and a primitive \(e\)-th root \(z_e\) of unity,
and put \(z_{\phi}=z_{E/F, e}^{q-1}\).
Then we set
\[
L:=E[z_e, z_{E/F, e}].
\]
As is easily seen, \(L\) is a Galois extension of \(F\) containing \(E\)
and is unramified over \(E\).
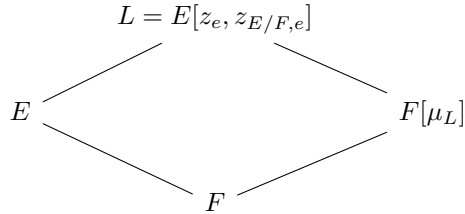
\[
\begin{tikzcd}
& L=E[z_e, z_{E/F, e}] & \\
E \arrow[dash]{ur} & & F[\mu_L] \arrow[dash]{ul} \\
& F \arrow[dash]{ul} \arrow[dash]{ur} &
\end{tikzcd}
\]
The Galois group \(\Gamma_{L/F}\) of \(L/F\) is a semi-direct product
\(\Gamma_{L/F}=\lan \sigma \ran \rtimes \lan \phi \ran\), where
\begin{align*}
&\sigma \colon \varpi_E \mapsto z_e \varpi_E, \quad z \mapsto z \quad \text{(for \(z \in \mu_L\)),} \\
&\phi \colon \varpi_E \mapsto z_{\phi} \varpi_E, \quad z\mapsto z^q \quad \text{(for \(z \in \mu_L\))}.
\end{align*}
As discussed in \cite[Proposition 3.3 (i)]{MR3509939},
the set
\[
\{ \sigma^i \phi^j \mid 0\leq i\leq e-1, 0\leq j\leq f-1\}
\]
forms a set of representatives of \(\Gamma_F/\Gamma_E\)
under the natural identification \(\Gamma_F/\Gamma_E \xrightarrow{1:1} \Gamma_{L/F}/\Gamma_{L/E}\),
Thus, when we deal with roots \(\alpha \in \Phi(\bfS, \bfG^{\ast})\) or its orbit using the bijection \eqref{eq:Galois-orb-as-coset}, 
we often take a representative \(g\) in the above set
(here we similarly identify \(\Gamma_E\backslash \Gamma_F/\Gamma_E\) with  \(\Gamma_{L/E}\backslash \Gamma_{L/F}/\Gamma_{L/E}\)).

For \(g\in \Gamma_F\), we write \([g]\) for the double coset \(\Gamma_E g \Gamma_E\).
We say that a double coset \([g]\in (\Gamma_E\backslash \Gamma_F/\Gamma_E)'\) is \emph{asymmetric} or \emph{symmetric}
if the corresponding orbit of roots has the same property.
In other words, \([g]\) is asymmetric (resp.\ symmetric)
if and only if \([g] \neq [g^{-1}]\) (resp.\ \([g]=[g^{-1}]\)).
Similarly, we also say that a symmetric double coset is \emph{unramified} or \emph{ramified} according to the property of the corresponding orbit of roots.

Let \((\Gamma_E \backslash \Gamma_F/\Gamma_E)_{\asym} \subset	(\Gamma_E \backslash \Gamma_F/\Gamma_E)'\) 
(resp.\ \((\Gamma_E \backslash \Gamma_F/\Gamma_E)_{\sym} \subset (\Gamma_E \backslash \Gamma_F/\Gamma_E)'\))
be the set of non-trivial asymmetric (resp.\ symmetric) double cosets.
Also, let \((\Gamma_E \backslash \Gamma_F/\Gamma_E)_{\asym /{\pm}}\) be the quotient set of \((\Gamma_E \backslash \Gamma_F/\Gamma_E)_{\asym}\) defined by identifying \([g]\) with \([g^{-1}]\) for every \([g]\in (\Gamma_E \backslash \Gamma_F/\Gamma_E)_{\asym}\).

\begin{prop} \label{prop:cond-by-Galois-elt}
	Let the notation be as above.
	\begin{enumerate}
		\item \label{item:cond-for-sym} If \(\alpha=\begin{bmatrix}1 \\ \sigma^i \phi^j
		\end{bmatrix} \in \Phi(\bfS, \bfG^{\ast})\) 
		\((0\leq i\leq e-1, 0\leq j\leq f-1)\) 
		is symmetric, 
		then \(j=0\) or \(j=\frac{f}{2}\) (in which case \(f\) is even).
		\item \label{item:cond-for-symram} The set \(\Phi(\bfS, \bfG^{\ast})\) has a symmetric ramified root if and only if \(e\) is even.
		In this case, there exists a unique such orbit and it is represented by  \(\begin{bmatrix}1 \\ \sigma^{\frac{e}{2}}\end{bmatrix}\).
	\end{enumerate}
\end{prop}
\begin{proof}
The assertion \eqref{item:cond-for-sym} is \cite[Proposition 3.3 (iii)]{MR3509939} (or \cite[Proposition 2.2]{MR3505131}).
The assertion \eqref{item:cond-for-symram} is \cite[Proposition 5.3]{MR4206603}, proven by slightly refining the arguments in \cite[\S 3.2]{MR3509939}.
\end{proof}
\begin{rem} \label{rem:Tam-sym-ur/ram}
The above definition of symmetric unramified and ramified cosets comes from the same property of the roots defined by Adler--Spice
\cite{MR2543925} and recalled in the beginning of Section \ref{sec:chi-data-L-emb}.
In fact, in \cite{MR3509939}, \cite{MR3505131}
Tam uses the same terms, but in a different way.
The below table summarizes the difference (here we take \(0\leq i\leq e-1\)). 
In this paper we usually use the terms in the sense of Adler--Spice, except for some parts where we explicitly state so.
\begin{table}[h!]
\centering
\begin{tabular}{|c|c|c|}
\hline 
& Adler--Spice's sense & Tam's sense \\ 
\hline 
symmetric coset \([\sigma^i \phi^{\frac{f}{2}}]\) & \text{unramified} &  \text{unramified} \\ 
\hline 
symmetric coset \([\sigma^i]\) (\(i\neq \frac{e}{2}\)) & \text{unramified} & \text{ramified} \\ 
\hline 
symmetric coset \([\sigma^\frac{e}{2}]\) & \text{ramified} & \text{ramified} \\ 
\hline 
\end{tabular}
\end{table}
\end{rem}

\subsection{Toral invariants} \label{sec:toral-invariant}
Let \(\bfG^{\ast}=\GL_n, \bfG=\underline{A}^{\times}, \bfS =\Res_{E/F}\Gm \subset \bfG^{\ast}\) be as in the previous section.
Let us prove an equality involving the toral invariant \(f_{(\bfG, \bfS)}(\alpha)\).

In the following, we work with the Brauer group \(\Br(F)\) (see Section \ref{sec:notation}).
Recall that \(A=M_m(D)\).
We write the Hasse invariant of \(A\) as \(\frac{h}{d}\) with \(h, d=\frac{n}{m} \in \bbZ\) prime to each other, namely
\[
\inv_F([A])=\frac{h}{d}
\]
in \(\bbQ/\bbZ\).
We also use the isomorphism
\[
\bfe \colon \bbQ/\bbZ \xrightarrow{\sim} \mu (\bbC):=\{\text{roots of unity in \(\bbC\)}\}; \quad x \mapsto \exp (2\pi \sqrt{-1} x)
\]
to pass from the additive notation to the multiplicative notation.
\begin{prop} \label{prop:toral-invariant}
With the above notation, we have the following.
\begin{enumerate}
\item \label{item:sign-inv} 
Let \(\alpha \in \Phi(\bfS, \bfG^{\ast})\) be a symmetric root
and put \(n_{\alpha}:=[F_{\pm \alpha}:F]\).
Let \([A]\in \Br(F)\) denote the class of \(A\).
Then \(n_{\alpha} \cdot [A]\) lies in the 2-torsion subgroup of \(\Br(F)\) and hence can be regarded as a sign: \(\bfe (\inv_F (n_{\alpha} \cdot [A]))\in \{\pm 1\}\).
We have
\[
f_{(\bfG^{\ast}, \bfS)}(\alpha) \cdot f_{(\bfG, \bfS)}(\alpha)=\bfe (\inv_F (n_{\alpha} \cdot [A]))
\]
in \(\{ \pm 1\}\).
\item \label{item:sign-symram} 
For a symmetric ramified root \(\alpha \in \Phi(\bfS, \bfG^{\ast})_{\sym, \ram}\), we have
\[
(\eps{\bfG^{\ast}}{\alpha} \cdot \eps{\bfG}{\alpha})(\gamma)=(-1)^{mv_E(\gamma)}
\]
for \(\gamma \in \bfS(F)=E^{\times}\).
\end{enumerate}
\end{prop}
\begin{rem}
In fact, we have \(f_{(\bfG^{\ast}, \bfS)}(\alpha)=1\) by \cite[Proposition 4.4]{MR4206603} and thus \(\eps{\bfG^{\ast}}{\alpha}=\mathbbm{1}\) .
However, we do not need this fact in what follows.
\end{rem}
\begin{proof}
Let us prove \eqref{item:sign-inv}.
As \(F_{\alpha}\) contains a conjugate of \(E\) and thus \(2n_{\alpha}=[F_{\alpha}:F]\) is divisible by \(n\), we see that 
\(2n_{\alpha}\cdot [A]\) is trivial as required.
	
To show the equality of signs we use \cite[Proposition 4.3.1]{MR3402796}, which we now recall.
Let \(\bfS_{\ad}\) be the image of \(\bfS \subset \bfG^{\ast}\) in the adjoint quotient \(\bfG^{\ast}_{\ad}\).
As in Remark \ref{rem:GL_n-and-ell-tori} \eqref{item:roots-for-inner-form-of-GL_n}, 
we may assume, upon replacing by an isomorphic twist, that the restriction of the canonical inner twist \(\psi \colon \bfG^{\ast}\to \bfG\) to \(\bfS\) is defined over \(F\).
Then the map \((\Gamma_F \ni \tau \mapsto \psi^{-1} \circ \tau (\psi))\) defines a 1-cocycle valued in \(\bfS_{\ad}\) and hence a cohomology class \(t\in H^1(F, \bfS_{\ad})\).
Let \(\bfS_{\alpha}\) be the one-dimensional anisotropic torus over \(F_{\pm \alpha}\) that is split over \(F_{\alpha}\),
	namely
	\[
	\bfS_{\alpha}=\Ker (\Nr_{F_{\alpha}/F_{\pm \alpha}}\colon \Res_{F_{\alpha}/F_{\pm \alpha}}\Gm \to \Gm)
	\simeq \Coker (\Gm \to \Res_{F_{\alpha}/F_{\pm \alpha}}\Gm).
	\]
	Then since it is symmetric,
	the root \(\alpha \colon \bfS_{\ad}\to \Gm\) defined over \(F_{\alpha}\) descends to 
	a homomorphism \(\bfS_{\ad}\to \bfS_{\alpha}\) over \(F_{\pm \alpha}\), which we continue to write as \(\alpha\).
Also, we have a canonical isomorphism
\(\kappa_{\alpha}^{\coh}\colon H^1(F_{\pm \alpha}, \bfS_{\alpha})\xrightarrow{\sim} \{ \pm 1\}\).
Let \(\eta_{t, \alpha}\) be the image of \(t \in H^1(F, \bfS_{\ad})\) under
\[
H^1(F, \bfS_{\ad}) \xrightarrow{\Res} H^1(F_{\pm \alpha}, \bfS_{\ad}) \xrightarrow{\alpha} H^1(F_{\pm \alpha}, \bfS_{\alpha}).
\]
Now \cite[Proposition 4.3.1]{MR3402796}
asserts that
\[
f_{(\bfG^{\ast}, \bfS)}(\alpha) \cdot f_{(\bfG, \bfS)}(\alpha)=	\kappa_{\alpha}^{\coh}(\eta_{t, \alpha}).
\]

Thus we compute \(\eta_{t, \alpha}\) in terms of the Brauer group.
For this, we use the exact sequence 
	\[
	1 \to \Gm \to \Res_{E/F} \Gm \to \bfS_{\ad} \to 1
	\]
	over \(F\) expressing \(\bfS_{\ad}\) as the quotient
	and observe that the base change of the above sequence to \(F_{\pm \alpha}\) fits into the following commutative diagram of the exact sequences over \(F_{\pm \alpha}\):
	\[
	\xymatrix{
		1 \ar[r] & \Gm \ar[r] \ar[d]^{\id} & (\Res_{E/F} \Gm)_{F_{\pm \alpha}} \ar[r] \ar[d]^{\tilde{\alpha}} & \bfS_{\ad} \ar[r] \ar[d]^{\alpha} & 1 \\
		1 \ar[r] & \Gm \ar[r] & \Res_{F_{\alpha}/F_{\pm \alpha}}\Gm \ar[r] & \bfS_{\alpha} \ar[r] & 1. 
	}
	\]
	Here,
	\begin{itemize}
		\item the term \((\Res_{E/F} \Gm)_{F_{\pm \alpha}}\) is the base change of 
		\(\Res_{E/F} \Gm\) to \(F_{\pm \alpha}\),
		\item the two horizontal homomorphisms from \(\Gm\) are both the natural inclusions to the targets,
		\item if \(\alpha=\begin{bmatrix}
		g_1 \\
		g_2
		\end{bmatrix}\) in the notation of Section \ref{sec:L-emb-and-ind},
		then \(\tilde{\alpha}\) is induced by
		\begin{align*}
		(\Res_{E/F} \Gm)_{F_{\pm \alpha}}(F_{\pm \alpha})=(E\otimes_{F}F_{\pm \alpha})^{\times} &\to F_{\alpha}^{\times}=\Res_{F_{\alpha}/F_{\pm \alpha}} \Gm(F_{\pm \alpha}) \\
		x\otimes a&\mapsto g_1(x)a,
		\end{align*}
		and
		\item the homomorphism \(\Res_{F_{\alpha}/F_{\pm \alpha}}\Gm \to \bfS_{\alpha}\) is induced by
		\begin{align*}
		\Res_{F_{\alpha}/F_{\pm \alpha}}\Gm(F_{\pm \alpha})=F_{\alpha}^{\times} &\to \Ker (\Nr_{F_{\alpha}/F_{\pm \alpha}}\colon F_{\alpha}^{\times} \to F_{\pm \alpha}^{\times})=\bfS_{\alpha}(F_{\pm \alpha}) \\
		x&\mapsto \frac{x}{\tau(x)},
		\end{align*}
		where \(\tau\) is the generator of \(\Gal(F_{\alpha}/F_{\pm \alpha})\).
	\end{itemize}
	These exact sequences induce the following commutative diagram with exact rows:
	\[
	\xymatrix{
		H^1(F, \Res_{E/F}\Gm) \ar[r] \ar[d] & H^1(F, \bfS_{\ad}) \ar[r] \ar[d]^{\Res} & H^2(F, \Gm) \ar[d]^{\Res} \\
		H^1(F_{\pm \alpha}, (\Res_{E/F}\Gm)_{F_{\pm \alpha}}) \ar[r] \ar[d] & H^1(F_{\pm \alpha}, \bfS_{\ad}) \ar[r] \ar[d]^{\alpha} & H^2(F_{\pm \alpha}, \Gm) \ar[d]^{\id} \\
		H^1(F_{\pm \alpha}, \Res_{F_{\alpha}/F_{\pm \alpha}}\Gm) \ar[r] & H^1(F_{\pm \alpha}, \bfS_{\alpha}) \ar[r] & H^2(F_{\pm \alpha}, \Gm).
	}
	\]
	Recall that \(\eta_{t, \alpha}\) is the image of \(t \in H^1(F, \bfS_{\ad})\) by the middle vertical arrows.
	As the lower left term \(H^1(F_{\pm \alpha}, \Res_{F_{\alpha}/F_{\pm \alpha}}\Gm)\) is trivial by Shapiro's lemma and Hilbert's theorem 90, we may consider the image \([\eta_{t, \alpha}]\) of \(\eta_{t, \alpha}\) in the lower right term \(H^2(F_{\pm \alpha}, \Gm)\).
Now the desired equality 
\[
[\eta_{t, \alpha}]=\bfe (\inv_F (n_{\alpha} \cdot [A]))
\]
is clear because the image of \(t\) in \(H^2(F, \Gm)\) is \([A]\) 
(this is easily seen by considering the coboundary maps associated to the short exact sequences
\[
\xymatrix{
1 \ar[r] & \Gm \ar[r] \ar[d]^{\id} & \Res_{E/F} \Gm \ar[r] \ar[d] & \bfS_{\ad} \ar[r] \ar[d] & 1 \\
1 \ar[r] & \Gm \ar[r] & \bfG^{\ast} \ar[r] & \bfG^{\ast}_{\ad} \ar[r] & 1
}
\]
with obvious maps)
and the restriction map \(\Res \colon H^2(F, \Gm) \to H^2(F_{\pm \alpha}, \Gm)\) induces the multiplication by \(n_{\alpha}\) under the canonical identifications 
	\[
	\inv_F \colon H^2(F, \Gm)\simeq \bbQ/\bbZ, \quad \inv_{F_{\pm \alpha}} \colon H^2(F_{\pm \alpha}, \Gm)\simeq \bbQ/\bbZ
	\]
	(see, e.g., \cite[Chapter XIII, \S 3, Proposition 7]{MR554237}).
	
Let \(\alpha\) be a symmetric ramified root
and we now prove \eqref{item:sign-symram}.
Let us first show \(\bfe (\inv_F (n_{\alpha} \cdot [A]))=(-1)^m\)
and apply \eqref{item:sign-inv} to obtain
the following equality
\begin{equation} \label{eq:prod-of-toral-inv}
f_{(\bfG^{\ast}, \bfS)}(\alpha) \cdot f_{(\bfG, \bfS)}(\alpha)=(-1)^{m}.
\end{equation}
By Proposition \ref{prop:cond-by-Galois-elt} \eqref{item:cond-for-symram} we may assume 
that \(e\) is even and \(\alpha=\begin{bmatrix}1 \\\sigma^{\frac{e}{2}}\end{bmatrix}\).
In particular, \(n\) is even and we have \(F_{\alpha}=E\) and \(n_{\alpha}=[F_{\pm \alpha}: F]=\frac{n}{2}\).
Now we compute
\[
\bfe (\inv_F (n_{\alpha} \cdot [A]))=\exp\Big( 2\pi \sqrt{-1}\cdot \frac{n}{2} \cdot \frac{h}{d}\Big)
=\exp\Big( 2\pi \sqrt{-1}\cdot \frac{mh}{2} \Big)=(-1)^{mh}=(-1)^m
\]
in \(\{\pm 1\}\),
where we have \(m(h-1)\in 2\bbZ\) because if \(m\) is odd then \(d=\frac{n}{m}\) is even and so \(h\) has to be odd.
Thus, we have shown \eqref{eq:prod-of-toral-inv}.

Using the concrete expression for \(\alpha\) and \(F_{\alpha}\),
we see that the composite of the first two maps in \eqref{eq:zeta-to-character} amounts to the reduction map \(\bfS(F)=E^{\times}=F_{\alpha}^{\times} \to F_{\alpha}^{\times}/F_{\pm \alpha}^{\times}\).
Now by \eqref{eq:prod-of-toral-inv} we compute
\[
(\eps{\bfG^{\ast}}{\alpha} \cdot \eps{\bfG}{\alpha})(\gamma)
=(f_{(\bfG^{\ast}, \bfS)}(\alpha) \cdot f_{(\bfG, \bfS)}(\alpha))^{v_E(\gamma)}
=(-1)^{mv_E(\gamma)}
\]
as desired.
\end{proof}

\section{Preliminaries II} \label{sec:prelim2}
Throughout this section, we let \(\bfG^{\ast}=\GL_n=\underline{A}^{\ast, \times}\) and \(\bfG=\underline{A}^{\times}\) as in Section \ref{sec:review-of-BH}.
Recall that \(A^{\ast}=M_n(F)\) and \(A=M_m(D)\). 
The main goal of this section is to recall some ingredients necessary for Tam's \(\zeta\)-data \(\zeta_{\Tam}\).
We also record a result (Proposition \ref{prop:vmod-ord}) comparing the conventions of Kaletha and Tam.

\subsection{Order and finite modules associated to a maximal subfield} \label{sec:hered-ord}
In this section we summarize parts of \cite[\S 2.1-\S 2.3]{MR2848585}, \cite[\S 2.4, \S 2.5, \S 4.1]{MR3505131} and review some objects associated to a maximal subfield \(E\subset A\).
In particular, we introduce finite modules \(\frakU_{[g]}\) and recall \cite[Proposition 4.2]{MR3505131} (Proposition \ref{prop:decomp-of-grpiece}).

\subsubsection{Principal orders} 
Similarly to the notation for \(F\),
we write \(\calO_D, \frakp_D, k_D\) for the maximal order \(\calO_D \subset D\), the maximal ideal \(\frakp_D \subset \calO_D\) and the residue field \(k_D=\calO_D/\frakp_D\).
An \(\calO_F\)-order \(\frakA \subset A\) (that is, a subring of \(A\) that is also an \(\calO_F\)-lattice) is said to be \textit{principal} if its Jacobson radical \(\frakP_{\frakA}\) is generated by one element \(a\in \frakA\): \(\frakP_{\frakA}=\frakA a=a\frakA\) (\cite[\S 1.3]{MR772712}).
Concretely, principal \(\calO_F\)-orders in \(A\) are characterized as \(\calO_F\)-orders \(\frakA\) that are \(A^{\times}\)-conjugate to the order of block matrices:
\[
\begin{pmatrix}
(\calO_D) & (\calO_D) &\ldots & (\calO_D) \\
(\frakp_D) & (\calO_D) &\ldots & (\calO_D) \\
\vdots & \ddots     & \ddots  & \vdots \\
(\frakp_D) & \ldots & (\frakp_D) & (\calO_D)
\end{pmatrix},
\]
where, for some divisor \(s\) of \(m\), each \((\calO_D)\) (resp.\ \((\frakp_D)\)) indicates \(s \times s\)-matrices with entries in \(\calO_D \) (resp.\ \(\frakp_D\))
(see \cite[(1.2.15), (1.3.2) Theorem]{MR772712} or \cite[\S 0.2, \S 0.3]{MR2347425} for a statement close to this form).
In other words, the latter order is the inverse image by the natural map \(M_m(\calO_D)\to M_m(k_D)\)
of the subset in \(M_m(k_D)\) of block upper-triangular matrices with respect to the partition \((s, \cdots, s)\) of \(m\).
In the same notation, the Jacobson radical of the latter order is 
\[
\begin{pmatrix}
(\frakp_D) & (\calO_D) &\ldots & (\calO_D) \\
(\frakp_D) & (\frakp_D) &\ldots & (\calO_D) \\
\vdots & \ddots     & \ddots  & \vdots \\
(\frakp_D) & \ldots & (\frakp_D) & (\frakp_D)
\end{pmatrix}.
\]
In the literature, hereditary \(\calO_F\)-orders (e.g., \cite[\S 1.2]{MR772712}), which are more general than principal \(\calO_F\)-orders, are also often considered. However, as all the orders that we deal with in this paper are principal, we do not discuss them here.

Let \(\frakA\) be a principal \(\calO_F\)-order as above and put \(r=\frac{m}{s}\). Then we have
\[
\frakA/\frakP_{\frakA} \simeq M_s(k_D)^r.
\]
We write these invariants of \(\frakA\) as \(s(\frakA)=s, r(\frakA)=r\).
We define the following:
\begin{align*}
&\frakP_{\frakA}^{x}:=\frakP_{\frakA}^{\lceil x\rceil}, 
&\frakP_{\frakA}^{x+}:=\bigcup_{y>x} \frakP_{\frakA}^{y} 
&\quad \text{ for \(x \in \bbR\)} \\
&U_{\frakA}:=U_{\frakA}^0:=\frakA^{\times},
&U_{\frakA}^i:=1+\frakP_{\frakA}^i 
&\quad \text{ for \(i \in \bbZ_{>0}\)}\\
&U_{\frakA}^x:=U_{\frakA}^{\lceil x\rceil},
&U_{\frakA}^{x+}:=\bigcup_{y>x} U_{\frakA}^{y} 
&\quad \text{ for \(x \in \bbR_{\geq 0}\)}. 
\end{align*}
Clearly, \(U_{\frakA}, U_{\frakA}^i, U_{\frakA}^x, U_{\frakA}^{x+}\) are compact open subgroups of \(A^{\times}\).
We also write
\[
\frakK_{\frakA}:=\{ g\in A \mid g\frakA g^{-1}=\frakA\}
\]
for the normalizer of \(\frakA\) in \(A^{\times}\).
Then \(\frakK_{\frakA}\) fits into an exact sequence
\[
1 \to U_{\frakA} \to \frakK_{\frakA} \xrightarrow{v_{\frakA}} \bbZ \to 0,
\]
where \(v_{\frakA}\) is defined by the condition \(g\frakA=\frakA g=\frakP_{\frakA}^{v_{\frakA}(g)}\) for \(g\in \frakK_{\frakA}\) (\cite[page 484]{MR2848585}),
and 
is in fact the semi-direct product of \(U_{\frakA}\) and the infinite cyclic group \(\langle a\rangle\):
\[
\frakK_{\frakA}=U_{\frakA} \rtimes \langle a\rangle,
\] 
where \(a\in \frakA\) is a generator of \(\frakP_{\frakA}\) as before (see, e.g., \cite[page 52]{MR1695699}).

We define \(e(\frakA/\calO_F) \in \bbZ_{>0}\) by \(v_{\frakA}|_{F^{\times}}=e(\frakA/\calO_F) \cdot v_F\), or equivalently by \(\frakp_F \frakA=\frakP_{\frakA}^{e(\frakA/\calO_F)}\).
Then we have 
\[
e(\frakA/\calO_F)=dr(\frakA).
\]

\subsubsection{Principal order associated to a maximal subfield}
Now let \(E/F\) be a separable extension of degree \(n\).
We put \(e=e(E/F)\) and \(f=f(E/F)\).
Take an \(F\)-embedding \(E\hookrightarrow A\), which is unique up to \(A^{\times}\)-conjugacy.
Then 
there exists a unique principal \(\calO_F\)-order \(\frakA\) in \(A\)
such that \(E^{\times} \subset \frakK_{\frakA}\)
(see, e.g., \cite [0.\ Theorem]{MR1695699} or \cite[(2.3.1) (2), (2.3.2)]{MR2848585}).

Before recalling the expression of \(r(\frakA)\) and \(s(\frakA)\) in terms of \(e\) and \(f\),
let us define some auxiliary integers \(m_0, f_0, d_0, e_0\) by the following equalities:
\begin{align*}
 m&=m_0 \cdot \gcd(m, f), & f&=f_0 \cdot \gcd(m, f), \\
 d&=d_0 \cdot \gcd(d, e), & e&=e_0 \cdot \gcd(d, e).
\end{align*}
As \(md=n=ef\), one can readily check that \(m_0=e_0\) and \(f_0=d_0\).
According to the explanation after (2.6) of \cite{MR3505131} (or the first sentence of Section 4.1 of loc.\ cit.),
we have
\begin{equation} \label{eq:r(frakA)}
r(\frakA)=\frac{e}{\gcd(d, e)}=e_0=m_0, 
\quad 
s(\frakA)=\gcd(f, m)=\frac{m}{m_0}=\frac{m}{r(\frakA)}.
\end{equation}
As before, we define \(e(\frakA/\calO_E) \in \bbZ_{>0}\) by \(v_{\frakA}|_{E^{\times}}=e(\frakA/\calO_E) \cdot v_E\).
Then we see that
\begin{equation} \label{eq:e(frakA/calO_E)}
e(\frakA/\calO_E)=\frac{e(\frakA/\calO_F)}{e(E/F)}=\frac{dr(\frakA)}{e}
=\frac{de_0}{e}=d_0=f_0.
\end{equation}

We also take an \(F\)-embedding \(E \hookrightarrow A^{\ast}\)
and write \(\frakA^{\ast} \subset A^{\ast}\) for the induced principal \(\calO_F\)-order.

\subsubsection{Finite modules associated to a maximal subfield}
Assume moreover that \(E/F\) is a tamely ramified extension.
As the conjugation action of \(E^{\times}\) on \(A\) preserves 
\(\frakA\) and more generally \(\frakP^{j}\),
it induces an action on the \(k_E\)-vector space \(\frakP^{j}/\frakP^{j+1}\) for any \(j \in \bbZ\).
This action factors through a finite group \(\Psi_{E/F}:=E^{\times}/F^{\times}U_E^1\) of order prime to \(p\).
We review results describing the \(k_F[\Psi_{E/F}]\)-module structures of \(\frakP^{j}/\frakP^{j+1}\).

We first treat the special case of \(\frakU_{\frakA^{\ast}}:=\frakA^{\ast}/\frakP_{\frakA^{\ast}}\)
to introduce the modules \(\frakU_{[g]}\).

Consider the following two isomorphisms
\begin{align*}
E\otimes_F E \xrightarrow{\sim} \End_{F}(E); \quad & x\otimes y \mapsto (v\mapsto \tr_{E/F}(yv)\cdot x), \\
E\otimes _F E \xrightarrow{\sim} \bigoplus_{[g]\in \Gamma_E \backslash \Gamma_F/\Gamma_E} E\cdot g(E); \quad & x\otimes y \mapsto (x\cdot g(y))_{[g]}
\end{align*}
of \(F\)-vector spaces.
Here, \(E\cdot g(E)\) denotes the composite field.
By composition, they induce a standard isomorphism \(\End_{F}(E) \simeq \bigoplus_{[g]\in \Gamma_E \backslash \Gamma_F/\Gamma_E} E\cdot g(E)\).
Moreover, once we fix an \(F\)-isomorphism \(E \simeq F^{n}\),
we obtain an \(F\)-isomorphism \(A^{\ast}=M_n(F) \simeq \End_{F}(E)\) and thus
\begin{equation} \label{eq:rel-root-decomp}
A^{\ast} \simeq \bigoplus_{[g]\in \Gamma_E \backslash \Gamma_F/\Gamma_E} E\cdot g(E).
\end{equation}
We arrange the isomorphism \(E \simeq F^{n}\) so that the composite \(E\hookrightarrow \End_{F}(E) \simeq A^{\ast}\) agrees with the already fixed embedding.
Through this isomorphism \eqref{eq:rel-root-decomp}, the conjugation action of \(E^{\times}\) on \(A^{\ast}\) induces an action on \(\bigoplus_{[g]\in \Gamma_E \backslash \Gamma_F/\Gamma_E} E\cdot g(E)\).
Explicitly, \(t \in E^{\times}\) acts as
\[
\bigoplus_{[g]\in \Gamma_E \backslash \Gamma_F/\Gamma_E} E\cdot g(E) \to \bigoplus_{[g]\in \Gamma_E \backslash \Gamma_F/\Gamma_E} E\cdot g(E); \quad (x)_{[g]} \mapsto \Big( \frac{t}{g(t)}\cdot x\Big)_{[g]}.
\]
As proved in \cite[\S 4.1]{MR3509939} (especially in \S 4.1.1) by using the assumption that \(E/F\) is tamely ramified,
the isomorphism \eqref{eq:rel-root-decomp} restricts to
\[
\frakA^{\ast} \simeq \bigoplus_{[g]\in \Gamma_E\backslash \Gamma_F/\Gamma_E} \calO_{E\cdot g(E)}, \quad 
\frakP_{\frakA^{\ast}} \simeq \bigoplus_{[g]\in \Gamma_E\backslash \Gamma_F/\Gamma_E} \frakp_{E\cdot g(E)}.
\]
Hence we have an isomorphism (see \cite[Proposition 4.4]{MR3509939})
\begin{equation} \label{eq:decomp-by-std-mods}
\frakU_{\frakA^{\ast}} =\frakA^{\ast}/\frakP_{\frakA^{\ast}} \simeq \bigoplus_{[g]\in \Gamma_E\backslash \Gamma_F/\Gamma_E} \frakU_{[g]}
\end{equation}
of \(k_F[\Psi_{E/F}]\)-modules, where we define
\[
\frakU_{[g]}:=k_{E\cdot g(E)}=\calO_{E\cdot g(E)}/\frakp_{E\cdot g(E)}
\]
with the induced \(k_F[\Psi_{E/F}]\)-module structure.
Explicitly, if \(g=\sigma^i \phi^j\) in the notation of Section \ref{sec:sigma-phi}, then
\(z \in \mu_E\) and \(\varpi_E\) act on \(\frakU_{[g]}=k_{E\cdot g(E)}\) as the multiplication by 
\[
\begin{bmatrix}
1 \\ \sigma^i \phi^j
\end{bmatrix}
(z)=z^{1-q^j}
\text{ and }
\begin{bmatrix}
1 \\ \sigma^i \phi^j
\end{bmatrix}
(\varpi_E)=(z_e^{i}z_{\phi^j})^{-1}
\]
respectively.

Writing the dual module of \(\frakU_{[g]}\) as \(\frakU_{[g]}^{\vee}:=\Hom_{k_F[\Psi_{E/F}]}(\frakU_{[g]}, k_F)\),
we have a natural identification
\(\frakU_{[g^{-1}]} \simeq \frakU_{[g]}^{\vee}\)
of \(k_F[\Psi_{E/F}]\)-modules induced by a \(\Psi_{E/F}\)-invariant non-degenerate \(k_F\)-bilinear pairing
\begin{equation} \label{eq:std-pairing}
k_{E\cdot g(E)} \times k_{E\cdot g^{-1}(E)} \to k_F; \quad (a, b)\mapsto \lan a, b\ran:=\tr_{k_{E\cdot g(E)}/k_F}(a \cdot g(b)).
\end{equation}

\begin{rem} \label{rem:frakU_{[g]}}
\begin{enumerate}
\item \label{item:not-a-simple-module} It is tempting to think that \eqref{eq:decomp-by-std-mods} is the decomposition into simple modules.
However, \(\frakU_{[g]}\) is not a simple \(k_F[\Psi_{E/F}]\)-module in general.
Indeed, if we view \(\begin{bmatrix}1 \\ g\end{bmatrix}\) as an element in \(\Hom(\Psi_{E/F}, \bar{k}_F^{\times})\) by abuse of notation, then \(\frakU_{[g]}\) is simple if and only if the length of orbit of \(\begin{bmatrix}1 \\ g\end{bmatrix}\) under the action of \(\Gamma_{k_F}\) is equal to \(\dim_{k_F}\frakU_{[g]}\).
If \(g=\sigma\), for example, then we can easily compute that the length is equal to \([k_F[z_e]: k_F]\) while we have \(\dim_{k_F}\frakU_{[g]}=[k_E[z_e]: k_F]\).

Moreover, even if \([g_1] \neq [g_2]\), the modules \(\frakU_{[g_1]}\) and \(\frakU_{[g_2]}\) can be isomorphic.
For example, take \(g_1=\sigma\) and \(g_2=\sigma^q\).
Then \(\frakU_{[g_1]}=k_{E\cdot g_1(E)}=k_E[z_e]\) and \(\frakU_{[g_2]}=k_{E\cdot g_2(E)}=k_E[z_e^q]\) are isomorphic \(k_F[\Psi_{E/F}]\)-modules because \(z_e\) and \(z_e^q\) are conjugate over \(k_F\).
Assume in addition that \(q \not\equiv 1 \pmod{e}\) (or equivalently \(\sigma \neq \sigma^q\)) and \(k_E=k_F[z_e]\) (which implies \(\sigma^{q^f}=\sigma\)).
Then we claim that \([g_1] \neq [g_2]\).
Indeed, since \(\Gamma_{L/E}\) is generated by a lift \(\phi_E\) of \(q^f\)-th power Frobenius element on \(k_L\),
which satisfies \(\phi_E \sigma \phi_E^{-1}=\sigma^{q^f}\),
we have 
\[
[\sigma] \cap \{\sigma^{t} \mid t\in \bbZ\}=\{\sigma^{q^{ft}}\mid t\in \bbZ\}=\{\sigma\},
\]
which shows that \(\sigma^q \notin [\sigma]\), as desired.

\item Neither of the above complications arise if we further consider the \(k_E\)-structure and view \(\frakU_{[g]}\) as a \(k_E[\Psi_{E/F}]\)-module instead; 
one can show that \(\frakU_{[g]}\) is a simple \(k_E[\Psi_{E/F}]\)-module for any \([g]\in \Gamma_E \backslash \Gamma_F /\Gamma_E\) 
and moreover \(\Hom_{k_E[\Psi_{E/F}]}(\frakU_{[g_1]}, \frakU_{[g_2]})\neq 0\) implies \([g_1]=[g_2]\).

However, the \(k_E[\Psi_{E/F}]\)-module structure is problematic in other respects; 
for instance, the pairing \eqref{eq:std-pairing} is not \(k_E\)-bilinear in general, leading to the failure of the analogous description of \(\frakU_{[g^{-1}]}\) in terms of the dual module.
Thus, we only make use of the \(k_F[\Psi_{E/F}]\)-module structure of \(\frakU_{[g]}\) in what follows.
\end{enumerate}
\end{rem}

\begin{rem}
Recall the bijection \eqref{eq:Galois-orb-as-coset} between \((\Gamma_E\backslash \Gamma_F/\Gamma_E)'=\Gamma_E\backslash \Gamma_F/\Gamma_E-\{\Gamma_E\}\) and \(\Gamma_F \backslash \Phi(\bfS, \bfG^{\ast})\).
If we take \(g\in \Gamma_F-\Gamma_E\) and put \(\alpha=\begin{bmatrix}1 \\ g\end{bmatrix}\), so that \([g]\) corresponds to the orbit of \(\alpha\) via the above bijection,
then \(E\cdot g(E)\) is nothing but the fixed field \(F_{\alpha}\) appearing in the definition of \(\chi\)-data and \(\zeta\)-data.
Moreover, the action of \(t\in E^{\times}=\bfS(F)\) on \(E\cdot g(E)\) (resp.\ on \(\frakU_{[g]}\)) is identified with the multiplication by \(\alpha(t)\) on \(F_{\alpha}\) (resp.\ \(\overline{\alpha(t)}\) on \(k_{F_{\alpha}}\)).
\end{rem}

Recall from Section \ref{sec:toral-invariant}
that \(\frac{h}{d}\) is the Hasse invariant of \(A\) with \(d, h \in \bbZ\) prime to each other.
The following proposition describes the \(k_F[\Psi_{E/F}]\)-module structure of \(\frakP^{j}/\frakP^{j+1}\) for general \(A\).
\begin{prop}\label{prop:decomp-of-grpiece}
For each \(j'\in \bbZ\), we have an isomorphism
\[
\frakP^{j'}/\frakP^{j'+1} \simeq \bigoplus_{\substack{[g]=[\sigma^i \phi^j] \in \Gamma_E\backslash \Gamma_F/\Gamma_E \\j \equiv hj' \mod{f_0}}}
\frakU_{[g]}
\]
of \(k_F[\Psi_{E/F}]\)-modules.
\end{prop}
Note that for any double coset \([\sigma^i \phi^j] \in \Gamma_E\backslash \Gamma_F/\Gamma_E\), the image of \(j\) in \(\bbZ/{f\bbZ}\) (and in particular in \(\bbZ/{f_0\bbZ}\)) is independent of the choice of a representative \(\sigma^i \phi^j\)
(this can be checked by looking at the action on \(\mu_E \subset L\)).
Also note that \(h\) is prime to \(d\) and hence to \(f_0=d_0\).
\begin{proof}
This is \cite[Proposition 4.2]{MR3505131}.
\end{proof}
\begin{rem}
When \(A=A^{\ast}\) and so \(f_0=1\), the condition \(j \equiv hj' \mod{f_0}\) is empty.
In particular, we have \(\frakA/\frakP \simeq \frakP^{j}/\frakP^{j+1}\) for all \(j \in \bbZ\).
This is also seen directly; any uniformizer \(\varpi_E \in E\) satisfies \(\frakP=\varpi_E \frakA =\frakA \varpi_E\) in this case and the multiplication by \(\varpi_E^j\) induces a \(\Psi_{E/F}\)-equivariant isomorphism.
\end{rem}

\subsection{Finite modules associated to an admissible pair} \label{sec:vmod}
In this section we associate to an admissible pair \((E/F, \xi)\) various \(k_F[\Psi_{E/F}]\)-modules \(\vmod{A}{}, \vmod{A}{k}, \vmod{A}{[g]}\)
and summarize some of their properties.
This is based on parts of \cite[\S 3.3, \S 3.6, \S 4.2]{MR3505131}.
Moreover, we will state a fact (Proposition \ref{prop:vmod-ord}) relating \(\vmod{A}{[g]}\) and \(\ordx{\bfG}{\alpha}\) from Section \ref{sec:epsilon-character}.

Let \((E/F, \xi)\) be an admissible pair
and take an \(F\)-embedding \(E \hookrightarrow A\) so that an embedding \(\bfS=\Res_{E/F} \Gm \hookrightarrow \underline{A}^{\times}=\bfG\) is induced.
Recall from Section \ref{sec:classical-Howe-factorization} that then we have
\begin{itemize}
\item subfields \(F=E_t \subsetneq E_{t-1} \subsetneq \cdots \subsetneq E_1 \subsetneq E_0 \subset E_{-1}=E\) and
\item integers \(a_t\geq a_{t-1} > \cdots >a_0 >a_{-1}=0\)
\end{itemize}
associated to \((E/F, \xi)\).
(Note that the indexing convention here is slightly different from that of \cite{MR3505131};
we set \(F=E_t\) whereas \(F=E_{t+1}\) in \cite{MR3505131}.)
Moreover, for each \(-1\leq k\leq t\), the centralizer \(A_k\subset A\) of \(E_k\) in \(A\) is a central simple algebra over \(E_k\) and the fixed embedding \(E\hookrightarrow A\) factors through \(E\hookrightarrow A_k\).
Thus the setting reviewed in Section \ref{sec:hered-ord} applies to \(E_k\) and \(A_k\) in place of \(F\) and \(A\);
the principal \(\calO_E\)-order associated to the embedding \(E\hookrightarrow A_k\) is \(\frakA_{k}:=A_k \cap \frakA\) and we also have
\[
\frakP_{\frakA_{k}}, \quad U_{\frakA_{k}}, \quad U_{\frakA_{k}}^x, \quad U_{\frakA_{k}}^{x+}
\]
for \(x\in \bbR_{\geq 0}\) and so on.
The central simple \(E_k\)-algebra \(A_k\) is isomorphic to \(M_{m(A_k)}(D_k)\) for a central division \(E_k\)-algebra \(D_k\) of dimension \(d(A_k)^2\), where
\begin{equation} \label{eq:m_k}
m(A_k)=\gcd(m, [E:E_k]), \quad d(A_k)=\frac{d}{\gcd(d, [E_k:F])}
\end{equation}
(see, e.g., \cite[page 357]{MR3505131} or \cite[(2.1.1)]{MR2848585}).
By \eqref{eq:e(frakA/calO_E)}, \eqref{eq:m_k}, we get
\begin{align} \label{eq:e(frakA_k/calO_E)}
\nonumber
e(\frakA_{k}/\calO_E)&=\frac{f(E/E_k)}{\gcd(m(A_k), f(E/E_k))}
=\frac{f(E/E_k)}{\gcd(m, [E:E_k], f(E/E_k))} \\ 
&=\frac{f(E/E_k)}{\gcd(m, f(E/E_k))}.
\end{align}
In particular, we see that \(e(\frakA_k/\calO_E)\) divides \(e(\frakA_{k+1}/\calO_E)\) for \(-1\leq k\leq t-1\).

We define compact open subgroups \(J^1, H^1\) of \(G\) and lattices \(\frakJ^1, \frakH^1\) of \(A\) by
\begin{align*}
J^1&:=U_{\frakA_0}^1U_{\frakA_1}^{\frac{1}{2}a_0e(\frakA_1/\calO_E)} \cdots U_{\frakA_{t-1}}^{\frac{1}{2}a_{t-2}e(\frakA_{t-1}/\calO_E)}U_{\frakA}^{\frac{1}{2}a_{t-1}e(\frakA/\calO_E)}, \\
H^1&:=U_{\frakA_0}^1U_{\frakA_1}^{\frac{1}{2}a_0e(\frakA_1/\calO_E)+} \cdots U_{\frakA_{t-1}}^{\frac{1}{2}a_{t-2}e(\frakA_{t-1}/\calO_E)+}U_{\frakA}^{\frac{1}{2}a_{t-1}e(\frakA/\calO_E)+}, \\
\frakJ^1&:=\frakP_{\frakA_0}+
\frakP_{\frakA_1}^{\frac{1}{2}a_0e(\frakA_1/\calO_E)}+
\cdots +
\frakP_{\frakA_{t-1}}^{\frac{1}{2}a_{t-2}e(\frakA_{t-1}/\calO_E)}+
\frakP_{\frakA}^{\frac{1}{2}a_{t-1}e(\frakA/\calO_E)}, \\
\frakH^1&:=\frakP_{\frakA_0}+
\frakP_{\frakA_1}^{\frac{1}{2}a_0e(\frakA_1/\calO_E)+}+
\cdots +
\frakP_{\frakA_{t-1}}^{\frac{1}{2}a_{t-2}e(\frakA_{t-1}/\calO_E)+}+
\frakP_{\frakA}^{\frac{1}{2}a_{t-1}e(\frakA/\calO_E)+}.
\end{align*}
The subgroups \(J^1, H^1\) and lattices \(\frakJ^1, \frakH^1\) depend on \(\xi\), but we suppress this dependence from the notation.
The same applies to \(\vmod{A}{}\) and related objects defined below.

The quotient \(\vmod{A}{}:=J^1/H^1=\frakJ^1/\frakH^1\) is a \(k_F\)-vector space.
As before, the conjugation action of \(E^{\times}\) induces an action of \(\Psi_{E/F}=E^{\times}/F^{\times}U_E^1\) on \(\vmod{A}{}\).
For \(0\leq k\leq t-1\), put
\begin{align*}
\vmod{A}{k}&:=\frac{U_{\frakA_{k+1}}^{\frac{1}{2}a_{k}e(\frakA_{k+1}/\calO_E)}}{U_{\frakA_k}^{\frac{1}{2}a_ke(\frakA_{k}/\calO_E)}U_{\frakA_{k+1}}^{\frac{1}{2}a_ke(\frakA_{k+1}/\calO_E)+}} \\
&\simeq \frac{\frakP_{\frakA_{k+1}}^{\frac{1}{2}a_{k}e(\frakA_{k+1}/\calO_E)}}{\frakP_{\frakA_k}^{\frac{1}{2}a_ke(\frakA_{k}/\calO_E)}+\frakP_{\frakA_{k+1}}^{\frac{1}{2}a_ke(\frakA_{k+1}/\calO_E)+}}.
\end{align*}
Note that the modules \(\vmod{A}{k}\) are identified with the successive quotients of the filtration of \(\vmod{A}{}\) given by the images of the subgroups
\[
U_{\frakA_{k+1}}^{\frac{1}{2}a_{k}e(\frakA_1/\calO_E)} \cdots U_{\frakA_{t-1}}^{\frac{1}{2}a_{t-2}e(\frakA_{t-1}/\calO_E)}U_{\frakA}^{\frac{1}{2}a_{t-1}e(\frakA/\calO_E)}
\]
of \(J^1\) for \(0\leq k\leq t-1\).
Thus, they are subquotients of \(\vmod{A}{}\) and we have an isomorphism
\begin{equation} \label{eq:vmod{A}{}into-vmod{A}{k}}
\vmod{A}{}\simeq \vmod{A}{0} \oplus \vmod{A}{1} \oplus \cdots \oplus \vmod{A}{t-1}
\end{equation}
of \(k_F[\Psi_{E/F}]\)-modules since the order of \(\Psi_{E/F}\) is prime to \(p\).

\begin{rem}
Following \cite{MR3505131}, we defined the \(k_F[\Psi_{E/F}]\)-module \(\vmod{A}{k}\) as a subquotient of \(\vmod{A}{}\).
In fact, as the following proposition shows, we can alternatively define \(\vmod{A}{k}\) as a \(k_F[\Psi_{E/F}]\)-submodule of \(\vmod{A}{}\),
which makes the isomorphism \eqref{eq:vmod{A}{}into-vmod{A}{k}} canonical
(this may be implicit in some parts of \cite{MR3505131}).
We will not use this definition in this paper.
\begin{prop} \label{prop:vmod{A}{}into-vmod{A}{k}}
Consider the non-degenerate symmetric bilinear pairing 
\[
B\colon A \times A \to F; \quad (x, y) \mapsto \Trd_{A/F} (xy)
\]
that is invariant under the conjugation action of \(A^{\times}\).
For \(0\leq k\leq t-1\) and \(r\in \bbR\) we define
\[
\frakP_{\frakA_{k+1}, \frakA_k}^r:=\frakP_{\frakA_{k+1}}^r \cap A_k^{\perp},
\]
where \(A_k^{\perp}\) is the orthogonal complement of \(A_k\) with respect to \(B\).
We similarly define \(\frakP_{\frakA_{k+1}, \frakA_k}^{r+}\).
By the \(A^{\times}\)-invariance of \(B\), both the modules \(\frakP_{\frakA_{k+1}, \frakA_k}^r, \frakP_{\frakA_{k+1}, \frakA_k}^{r+}\) are stable under the conjugation action of \(E^{\times}\).
Then we have
\begin{align*}
\frakJ^1&=\frakP_{\frakA_0}\oplus
\frakP_{\frakA_1, \frakA_0}^{\frac{1}{2}a_0e(\frakA_1/\calO_E)}\oplus
\cdots \oplus
\frakP_{\frakA_{t-1}, \frakA_{t-2}}^{\frac{1}{2}a_{t-2}e(\frakA_{t-1}/\calO_E)}\oplus
\frakP_{\frakA_t, \frakA_{t-1}}^{\frac{1}{2}a_{t-1}e(\frakA/\calO_E)}, \\
\frakH^1&=\frakP_{\frakA_0}\oplus
\frakP_{\frakA_1, \frakA_0}^{\frac{1}{2}a_0e(\frakA_1/\calO_E)+}\oplus
\cdots \oplus
\frakP_{\frakA_{t-1}, \frakA_{t-2}}^{\frac{1}{2}a_{t-2}e(\frakA_{t-1}/\calO_E)+}\oplus
\frakP_{\frakA_t, \frakA_{t-1}}^{\frac{1}{2}a_{t-1}e(\frakA/\calO_E)+}.
\end{align*}
In particular, we have canonical identifications
\begin{align*}
\vmod{A}{} &\simeq \bigoplus_{0\leq k\leq t-1} \frakP_{\frakA_{k+1}, \frakA_k}^{\frac{1}{2}a_ke(\frakA_{k+1}/\calO_E)}/\frakP_{\frakA_{k+1}, \frakA_k}^{\frac{1}{2}a_ke(\frakA_{k+1}/\calO_E)+}, \\
\vmod{A}{k} &\simeq \frakP_{\frakA_{k+1}, \frakA_k}^{\frac{1}{2}a_ke(\frakA_{k+1}/\calO_E)}/\frakP_{\frakA_{k+1}, \frakA_k}^{\frac{1}{2}a_ke(\frakA_{k+1}/\calO_E)+}.
\end{align*}
\end{prop}
\begin{proof}
Once the relevant lattices are described in terms of the Moy--Prasad filtration (cf.\ Appendix \ref{sec:app}),
the desired direct sum decompositions are obtained by repeatedly applying \cite[Proposition 1.9.3]{MR1653184}, which is based on \cite[Proposition 1.4.1]{MR1653184} and is applicable by the assumption that \(E/F\) is tame.
We omit the details.
\end{proof}
\end{rem}

We put 
\[
j_k=\frac{1}{2}a_ke(\frakA/\calO_E) \in \frac{1}{2}\bbZ,
\]
so that we have
\[
\frakP_{\frakA}^{j_k}\cap \frakA_{k+1}=\frakP_{\frakA_{k+1}}^{\frac{1}{2}a_ke(\frakA_{k+1}/\calO_E)}.
\]

\begin{prop} \label{prop:comp-decomp}
For \(0\leq k\leq t-1\), the \(k_F[\Psi_{E/F}]\)-module \(\vmod{A}{k}\) is described as follows.
\begin{enumerate}
\item\label{item:comp-decomp-odd} If \(a_ke(\frakA_{k+1}/\calO_E)\) is odd, then \(\vmod{A}{k}\) is trivial.
\item\label{item:comp-decomp-even} If \(a_ke(\frakA_{k+1}/\calO_E)\) is even, then
\[
\vmod{A}{k} \simeq \bigoplus_{\substack{[g]=[\sigma^i\phi^j]\in \Gamma_E\backslash (\Gamma_{E_{k+1}}-\Gamma_{E_{k}})/\Gamma_E \\ j\equiv hj_k \mod{e(\frakA/\calO_E)}}} \frakU_{[g]}.
\]
\end{enumerate}
\end{prop}
As in Proposition \ref{prop:decomp-of-grpiece}, two remarks apply.
First, the image of \(j\) in \(\bbZ/{e(\frakA/\calO_E)\bbZ}\) is independent of the choice of a coset representative (recall that \(e(\frakA/\calO_E)=f_0\)) and thus the direct sum in \eqref{item:comp-decomp-even} is well-defined.
Also \(h\) is prime to \(d\) and hence to \(e(\frakA/\calO_E)=d_0\).

Moreover, if \(a_ke(\frakA_{k+1}/\calO_E)\) is even, then so is \(a_ke(\frakA/\calO_E)\) and hence \(j_k\) lies in \(\bbZ\), so that the summation in \eqref{item:comp-decomp-even} makes sense.

\begin{proof}
This is a minor rephrasing of \cite[Proposition 4.6]{MR3505131}, which, under the same assumptions, asserts the following:
\begin{enumerate}
\item[(T1)] If \(a_k\) is odd and \(e(\frakA_{k+1}/\calO_E)\) is odd, then \(\vmod{A}{k}\) is trivial.
\item[(T2)] If \(a_k\) is odd, \(e(\frakA_{k}/\calO_E)\) is odd and \(e(\frakA_{k+1}/\calO_E)\) is even, then
\[
\vmod{A}{k} \simeq \bigoplus_{\substack{[g]=[\sigma^i\phi^j]\in \Gamma_E\backslash \Gamma_{E_{k+1}}/\Gamma_E \\ j\equiv hj_k \mod{e(\frakA/\calO_E)}}} \frakU_{[g]}.
\]
\item[(T3)] If \(a_k\) is even or \(e(\frakA_{k}/\calO_E)\) is even, then
\[
\vmod{A}{k} \simeq \bigoplus_{\substack{[g]=[\sigma^i\phi^j]\in \Gamma_E\backslash (\Gamma_{E_{k+1}}-\Gamma_{E_{k}})/\Gamma_E \\ j\equiv hj_k \mod{e(\frakA/\calO_E)}}} \frakU_{[g]}.
\]
\end{enumerate}
We note that it is proved basically by dividing the cases according to whether 
each of
\[
\frac{1}{2}a_ke(\frakA_{k+1}/\calO_E), \frac{1}{2}a_ke(\frakA_k/\calO_E) \in \frac{1}{2} \bbZ
\]
lies in \(\bbZ\)
and applying \cite[Proposition 4.2]{MR3505131} (reviewed as Proposition \ref{prop:decomp-of-grpiece}) to \(\frakA_{k+1}\) and \(\frakA_{k}\).

Let us deduce Proposition \ref{prop:comp-decomp} from \cite[Proposition 4.6]{MR3505131}.
As \eqref{item:comp-decomp-odd} and (T1) are clearly equivalent,
we have to show that (T2) and (T3) imply \eqref{item:comp-decomp-even}.
For this, it suffices to check that if \(a_ke(\frakA_{k}/\calO_E)\) is odd, \(e(\frakA_{k+1}/\calO_E)\) is even and \(g=\sigma^i\phi^j \in \Gamma_{E_{k+1}}\) satisfies \(j\equiv hj_k \pmod{e(\frakA/\calO_E)}\)
then \(g\) does not lie in \(\Gamma_{E_{k}}\).

Since \(e(\frakA_{k+1}/\calO_E)\) is even, so is \(e(\frakA/\calO_E)\) and hence \(h\) is odd because it is prime to \(e(\frakA/\calO_E)\) as remarked right after the statement of the proposition.
Together with the assumption that \(a_k\) is odd, this shows
\[
j \equiv hj_k \equiv \frac{1}{2} e(\frakA/\calO_E) \mod e(\frakA/\calO_E).
\]
Thus, we may write \(j=(l+\frac{1}{2})\cdot e(\frakA/\calO_E)\) with some \(l \in \bbZ\).
Now we claim that \(f(E_k/F)\) does not divide \(j\).
Indeed, by \eqref{eq:e(frakA/calO_E)}, we have
\[
\frac{j}{f(E_k/F)}
=(l+\frac{1}{2}) \cdot \frac{e(\frakA/\calO_E)}{f(E_k/F)}
=(l+\frac{1}{2}) \cdot \frac{f(E/E_k)}{\gcd(m, f(E/F))}.
\]
On the other hand, since by \eqref{eq:e(frakA_k/calO_E)} we have
\[
e(\frakA_{k}/\calO_E)=\frac{f(E/E_k)}{\gcd(m, f(E/E_k))},
\]
which is odd by assumption, it follows that \(f(E_k/F)\) does not divide \(j\), as claimed.
By looking at the action on the residue fields, we now see that \(g \notin \Gamma_{E_k}\).
\end{proof}

\begin{rem}\label{rem:split-comp-decomp}
When \(A=A^{\ast}\) and so \(e(\frakA_{k}/\calO_E)=1\) for any \(k\), the above statement reduces to the following simpler one.
\begin{enumerate}
	\item[(1)] If \(a_k\) is odd, then \(\vmod{A^{\ast}}{k}\) is trivial.
	\item[(2)] If \(a_k\) is even, then
	\[
	\vmod{A^{\ast}}{k} \simeq \bigoplus_{[g]=[\sigma^i\phi^j]\in \Gamma_E\backslash (\Gamma_{E_{k+1}}-\Gamma_{E_{k}})/\Gamma_E} \frakU_{[g]}.
	\]
\end{enumerate}
\end{rem}
In view of Proposition \ref{prop:comp-decomp}, for \([g] \in (\Gamma_E \backslash \Gamma_F/\Gamma_E)'\) we define a \(k_F[\Psi_{E/F}]\)-module \(\vmod{A}{[g]}\) in the following way.
If \(g\in \Gamma_{E_0}\), then we define \(\vmod{A}{[g]}=0\) to be the trivial module.
If \(g\notin \Gamma_{E_0}\), then let \(0\leq k\leq t-1\) be the unique integer such that \(g\in \Gamma_{E_{k+1}} - \Gamma_{E_{k}}\).
Also we express \([g]=[\sigma^i\phi^j]\) for some \(0\leq i\leq e-1\) and \(0\leq j\leq f-1\).
Then we define
\begin{equation} \label{eq:vmodA[g]}
\vmod{A}{[g]}:=
\begin{cases}
\frakU_{[g]} & \text{if \(a_ke(\frakA_{k+1}/\calO_E)\) is even and \(j\equiv hj_k \pmod{e(\frakA/\calO_E)}\)} \\
0 & \text{otherwise.}
\end{cases}
\end{equation}
We also define 
\[
\bfrakU_{[g]}:=
\begin{cases*}
\frakU_{[g]} \oplus \frakU_{[g^{-1}]} &\text{if \([g]\) is asymmetric} \\
\frakU_{[g]} &\text{if \([g]\) is symmetric,}
\end{cases*}
\]
so that \(\bfrakU_{[g]}\) is a self-dual \(k_F[\Psi_{E/F}]\)-module.
Similarly, we put
\[
\bvmod{A}{[g]}:=
\begin{cases*}
\vmod{A}{[g]} \oplus \vmod{A}{[g^{-1}]} &\text{if \([g]\) is asymmetric} \\
\vmod{A}{[g]} &\text{if \([g]\) is symmetric.}
\end{cases*}
\]
Note that \(\bvmod{A}{[g]}\) is either trivial or isomorphic to \(\bfrakU_{[g]}\).
Indeed, we need to check that if \(\vmod{A}{[g]}=\frakU_{[g]}\) then the same holds for \(g^{-1}\) in place of \(g\).
This is clear because if \(a_ke(\frakA_{k+1}/\calO_E)\) is even, \(0\leq j\leq f-1\) and \(j\equiv hj_k \pmod{e(\frakA/\calO_E)}\), then
\(-j \equiv -hj_k =-\frac{h}{2}a_ke(\frakA/\calO_E)\equiv \frac{h}{2}a_ke(\frakA/\calO_E)=hj_k \pmod{e(\frakA/\calO_E)}.\)

By \eqref{eq:vmod{A}{}into-vmod{A}{k}} and Proposition \ref{prop:comp-decomp}, we have isomorphisms of \(k_F[\Psi_{E/F}]\)-modules
(see Section \ref{sec:sigma-phi} for the definition of \((\Gamma_E\backslash \Gamma_F/\Gamma_E)_{\sym}\),
\((\Gamma_E\backslash \Gamma_F/\Gamma_E)_{\asym}/{\pm}\)):
\begin{equation} \label{eq:vmod{A}{}into-vmod{A}{[g]}}
\vmod{A}{}
\simeq \bigoplus_{[g]\in (\Gamma_E\backslash \Gamma_F/\Gamma_E)'} \vmod{A}{[g]}
\simeq \bigoplus_{[g]\in (\Gamma_E\backslash \Gamma_F/\Gamma_E)_{\sym}\amalg (\Gamma_E\backslash \Gamma_F/\Gamma_E)_{\asym}/{\pm}} \bvmod{A}{[g]}.
\end{equation}

\begin{rem} \label{rem:vmod{A}{}into-vmod{A}{[g]}}
By Remark \ref{rem:frakU_{[g]}} \eqref{item:not-a-simple-module}, it is probably not possible to define \(\vmod{A}{[g]}\) simply as ``the \(\frakU_{[g]}\)-isotypic part" of \(\vmod{A}{}\).

However, similarly to Proposition \ref{prop:vmod{A}{}into-vmod{A}{k}}, one can again use \cite[Proposition 1.4.1]{MR1653184} and show the existence of direct sum decompositions 
of the form:
\[
\frakJ^1=\bigoplus_{[g]\in \Gamma_E \backslash \Gamma_F/\Gamma_E} \frakJ^1_{[g]}, \quad
\frakH^1=\bigoplus_{[g]\in \Gamma_E \backslash \Gamma_F/\Gamma_E} \frakH^1_{[g]},
\]
induced by the root space decomposition.
A more natural definition of \(\frakU_{[g]}\) may then be \(\vmod{A}{[g]}:=\frakJ^1_{[g]}/\frakH^1_{[g]}\).
This would make the isomorphisms \eqref{eq:vmod{A}{}into-vmod{A}{[g]}} canonical.

In this paper we will not use this definition and adopt the above ad hoc definition.
\end{rem}

\begin{prop}\label{prop:symp-mods}
Let \([g]\in (\Gamma_E\backslash \Gamma_F/\Gamma_E)_{\sym}\) be a non-trivial symmetric double coset.
Then we have \(\vmod{A}{[g]} \simeq \vmod{A^{\ast}}{[g]}\) 
if and only if 
\([g]=[\sigma^i]\) for some \(i\) or \(m\) is even.	
\end{prop}
\begin{proof}
We prove the proposition by studying the condition for \(\vmod{A}{[g]} \simeq \vmod{A^{\ast}}{[g]}\).

We may assume \(g=\sigma^i \phi^j\) for some \(0\leq i\leq e-1, 0\leq j\leq f-1\).
Let \(k\) be the unique integer satisfying \(g=\sigma^i\phi^j \in \Gamma_{E_{k+1}}-\Gamma_{E_{k}}\).
Then by the definition \eqref{eq:vmodA[g]} of \(\vmod{A}{[g]}\) and Remark \ref{rem:split-comp-decomp}
it follows that we have \(\vmod{A}{[g]} \simeq \vmod{A^{\ast}}{[g]}\) if and only if one of the following cases occurs:
\begin{itemize}
\item \(a_k\) and \(e(\frakA_{k+1}/\calO_E)\) are odd,
\item \(a_k\) is odd, \(e(\frakA_{k+1}/\calO_E)\) is even and 
\[
j\not\equiv hj_k \mod e(\frakA/\calO_E),
\]
\item \(a_k\) is even and
\[
j\equiv hj_k \mod e(\frakA/\calO_E).
\]
\end{itemize}
Note that in the first two cases we have \(\vmod{A}{[g]}=0=\vmod{A^{\ast}}{[g]}\)
while in the last case we have \(\vmod{A}{[g]} \simeq \frakU_{[g]} \simeq \vmod{A^{\ast}}{[g]}\).

Recall that since \(g=\sigma^i\phi^j\) is assumed to be symmetric, we have either \(j=0\) or \(j=\frac{f}{2}\) (and in particular \(f\) is even in the latter case) by Proposition \ref{prop:cond-by-Galois-elt} \eqref{item:cond-for-sym}.
We are to prove that we have \(\vmod{A}{[g]} \simeq \vmod{A^{\ast}}{[g]}\) if and only if \(j=0\) or \(m\) is even.
By the above rephrasing of the condition \(\vmod{A}{[g]} \simeq \vmod{A^{\ast}}{[g]}\) in terms of the three cases, we are reduced to proving the following three assertions.
\begin{enumerate}
\item \label{item:isom1} Suppose that \(a_k\) and \(e(\frakA_{k+1}/\calO_E)\) are odd. Then \(j=0\) or \(m\) is even.
\item \label{item:isom2} Suppose that \(a_k\) is odd and \(e(\frakA_{k+1}/\calO_E)\) is even. Then \(j\not\equiv hj_k \pmod{e(\frakA/\calO_E)}\) if and only if \(j=0\) or \(m\) is even.
\item \label{item:isom3} Suppose that \(a_k\) is even. Then \(j\equiv hj_k \pmod{e(\frakA/\calO_E)}\) if and only if \(j=0\) or \(m\) is even.
\end{enumerate}
To study the congruence condition between \(j\) and \(hj_k\), we use
\begin{equation} \label{eq:cong-j-hjk}
\frac{j-hj_k}{e(\frakA/\calO_E)}=
\begin{cases}
\frac{-ha_k}{2} &\text{if \(j=0\)} \\
\frac{s-ha_k}{2} &\text{if \(j=\frac{f}{2}\).}
\end{cases}
\end{equation}
Here we simply write \(s=s(\frakA)=\gcd(f, m)\) 
(cf.\ \eqref{eq:r(frakA)}, \eqref{eq:e(frakA/calO_E)}).

We first treat the assertion \eqref{item:isom1}.
Suppose that \(a_k\) and \(e(\frakA_{k+1}/\calO_E)\) are odd and moreover \(j=\frac{f}{2}\) (so that \(f\) is even).
Then the condition \(\sigma^i\phi^{\frac{f}{2}} \in \Gamma_{E_{k+1}}\) shows that
\(f(E_{k+1}/F)\) divides \(\frac{f}{2}\), or equivalently, 
\[
\frac{f}{f(E_{k+1}/F)}=f(E/E_{k+1})
\]
is even.
Now because by \eqref{eq:e(frakA_k/calO_E)} we have
\[
e(\frakA_{k+1}/\calO_E)=\frac{f(E/E_{k+1})}{\gcd(f(E/E_{k+1}), m)},
\]
which is odd by assumption, \(m\) is even as required.

Next we show the assertion \eqref{item:isom2}.
Suppose that \(a_k\) is odd and \(e(\frakA_{k+1}/\calO_E)\) is even.
As \(e(\frakA_{k+1}/\calO_E)\) is even, so is \(e(\frakA/\calO_E)\) and hence \(h\) is odd because it is prime to \(e(\frakA/\calO_E)\).
Together with the assumption that \(a_k\) is odd, 
we see that \(ha_k\) is odd.
By \eqref{eq:cong-j-hjk}, this implies that we have \(j\not\equiv hj_k \pmod{e(\frakA/\calO_E)}\)
if and only if either
\begin{itemize}
\item \(j=0\), or 
\item \(j=\frac{f}{2}\) and \(s\) is even.
\end{itemize}
The latter in turn is equivalent to the condition that \(j=\frac{f}{2}\) and \(m\) is even
because \(s=\gcd(f, m)\).
Thus, the assertion \eqref{item:isom2} is settled.

Finally, the assertion \eqref{item:isom3} is proved similarly.
Assuming that \(a_k\) is even, we see that \(ha_k\) is even.
Again by \eqref{eq:cong-j-hjk}, we have \(j\equiv hj_k \pmod{e(\frakA/\calO_E)}\) if and only if either
\begin{itemize}
\item \(j=0\), or 
\item \(j=\frac{f}{2}\) and \(s\) is even,
\end{itemize}
the latter of which is equivalent to the condition that \(j=\frac{f}{2}\) and \(m\) is even, as before.
\end{proof}

The following proposition gives a relation between \(\vmod{A}{[g]}\) and the objects in Section \ref{sec:epsilon-character}.
\begin{prop} \label{prop:vmod-ord}
Let \(g\in \Gamma_F-\Gamma_E\) and put \(\alpha :=\begin{bmatrix}1 \\g\end{bmatrix} \in \Phi(\bfS, \bfG)\).
Let \(-1\leq k\leq t-1\) be the unique integer such that \(\alpha \in \Phi(\bfS, \bfG^{k+1})-\Phi(\bfS, \bfG^k)\) (or equivalently, \([g] \in \Gamma_E \backslash (\Gamma_{E_{k+1}}-\Gamma_{E_k}) /\Gamma_E\)).
Let \(r_k \in \bbR_{\geq 0}\) and \(\ordx{\bfG}{\alpha} \subset \bbR\)
be as in Section \ref{sec:epsilon-character}.

When \(k=-1\), we have \(\vmod{A}{[g]}=0\).
When \(k\neq -1\), we have \(\vmod{A}{[g]} \neq 0\) if and only if \(\frac{r_i}{2} \in \ordx{\bfG}{\alpha}\).

The same holds with \(A\) and \(\bfG\) replaced by \(A^{\ast}\) and \(\bfG^{\ast}\) respectively.
\end{prop}
\begin{proof}
For \(A^{\ast}\) and \(\bfG^{\ast}\), this is \cite[Proposition 5.12]{MR4206603}, which is based on \cite[Proposition 3.8]{MR2543925} and \cite[Proposition 4.21]{MR3849622}.
The same proof works for \(A\) and \(\bfG\).
\end{proof}

\subsection{t-factors}
As reviewed in Section \ref{sec:zeta-data} below,
Tam's \(\zeta\)-data \(\zeta_{\Tam}\) (see Section \ref{sec:Tam-thm}) is defined in terms of certain invariants \(t_{\bGamma}^0(V)\), \(t_{\bGamma}^1(V)\) and \(t_{\bGamma}(V)\) introduced by Bushnell--Fr\"{o}hlich \cite[Chapter 8]{MR701540} and Bushnell--Henniart \cite[\S 4]{MR2148193}, \cite[\S 3]{MR2679700}
and their variants by Tam \cite[page 376]{MR3505131}.
They are called t-factors in \cite{MR3505131} and \cite{MR3509939}.
Here we only recall some formal properties.

Let \(\bGamma\) be a cyclic group of order prime to \(p\).
Bushnell--Henniart \cite{MR2679700}, building on the work \cite{MR701540} of Bushnell--Fr\"{o}hlich, defined two invariants of a finite symplectic \(\bbF_p[\bGamma]\)-module \((V, h)\)
(i.e., \(V\) is a finite-dimensional representation of \(\bGamma\) over \(\bbF_p\) and \(h\colon V \times V \to \bbF_p\) is a \(\bGamma\)-invariant non-degenerate alternating \(\bbF_p\)-bilinear form):
\begin{itemize}
\item a sign \(t_{\bGamma}^0(V) \in \{\pm 1\}\) and 
\item a quadratic character \(t_{\bGamma}^1(V) \colon \bGamma \to \{\pm 1\}\) of \(\bGamma\).
\end{itemize}
It is proved in \cite[Proposition 8.2.3 (c)]{MR701540} that the isometry class of a finite symplectic \(\bbF_p[\bGamma]\)-module is determined by the isomorphism class of the underlying \(\bbF_p[\bGamma]\)-module.
It is for this reason that, in the above and in the following, we do not specify the \(\bGamma\)-invariant non-degenerate alternating \(\bbF_p\)-bilinear form when we discuss these invariants.
We record that they are additive in \(V\): for finite symplectic \(\bbF_p[\bGamma]\)-modules \(V_1, V_2\),
\[
t_{\bGamma}^0(V_1 \oplus V_2)=t_{\bGamma}^0(V_1) \cdot t_{\bGamma}^0(V_2), \quad
t_{\bGamma}^1(V_1 \oplus V_2)=t_{\bGamma}^1(V_1) \cdot t_{\bGamma}^1(V_2)
\]
and are trivial for a trivial finite symplectic \(\bbF_p[\bGamma]\)-module \(V=0\):
\[
t_{\bGamma}^0(V)=1, \quad
t_{\bGamma}^1(V)=\mathbbm{1}_{\bGamma}.
\]
Moreover, the sign \(t_{\bGamma}(V) \in \{\pm 1\}\) is defined as
\[
t_{\bGamma}(V):=t_{\bGamma}^0(V) \cdot t_{\bGamma}^1(V)(\gamma)
\]
for a generator \(\gamma \in \bGamma\).
It is clear that \(t_{\bGamma}(V)\) is independent of the choice of \(\gamma\).
In the following, \(t_{\bGamma}^0(V), t_{\bGamma}^1(V), t_{\bGamma}(V)\) are collectively referred to as t-factors.

Let us return to the situation of Section \ref{sec:vmod}.
We consider the following two cyclic subgroups of \(\Psi_{E/F}=E^{\times}/F^{\times}U_E^1\):
\begin{itemize}
\item \(\bmu:=\mu_E/\mu_F\) and
\item \(\lan \varpi \ran:=\)(the cyclic subgroup generated by the image of \(\varpi_E\)).
\end{itemize}
There exists a \(\Psi_{E/F}\)-invariant non-degenerate alternating \(\bbF_p\)-bilinear form on \(\bvmod{A}{[g]}\) 
(see Remark \ref{rem:alt-form} below) 
and thus we can view \(k_F[\Psi_{E/F}]\)-modules \(\bvmod{A}{[g]}\) as finite symplectic \(\bbF_p[\bGamma]\)-modules for \(\bGamma=\bmu, \lan \varpi \ran\)
so that
t-factors such as 
\(t_{\bmu}^1(\bvmod{A}{[g]})\), \(t_{\lan \varpi \ran}(\bvmod{A}{[g]})\) make sense.
\begin{rem} \label{rem:alt-form}
One can show that \(\bvmod{A}{[g]}\) admits a \(\Psi_{E/F}\)-invariant non-degenerate alternating \(\bbF_p\)-bilinear form in at least two ways.
\begin{enumerate}
\item By mimicking the construction of finite symplectic \(\bbF_p[\bGamma]\)-modules in \cite[Proposition (8.2.3)]{MR701540} (see also \cite[Lemma 4.7]{MR2148193}) in the setting of \(k_F[\Psi_{E/F}]\)-modules
with some additional argument in the symmetric case, 
we can directly define a 
\(\Psi_{E/F}\)-invariant non-degenerate alternating \(k_F\)-bilinear form as follows and in particular get a similar alternating \(\bbF_p\)-bilinear form by post-composing \(\tr_{k_F/\bbF_p}\).

Suppose first that \([g]\) is asymmetric.
Then we have \(\bvmod{A}{[g]}=\vmod{A}{[g]}\oplus \vmod{A}{[g^{-1}]}\) by definition.
Under the natural identification \(\vmod{A}{[g^{-1}]} \simeq \vmod{A}{[g]}^{\vee}\),
the standard \(\Psi_{E/F}\)-invariant non-degenerate alternating \(k_F\)-bilinear form
\begin{align*}
(\vmod{A}{[g]}\oplus \vmod{A}{[g]}^{\vee}) \times (\vmod{A}{[g]}\oplus \vmod{A}{[g]}^{\vee}) &\to k_F \\
\big((v_1, w_1), (v_2, w_2)\big) &\mapsto \lan v_1, w_2\ran -\lan v_2, w_1\ran
\end{align*}
on \(\vmod{A}{[g]}\oplus \vmod{A}{[g]}^{\vee}\) induces a similar alternating form on \(\bvmod{A}{[g]}\).

Suppose now that \([g]\) is symmetric, so that we have \(\bvmod{A}{[g]}=\vmod{A}{[g]}\) by definition.
We may also assume \(\vmod{A}{[g]}\) is non-trivial.
We claim that then it is symmetric unramified.
Indeed, if it were symmetric ramified, then by Proposition \ref{prop:cond-by-Galois-elt} \eqref{item:cond-for-symram}, \(e\) would be even and the orbit of \(\alpha=\begin{bmatrix}1 \\ g\end{bmatrix}\) would be represented by \(\begin{bmatrix}1 \\ \sigma ^{\frac{e}{2}}\end{bmatrix}\).
Then we would have \(\vmod{A}{[g]}\simeq \vmod{A^{\ast}}{[g]} =0\)
by Proposition \ref{prop:symp-mods} and \cite[Proposition 5.3 (ii)]{MR3509939}, which is a contradiction.

Put \(\alpha =\begin{bmatrix}1 \\ g\end{bmatrix}\).
As \(\vmod{A}{[g]}\) is non-trivial, we have \(\vmod{A}{[g]}\simeq k_{E\cdot g(E)}=k_{F_{\alpha}}\).
Since \(\alpha\) is symmetric, there exists an element \(x\in \Gamma_F\) such that
\begin{equation} \label{eq:cons-of-sym}
x\cdot \alpha
=-\alpha.
\end{equation}
Note that \(x\) lies in \(\Gamma_{F_{\pm \alpha}}\) and the image generates the Galois group \(\Gamma_{F_{\alpha}/F_{\pm \alpha}}\),
for which we use the same letter.
Moreover, since \(\alpha\) is symmetric unramified, we have a non-trivial extension \(k_{F_{\alpha}}/k_{F_{\pm \alpha}}\) and 
the Galois group \(\Gamma_{k_{F_{\alpha}}/k_{F_{\pm \alpha}}}\) is generated by the image of \(x \in \Gamma_{F_{\alpha}/F_{\pm \alpha}}\), which we again continue to write as \(x\).
We can take an element \(\gamma \in k_{F_{\alpha}}^{\times}\) such that \(x(\gamma)=-\gamma\).
Under the identification \(k_{F_{\alpha}} \simeq \vmod{A}{[g]}\), the non-degenerate alternating \(k_F\)-bilinear form
\[
k_{F_{\alpha}} \times k_{F_{\alpha}} \to k_F; \quad
(a, b) \mapsto \tr_{k_{F_{\alpha}}/k_F}(a\cdot \gamma \cdot x(b))
\]
on \(k_{F_{\alpha}}\) induces a similar alternating form on \(\vmod{A}{[g]}\), which is \(\Psi_{E/F}\)-invariant by \eqref{eq:cons-of-sym}.

\item In fact, \(\vmod{A}{}\) carries a \(\Psi_{E/F}\)-invariant non-degenerate alternating \(\bbF_p\)-bilinear form, denoted by \(h_{\theta}\) in \cite{MR3505131}, which naturally arises in the the construction of \({}_{\bfG}\pi_{(\bfS, \xi)}\).
One can show that the canonical decomposition 
\[
\vmod{A}{}=\bigoplus_{[g]\in (\Gamma_E\backslash \Gamma_F/\Gamma_E)_{\sym}\amalg (\Gamma_E\backslash \Gamma_F/\Gamma_E)_{\asym}/{\pm}} \bvmod{A}{[g]}
\]
in Remark \ref{rem:vmod{A}{}into-vmod{A}{[g]}} is orthogonal with respect to \(h_{\theta}\) (cf.\ \cite[Proposition 4.4]{MR3505131}).
In particular, the restriction of \(h_{\theta}\) to each submodule
\(\bvmod{A}{[g]}\subset \vmod{A}{}\) provides a \(\Psi_{E/F}\)-invariant non-degenerate alternating \(\bbF_p\)-bilinear form.
\end{enumerate}
\end{rem}

Let \(g=\sigma^i \phi^j \in \Gamma_F\) be an element as in Section \ref{sec:sigma-phi} and put
\[
\bmu_g:=\mu_{E\cdot g(E)}/\mu_F.
\]
In \cite[page 376]{MR3505131}, Tam defined 
\begin{itemize}
\item a sign \(t_{\bmu_g}^0(\bfrakU_{[g]}) \in \{\pm 1\}\) and
\item a quadratic character \(t_{\bmu_g}^1(\bfrakU_{[g]}) \colon \bmu_g \to \{\pm 1\}\) of \(\bmu_g\).
\end{itemize}
The following definition is also implicitly used in \cite{MR3505131}:
\[
t_{\bmu_g}^1(\bvmod{A}{[g]}):=
\begin{cases}
t_{\bmu_g}^1(\bfrakU_{[g]}) &\text{if \(\bvmod{A}{[g]}\simeq \bfrakU_{[g]}\)} \\
\mathbbm{1}_{\bmu_g} &\text{if \(\bvmod{A}{[g]}\) is trivial.}
\end{cases}
\]
We only record the following proposition.
\begin{prop} \label{prop:res-of-t^1}
Let \(g=\sigma^i \phi^j \in \Gamma_F\) be as above.
Then we have
\[
t_{\bmu_g}^1(\bvmod{A}{[g]})|_{\bmu}=t_{\bmu}^1(\bvmod{A}{[g]}).
\]
\end{prop}
\begin{proof}
This follows from \cite[Proposition 4.12]{MR3505131}.
\end{proof}

\subsection{Tam's \(\zeta\)-data} \label{sec:zeta-data}
In order to review Tam's \(\zeta\)-data \(\zeta_{\Tam}=\{\zeta_{\Tam, \alpha}\}_{\alpha \in \Phi(\bfS, \bfG^{\ast})}\),
we fix some sets of representatives of relevant quotient sets of \(\Gamma_F\).

We fix a set \(\calD_{\asym /\pm} \subset \Gamma_F\) (resp.\ \(\calD_{\sym} \subset \Gamma_F\)) of representatives of \((\Gamma_E \backslash \Gamma_F/\Gamma_E)_{\asym /{\pm}}\)
(resp.\ \((\Gamma_E \backslash \Gamma_F/\Gamma_E)_{\sym}\)).
We arrange them so that we have
\[
\calD_{\asym /\pm}, \calD_{\sym} \subset \{ \sigma^i \phi^j \mid 0\leq i\leq e-1, 0\leq j\leq f-1\}.
\]
We also define the following subsets (cf.\ Remark \ref{rem:Tam-sym-ur/ram}):
\[
\calD_{\sym, \ur}:=\{ \sigma^i \phi^j \in \calD_{\sym} \mid j=f/2 \}, \quad
\calD_{\sym, \ram}:=\{\sigma^i \phi^j \in \calD_{\sym} \mid j=0\},
\]
so that we have \(\calD_{\sym, \ur} \cup \calD_{\sym, \ram}=\calD_{\sym}\) by Proposition \ref{prop:cond-by-Galois-elt} \eqref{item:cond-for-sym}.

In the following, we recall Tam's \(\zeta\)-data from \cite[Theorem 5.5]{MR3505131} and present them in a form useful for our purpose.
Note that while a set of \(\zeta\)-data consists of characters indexed by \(\alpha \in \Phi(\bfS, \bfG^{\ast})\),
we only specify the characters \(\zeta_{\Tam, \alpha}\) for \(\alpha=\begin{bmatrix}1 \\ g\end{bmatrix}\) with \(g\in \calD_{\asym /\pm} \cup \calD_{\sym}\)
as in \cite[Theorem 5.5]{MR3505131}.
By the axioms of \(\zeta\)-data, this uniquely determines a set of \(\zeta\)-data.

\subsubsection{Asymmetric roots} \label{sec:review-Tam-asymm}
Let \(g\in \calD_{\asym /\pm}\) and put \(\alpha :=\begin{bmatrix}1 \\ g\end{bmatrix}\) and \(\alpha' :=\begin{bmatrix}1 \\ g^{-1}\end{bmatrix}\).

By \cite[Theorem 5.5 (i)]{MR3505131}, 
the character \(\zeta_{\Tam, \alpha} \colon F_{\alpha}^{\times}=(E \cdot g(E))^{\times} \to \bbC^{\times}\) satisfies
\[
\begin{cases}
\text{\(\zeta_{\Tam, \alpha}\) is trivial on \(1+\frakp_{F_{\alpha}}\),} \\
\zeta_{\Tam, \alpha}|_{\mu_{F_{\alpha}}}
=t^1_{\bmu_g}(\bvmod{A}{[g]})\cdot t^1_{\bmu_g}(\bvmod{A^{\ast}}{[g]}).
\end{cases}
\]
Note that while the equality 
\[
\zeta_{\Tam, \alpha}(\varpi_E) \cdot \zeta_{\Tam, \alpha'} (\varpi_E)=t_{\lan \varpi \ran}(\bvmod{A}{[g]}) \cdot t_{\lan \varpi \ran}(\bvmod{A^{\ast}}{[g]})
\]
in \cite[Theorem 5.5 (i) (b)]{MR3505131} may appear to impose some condition, in fact \(\zeta_{\Tam, \alpha}(\varpi_E)\) can be any value.
Indeed, we have 
\[
\alpha'
=g^{-1} \cdot \begin{bmatrix}g \\ 1\end{bmatrix} 
=g^{-1} \cdot (-\alpha), 
\quad
\zeta_{\Tam, \alpha'}=\zeta_{\Tam, g^{-1} \cdot (-\alpha)} 
=\zeta_{\Tam, -\alpha} \circ g 
=\zeta_{\Tam, \alpha}^{-1} \circ g, 
\]
and thus 
\begin{equation} \label{eq:prod-of-zeta}
\zeta_{\Tam, \alpha}(z) \cdot \zeta_{\Tam, \alpha'}(z)=\zeta_{\Tam, \alpha} (z) \cdot \zeta_{\Tam, \alpha}(g(z))^{-1}=\zeta_{\Tam, \alpha} (\alpha (z))
\end{equation}
for any \(z\in E^{\times} (\subset F_{\alpha}^{\times} \cap F_{\alpha'}^{\times})\).
In particular, as \(\alpha (\varpi_E)\) lies in \(\mu_{F_{\alpha}}\), 
the value \(\zeta_{\Tam, \alpha}(\varpi_E) \cdot \zeta_{\Tam, \alpha'} (\varpi_E)\) depends only on the restriction \(\zeta_{\Tam, \alpha}|_{\mu_{F_{\alpha}}}\) and is indeed equal to 
\(t_{\lan \varpi \ran}(\bvmod{A}{[g]}) \cdot t_{\lan \varpi \ran}(\bvmod{A^{\ast}}{[g]})\) as computed in \cite[(5.3)]{MR3505131}.
We later use a more concrete expression
\[
\zeta_{\Tam, \alpha}(\varpi_E) \cdot \zeta_{\Tam, \alpha'} (\varpi_E)
=\sgn_{\alpha(\varpi_E)}(\vmod{A}{[g]}) \cdot \sgn_{\alpha(\varpi_E)}(\vmod{A^{\ast}}{[g]})
\]
in the same equation \cite[(5.3)]{MR3505131}.
Here, the notation for the right-hand side is as follows:
\begin{itemize}
\item for a finite module \(\frakV\) and its automorphism \(x\), we write \(\sgn_{x}(\frakV)\) for the signature of \(x\) as a permutation of the set \(\frakV\) and
\item if the module is isomorphic to \(\frakU_{[g]}\) (resp.\ is trivial), then we consider the multiplication action of \(\alpha (\varpi_E) \in \mu_{F_{\alpha}} \simeq k_{F_{\alpha}}^{\times}\) on \(\frakU_{[g]} \simeq k_{F_{\alpha}}\) (resp.\ the trivial action).
\end{itemize}

Similarly, the equality
\[
\zeta_{\Tam, \alpha}(z) \cdot \zeta_{\Tam, \alpha'} (z)=\sgn_{\alpha(z)}(\vmod{A}{[g]}) \cdot \sgn_{\alpha(z)}(\vmod{A^{\ast}}{[g]})
\]
for \(z\in \mu_E\) is in the equation preceding \cite[(5.3)]{MR3505131};
there, \(\sgn_{\mu_E}(\vmod{A}{[g]})\) (resp.\ \(\sgn_{\mu_E}(\vmod{A^{\ast}}{[g]})\)) denotes the character \(\mu_E \to \{ \pm 1\}\) sending
\(x \in \mu_E\) to \(\sgn_{x}(\vmod{A}{[g]}) \in \{ \pm 1\}\) (resp.\ \(\sgn_{x}(\vmod{A^{\ast}}{[g]}) \in \{ \pm 1\}\)).

Also, it is immediate from \eqref{eq:prod-of-zeta} that \(\zeta_{\Tam, \alpha}|_{E^{\times}} \cdot \zeta_{\Tam, \alpha'}|_{E^{\times}}\) is trivial on \(1+\frakp_E\).

In summary, the product \(\zeta_{\Tam, \alpha}|_{E^{\times}} \cdot \zeta_{\Tam, \alpha'}|_{E^{\times}}\) satisfies 
\begin{equation} \label{eq:zeta-asymm}
\begin{cases}
\text{\(\zeta_{\Tam, \alpha}|_{E^{\times}} \cdot \zeta_{\Tam, \alpha'}|_{E^{\times}}\) is trivial on \(1+\frakp_E\),} \\
(\zeta_{\Tam, \alpha}|_{E^{\times}} \cdot \zeta_{\Tam, \alpha'}|_{E^{\times}}) (z)=\sgn_{\alpha(z)}(\vmod{A}{[g]}) \cdot \sgn_{\alpha(z)}(\vmod{A^{\ast}}{[g]}) \text{ for \(z \in \mu_E \cup \{ \varpi_E\}\)}.
\end{cases}
\end{equation}

\subsubsection{Symmetric roots} \label{sec:review-Tam-symm}
Let \(g\in \calD_{\sym}\) and put \(\alpha :=\begin{bmatrix}1 \\ g\end{bmatrix}\).

By \cite[Theorem 5.5 (i)]{MR3505131}, we have
\[
\begin{cases}
\text{\(\zeta_{\Tam, \alpha}\) is trivial on \(1+\frakp_{F_{\alpha}}\),} \\
\text{\(\zeta_{\Tam, \alpha}|_{\mu_{F_{\alpha}}}
=t^1_{\bmu_g}(\vmod{A}{[g]})\cdot t^1_{\bmu_g}(\vmod{A^{\ast}}{[g]})\), and} \\
\text{\(\zeta_{\Tam, \alpha}(\varpi_E)
=\iota_g \cdot t_{\lan \varpi \ran}(\vmod{A}{[g]})\cdot t_{\lan \varpi \ran}(\vmod{A^{\ast}}{[g]})\), where}
\end{cases}
\]	
\[
\iota_g=
\begin{cases}
1 &\text{if \(g \in (\calD_{\sym, \ram}-\{\sigma ^{\frac{e}{2}}\}) \cup W_{F[\varpi_E]}\)} \\
(-1)^m &\text{if \(g\in (\calD_{\sym, \ur}-W_{F[\varpi_E]}) \cup \{ \sigma^{\frac{e}{2}}\}\).}
\end{cases} 
\]
Let us slightly rewrite \(\iota_g\) by dividing the cases in a different way.
By Remark \ref{rem:Tam-sym-ur/ram} and Proposition \ref{prop:cond-by-Galois-elt} \eqref{item:cond-for-symram},
we see that if \(\alpha\) is symmetric unramified (in the sense recalled in Section \ref{sec:chi-data-L-emb}), then
\begin{equation} \label{eq:iota_g-symur}
\iota_g=
\begin{cases}
1 &\text{if \(g=\sigma^i\) for some \(i\) or \(g(\varpi_E)=\varpi_E\)} \\
(-1)^m &\text{otherwise}
\end{cases} 
\end{equation}
and that if \(\alpha\) is symmetric ramified (in the sense recalled in Section \ref{sec:chi-data-L-emb}), then
\begin{equation} \label{eq:iota_g-symram}
\iota_g=(-1)^m.
\end{equation}

Later we only need the restriction \(\zeta_{\Tam, \alpha}|_{E^{\times}}\) to \(E^{\times}\).
By Proposition \ref{prop:res-of-t^1}, it satisfies
\begin{equation} \label{eq:zeta-sym}
\begin{cases}
\text{\(\zeta_{\Tam, \alpha}|_{E^{\times}}\) is trivial on \(1+\frakp_{E}\),} \\
\text{\(\zeta_{\Tam, \alpha}|_{\mu_E}
	=t^1_{\bmu}(\vmod{A}{[g]})\cdot t^1_{\bmu}(\vmod{A^{\ast}}{[g]})\), and} \\
\text{\(\zeta_{\Tam, \alpha}(\varpi_E)
	=\iota_g \cdot t_{\lan \varpi \ran}(\vmod{A}{[g]})\cdot t_{\lan \varpi \ran}(\vmod{A^{\ast}}{[g]})\).}
\end{cases}
\end{equation}

\section{Proof of the main theorem} \label{sec:pf-main-thm}
Let us finally prove Theorem \ref{thm:main}.

Let \(\pi\) be a regular supercuspidal representation of \(\bfG^{\ast}(F)=\GL_n(F)\) as in Theorem \ref{thm:main}.
There exists a tame elliptic regular pair \((\bfS, \xi)\) of \(\bfG^{\ast}\) such that \(\pi \simeq {}_{\bfG^{\ast}}\pi_{(\bfS, \xi)}\).
Let \(E/F\) be a tamely ramified extension of degree \(n\) such that \(\bfS \simeq \Res_{E/F} \Gm\), so that \((\bfS, \xi)\) is identified with an admissible pair \((E/F, \xi)\).
As reviewed in Section \ref{sec:Tam-thm}, we have \(\chi_{\Tam}\) associated to \((E/F, \xi)\) thanks to \cite{MR3505131}.
Similarly, Kaletha's theory reviewed in Section \ref{sec:Kal-LLC} (and Section \ref{sec:LLC^Kal-for-GL_n}) yields \(\eps{\bfG^{\ast}}{}\) associated to \((\bfS, \xi)\).
By Corollary \ref{cor:Tam} and Proposition \ref{prop:LLC^Kal-explicit},
we have
\[
\LJLC^{\cl}(\pi)={}_{\bfG}\pi_{(\bfS, \nu_{\zeta_{\Tam}}^{-1} \cdot \xi)},
\quad
\LJLC^{\Kal}(\pi)={}_{\bfG}\pi_{(\bfS, \xi \cdot \eps{\bfG^{\ast}}{}^{-1} \cdot \eps{\bfG}{})}.
\]
As we have \(\eps{\bfG^{\ast}}{}=\eps{\bfG^{\ast}}{f, \ram} \cdot \eps{\bfG^{\ast}}{}^{\ram}\) and a similar equality for \(\eps{\bfG}{}\) by definition (see \eqref{eq:def-of-epsilon} and Section \ref{sec:LLC^Kal-for-GL_n}), it suffices to show an equality of characters of \(E^{\times}\):
\[
\nu_{\zeta_{\Tam}}^{-1} =\eps{\bfG^{\ast}}{f, \ram}^{-1} \cdot \eps{\bfG}{f, \ram} \cdot (\eps{\bfG^{\ast}}{}^{\ram})^{-1} \cdot \eps{\bfG}{}^{\ram}.
\]
As all the characters in this equality are quadratic, we can and do omit the inverse signs throughout.
In view of Definition \ref{def:nu_zeta}, \eqref{eq:factorization-epsilon^ram}, \eqref{eq:factorization-epsilon_fram},
we can work with each orbit of roots separately and will show the following:
\begin{description}
\item[Asymmetric roots]
For \(\alpha =\begin{bmatrix}1 \\ g\end{bmatrix}\in \Phi(\bfS, \bfG^{\ast})^{\sym}\)  with \(g\in \calD_{\asym /\pm}\),
we have
\begin{equation}\label{eq:asym-comp}
\zeta_{\Tam, \alpha}|_{E^{\times}} \cdot \zeta_{\Tam, \alpha'}|_{E^{\times}}=\eps{\bfG^{\ast}}{\alpha} \cdot \eps{\bfG}{\alpha},
\end{equation}
where \(\alpha' :=\begin{bmatrix}1 \\ g^{-1}\end{bmatrix}\) as in Section \ref{sec:review-Tam-asymm}.

\item[Symmetric unramified roots]
For \(\alpha =\begin{bmatrix}1 \\ g\end{bmatrix}\in \Phi(\bfS, \bfG^{\ast})_{\sym, \ur}\)  with \(g\in \calD_{\sym}\),
we have
\begin{equation}\label{eq:symur-comp}
\zeta_{\Tam, \alpha} |_{E^{\times}}=\eps{\bfG^{\ast}}{\alpha} \cdot \eps{\bfG}{\alpha}.
\end{equation}

\item[Symmetric ramified roots]
For \(\alpha =\begin{bmatrix}1 \\ g\end{bmatrix}\in \Phi(\bfS, \bfG^{\ast})_{\sym, \ram}\)  with \(g\in \calD_{\sym}\),
we have
\begin{equation}\label{eq:symram-comp}
\zeta_{\Tam, \alpha} |_{E^{\times}}=\eps{\bfG^{\ast}}{\alpha} \cdot \eps{\bfG}{\alpha}.
\end{equation}
\end{description}

\subsection{Asymmetric roots} \label{sec:asymm}
Let \(\alpha =\begin{bmatrix}1 \\ g\end{bmatrix}\in \Phi(\bfS, \bfG^{\ast})^{\sym}\)  such that \(g\in \calD_{\asym /\pm}\)
and let us prove \eqref{eq:asym-comp}.
Put \(\alpha' :=\begin{bmatrix}1 \\ g^{-1}\end{bmatrix}\).
Let \(-1 \leq k\leq t-1\) be the unique integer such that \(\alpha \in \Phi(\bfS, \bfG^{\ast, k+1})-\Phi(\bfS, \bfG^{\ast, k})\).

The following argument is very similar to the one in \cite[\S 6.2]{MR4206603}.

Recall that by \eqref{eq:zeta-asymm} we have
\[
\begin{cases}
\text{\(\zeta_{\Tam, \alpha}|_{E^{\times}} \cdot \zeta_{\Tam, \alpha'}|_{E^{\times}}\) is trivial on \(1+\frakp_E\),} \\
(\zeta_{\Tam, \alpha}|_{E^{\times}} \cdot \zeta_{\Tam, \alpha'}|_{E^{\times}}) (z)=\sgn_{\alpha(z)}(\vmod{A}{[g]}) \cdot \sgn_{\alpha(z)}(\vmod{A^{\ast}}{[g]}) \text{ for \(z \in \mu_E \cup \{ \varpi_E\}\)}.
\end{cases}
\]
On the other hand, by \eqref{eq:epsilon-asym} we have
\[
\eps{\bfG}{\alpha}(z)=
\begin{cases}
\Jac{\overline{\alpha(z)}}{k_{F_{\alpha}}^{\times}} &\text{if \(k\neq -1\) and \(\frac{r_k}{2} \in \ordx{\bfG}{\alpha}\)} \\
1 & \text{otherwise}
\end{cases}
\]
for \(z \in \bfS(F)=E^{\times}\) and a similar equality for \(\eps{\bfG^{\ast}}{\alpha}\).

Since both \(\zeta_{\Tam, \alpha}|_{E^{\times}} \cdot \zeta_{\Tam, \alpha'}|_{E^{\times}}\) and \(\eps{\bfG^{\ast}}{\alpha} \cdot \eps{\bfG}{\alpha}\) are trivial on \(1+\frakp_{E}\), it suffices to show
\[
\sgn_{\alpha(z)}(\vmod{A}{[g]})=\eps{\bfG}{\alpha}(z), \quad \sgn_{\alpha(z)}(\vmod{A^{\ast}}{[g]})=\eps{\bfG^{\ast}}{\alpha}(z)
\]
for \(z\in \mu_E \cup \{\varpi_E\}\).
As the proof for the latter is essentially the same, we only show the former equality.

Assume that \(\vmod{A}{[g]}\) is trivial.
Then \(\sgn_{\alpha(z)}(\vmod{A}{[g]})=1\).
Also by Proposition \ref{prop:vmod-ord} we have \(k=-1\) or \(\frac{r_k}{2}\notin \ordx{\bfG}{\alpha}\) and in both cases \(\eps{\bfG}{\alpha}(z)=1\).

We now assume that \(\vmod{A}{[g]}\) is non-trivial and hence isomorphic to \(\frakU_{[g]}\).
Again by Proposition \ref{prop:vmod-ord} we have \(k\neq -1\) and \(\frac{r_k}{2}\in \ordx{\bfG}{\alpha}\).
Thus in this case, we are to show
\[
\sgn_{z}(\frakU_{[g]})=\Jac{z}{k_{F_{\alpha}}^{\times}}
\]
for \(z\in \mu_{F_{\alpha}}\).
This essentially follows from the last paragraph of \cite[\S 6.2]{MR4206603} showing 
\begin{equation} \label{eq:OT-Sec6.2}
\sgn_{z}(\vmod{A^{\ast}}{[g]})=\Jac{z}{k_{F_{\alpha}}^{\times}}
\end{equation}
for \(z\in \mu_{F_{\alpha}}\) in the case where \(\vmod{A^{\ast}}{[g]}\) is non-trivial
(note that \(\vmod{A^{\ast}}{[g]}\) in this paper is denoted simply by \(\frakV_{[g]}\) in \cite{MR4206603} because only \(\bfG^{\ast}\) is considered).
Even though \(\vmod{A^{\ast}}{[g]}\) may be trivial, the proof of \eqref{eq:OT-Sec6.2} clearly works for \(\frakU_{[g]}\) in place of \(\vmod{A^{\ast}}{[g]}\).

\subsection{Symmetric unramified roots}
Let \(\alpha =\begin{bmatrix}1 \\ g\end{bmatrix}\in \Phi(\bfS, \bfG^{\ast})_{\sym, \ur}\) such that \(g\in \calD_{\sym}\) and let us prove \eqref{eq:symur-comp}.
Let \(-1 \leq k\leq t-1\) be the unique integer such that \(\alpha \in \Phi(\bfS, \bfG^{\ast, k+1})-\Phi(\bfS, \bfG^{\ast, k})\).

Recall that by \eqref{eq:zeta-sym} and \eqref{eq:iota_g-symur} we have
\begin{itemize}
\item \(\zeta_{\Tam, \alpha}|_{E^{\times}}\) is trivial on \(1+\frakp_{E}\),
\item \(\zeta_{\Tam, \alpha}|_{\mu_E}
=t^1_{\bmu}(\vmod{A}{[g]})\cdot t^1_{\bmu}(\vmod{A^{\ast}}{[g]})\), and
\item \(\zeta_{\Tam, \alpha}(\varpi_E)
=\iota_g \cdot t_{\lan \varpi \ran}(\vmod{A}{[g]})\cdot t_{\lan \varpi \ran}(\vmod{A^{\ast}}{[g]})\), where
\end{itemize}
\[
\iota_g=
\begin{cases}
1 &\text{if \(g=\sigma^i\) for some \(i\) or \(g(\varpi_E)=\varpi_E\)} \\
(-1)^m &\text{otherwise.}
\end{cases}
\]
On the other hand, by \eqref{eq:epsilon-symur} we have
\[
\eps{\bfG}{\alpha}(z):=
\begin{cases}
\Jac{\overline{\alpha(z)}}{k_{F_{\alpha}}^1} &\text{if \(k\neq -1\) and \(\frac{r_k}{2} \in \ord_x(\alpha)\)} \\
1 & \text{otherwise}
\end{cases}
\]
for \(z \in \bfS(F)=E^{\times}\) and a similar equality for \(\eps{\bfG^{\ast}}{\alpha}\).

We first assume that \(\vmod{A}{[g]} \simeq \vmod{A^{\ast}}{[g]}\).
By Proposition \ref{prop:symp-mods} this means 
that \([g] =[\sigma^i]\) for some \(i\) or \(m\) is even.
In both cases we have \(\iota_g=1\).
As \(t^1_{\bmu}(\cdot)\) is quadratic and \(t_{\lan \varpi\ran}(\cdot ) \in \{ \pm 1\}\),
we see that \(\zeta_{\Tam, \alpha}|_{E^{\times}}\) is trivial in this case.
On the other hand, by Proposition \ref{prop:vmod-ord}
we have \(\eps{\bfG^{\ast}}{\alpha}=\eps{\bfG}{\alpha}\) 
and hence \(\eps{\bfG^{\ast}}{\alpha} \cdot \eps{\bfG}{\alpha}\) is trivial and 
\eqref{eq:symur-comp} follows in this case.

Next we assume that \(\vmod{A}{[g]} \not\simeq \vmod{A^{\ast}}{[g]}\)
and hence \(\vmod{A}{[g]} \oplus \vmod{A^{\ast}}{[g]} \simeq \frakU_{[g]}\).
Since both \(\zeta_{\Tam, \alpha}|_{E^{\times}}\) and \(\eps{\bfG^{\ast}}{\alpha} \cdot \eps{\bfG}{\alpha}\) are trivial on \(1+\frakp_{E}\), we need to show
\[
\zeta_{\Tam, \alpha}|_{\mu_E}=(\eps{\bfG^{\ast}}{\alpha} \cdot \eps{\bfG}{\alpha})|_{\mu_E},
\quad
\zeta_{\Tam, \alpha}(\varpi_E)=(\eps{\bfG^{\ast}}{\alpha} \cdot \eps{\bfG}{\alpha})(\varpi_E).
\]
Moreover, by the additivity of t-factors we have
\begin{gather*}
t^1_{\bmu}(\vmod{A}{[g]})\cdot t^1_{\bmu}(\vmod{A^{\ast}}{[g]})=t^1_{\bmu}(\frakU_{[g]}),
\quad
t_{\lan \varpi \ran}(\vmod{A}{[g]})\cdot t_{\lan \varpi \ran}(\vmod{A^{\ast}}{[g]})=t_{\lan \varpi \ran}(\frakU_{[g]}).
\end{gather*}
By Proposition \ref{prop:symp-mods} (and Proposition \ref{prop:cond-by-Galois-elt} \eqref{item:cond-for-sym})
we have \([g] =[\sigma^i\phi^{\frac{f}{2}}]\) for some \(i\) (and in particular \(f\) is even) and \(m\) is odd.
Thus the equation for \(\iota_g\) translates into
\begin{equation}
\label{eq:iota_g-symur2}
\iota_g=
\begin{cases*}
1 &\text{if \(g(\varpi_E)=\varpi_E\)} \\
-1&\text{otherwise.}
\end{cases*}
\end{equation}
On the other hand, by Proposition \ref{prop:vmod-ord}
we have
\[
(\eps{\bfG^{\ast}}{\alpha} \cdot \eps{\bfG}{\alpha})(\gamma)
=\Jac{\alpha(\gamma)}{k_{F_{\alpha}}^1}
\]
for \(\gamma \in E^{\times}\).
Therefore, we must show 
\begin{align}
\label{eq:symur-mu} t^1_{\bmu}(\frakU_{[g]})(\gamma)&=\Jac{\alpha(\gamma)}{k_{F_{\alpha}}^1} \text{ for } \gamma \in \mu_E, \\
\label{eq:symur-varpi} \iota_g \cdot t_{\lan \varpi \ran}(\frakU_{[g]})&=\Jac{\alpha(\varpi_E)}{k_{F_{\alpha}}^1}.
\end{align}
As we explain below, both \eqref{eq:symur-mu} and \eqref{eq:symur-varpi} follow from the proofs of related propositions in \cite[\S 6.2]{MR4206603}, which in turn are essentially carried out by combining Tam's computation of t-factors and elementary manipulations of quadratic characters (cf. a less elaborate but similar argument in Section \ref{sec:asymm}).

The equation \eqref{eq:symur-mu} follows from the proof of \cite[Proposition 6.3]{MR4206603}.
Indeed, the proposition asserts the equality \(\chi_{\Tam, \alpha}^{-1}|_{\mu_E}=\eps{\bfG^{\ast}}{\alpha}^{-1}|_{\mu_E} \cdot \chi_{\Kal, \alpha}|_{\mu_E}\) (as with \(\vmod{A^{\ast}}{[g]}\), the character \(\eps{\bfG^{\ast}}{\alpha}\) in this paper is denoted simply by \(\epsilon_{\alpha}\) in \cite{MR4206603}).
As in \cite[page 2050, (3)]{MR4206603}, this amounts to 
\[
\eps{\bfG^{\ast}}{\alpha}|_{\mu_E}=\chi_{\Tam, \alpha}|_{\mu_E}
\]
by the triviality of \(\chi_{\Kal, \alpha}|_{\mu_E}\).
If \(\vmod{A^{\ast}}{[g]}\) is non-trivial and thus isomorphic to \(\frakU_{[g]}\), 
we have
\[
\eps{\bfG^{\ast}}{\alpha}(z)=\Jac{\alpha(z)}{k_{F_{\alpha}}^1}
\text{ and }
\chi_{\Tam, \alpha}(z)=t^1_{\bmu}(\vmod{A^{\ast}}{[g]})(z)
\]
for \(z\in \mu_E\)
(the latter is part of \cite[Theorem 7.1 (i)]{MR3509939} 
and is in the displayed equation immediately after (3) in \cite[page 2050]{MR4206603}).
Thus, the proof of \cite[Proposition 6.3]{MR4206603} is mainly devoted to showing 
\begin{equation} \label{eq:OT-prop6.3}
\Jac{\alpha(z)}{k_{F_{\alpha}}^1}=t^1_{\bmu}(\vmod{A^{\ast}}{[g]})(z)
\end{equation}
for \(z\in \mu_E\) in the case where \(\vmod{A^{\ast}}{[g]}\) is non-trivial.
Even though \(\vmod{A^{\ast}}{[g]}\) may be trivial, the proof of \eqref{eq:OT-prop6.3} clearly works for \(\frakU_{[g]}\) in place of \(\vmod{A^{\ast}}{[g]}\) and \eqref{eq:symur-mu} follows.
(Note that the proof in the relevant case (i.e., \(g=\sigma^i \phi^{\frac{f}{2}}\)) is given in the last paragraph of the proof of \cite[Proposition 6.3]{MR4206603}.)

Similarly, the equation \eqref{eq:symur-varpi} follows from the proof of \cite[Proposition 6.4]{MR4206603}.
Indeed, the proposition asserts that \(\chi_{\Tam, \alpha}^{-1}(\varpi_E)=\eps{\bfG^{\ast}}{\alpha}^{-1}(\varpi_E) \cdot \chi_{\Kal, \alpha}(\varpi_E)\).
As in the first displayed equation in the proof of \cite[Proposition 6.4]{MR4206603}, this amounts to
\[
(\eps{\bfG^{\ast}}{\alpha} \cdot \chi_{\Tam, \alpha}^{-1})(\varpi_E)=-1
\]
because \(\chi_{\Kal, \alpha}(\varpi_E)=-1\).
If \(\vmod{A^{\ast}}{[g]}\) is non-trivial and thus isomorphic to \(\frakU_{[g]}\),
then, as shown in the second and third displayed equations in the proof of \cite[Proposition 6.4]{MR4206603}, we have
\begin{align*}
\eps{\bfG^{\ast}}{\alpha}(\varpi_E)&=\Jac{\alpha(\varpi_E)}{k_{F_{\alpha}}^1}, \\
\chi_{\Tam, \alpha}(\varpi_E)&=t_{\bmu}^0(\vmod{A^{\ast}}{[g]}^{\lan \varpi\ran}) \cdot t_{\lan \varpi \ran}^0(\vmod{A^{\ast}}{[g]}) \cdot t_{\lan \varpi \ran}^1(\vmod{A^{\ast}}{[g]}) \nonumber \\
&=t_{\bmu}^0(\vmod{A^{\ast}}{[g]}^{\lan \varpi\ran}) \cdot t_{\lan \varpi \ran}(\vmod{A^{\ast}}{[g]}),
\end{align*}
where \(\vmod{A^{\ast}}{[g]}^{\lan \varpi\ran}\) denotes the \(\lan \varpi \ran\)-fixed part of \(\vmod{A^{\ast}}{[g]}\) and we have \(t_{\lan \varpi \ran}(\vmod{A^{\ast}}{[g]})=t_{\lan \varpi \ran}^0(\vmod{A^{\ast}}{[g]}) \cdot t_{\lan \varpi \ran}^1(\vmod{A^{\ast}}{[g]})\) by the definition of \(t_{\lan \varpi \ran}(\cdot)\).
(Note that \(\lan \varpi \ran\) in this paper and \cite{MR3505131} is denoted by \(\varpi\) in \cite{MR4206603} and \cite{MR3509939}.)
Thus, the main part of the proof of \cite[Proposition 6.4]{MR4206603} is to establish
\begin{equation} \label{eq:OT-prop6.4}
\Jac{\alpha(\varpi_E)}{k_{F_{\alpha}}^1}=-t_{\bmu}^0(\vmod{A^{\ast}}{[g]}^{\lan \varpi\ran}) \cdot t_{\lan \varpi \ran}(\vmod{A^{\ast}}{[g]})
\end{equation}
in the case where \(\vmod{A^{\ast}}{[g]}\) is non-trivial.
Again the proof of \eqref{eq:OT-prop6.4} clearly works for \(\frakU_{[g]}\) in place of \(\vmod{A^{\ast}}{[g]}\)
and hence we obtain
\begin{equation}
\label{eq:OT-prop6.4'}
\Jac{\alpha(\varpi_E)}{k_{F_{\alpha}}^1}=-t_{\bmu}^0(\frakU_{[g]}^{\lan \varpi\ran}) \cdot t_{\lan \varpi \ran}(\frakU_{[g]}).
\end{equation}
Moreover, as in the second displayed equation of \cite[page 2052]{MR4206603}, we have
\[
t_{\bmu}^0(\vmod{A^{\ast}}{[g]}^{\lan \varpi\ran})
=\begin{cases}
-1 & \text{if \(\beta:=\alpha (\varpi_E)=1\)} \\
1 & \text{otherwise.}
\end{cases}
\]
However, yet again, the proof of this equality in \cite[Remark 7.3]{MR3509939} is valid for \(\frakU_{[g]}\) in place of \(\vmod{A^{\ast}}{[g]}\).
Comparing the resulting equation with \eqref{eq:iota_g-symur2}, we obtain
\[
t_{\bmu}^0(\frakU_{[g]}^{\lan \varpi\ran})=-\iota_g.
\]
Combining this with \eqref{eq:OT-prop6.4'}, we finally deduce \eqref{eq:symur-varpi}.

The proof of \(\eqref{eq:symur-comp}\) is complete.

\subsection{Symmetric ramified roots}
Let \(\alpha =\begin{bmatrix}1 \\ g\end{bmatrix}\in \Phi(\bfS, \bfG^{\ast})_{\sym, \ram}\) such that \(g\in \calD_{\sym}\)
and let us prove \eqref{eq:symram-comp}.
By Proposition \ref{prop:cond-by-Galois-elt} \eqref{item:cond-for-symram}, this case occurs if and only if \(e\) is even, in which case the orbit of \(\alpha\) is represented by \(\begin{bmatrix}1 \\ \sigma ^{\frac{e}{2}}\end{bmatrix}\).
Thus we assume that \(e\) is even and \(g=\sigma ^{\frac{e}{2}}\).

Recall that by \eqref{eq:zeta-sym} and \eqref{eq:iota_g-symram} we have
\begin{itemize}
\item \(\zeta_{\Tam, \alpha}|_{E^{\times}}\) is trivial on \(1+\frakp_{E}\),
\item \(\zeta_{\Tam, \alpha}|_{\mu_E}
=t^1_{\bmu}(\vmod{A}{[g]})\cdot t^1_{\bmu}(\vmod{A^{\ast}}{[g]})\), and
\item \(\zeta_{\Tam, \alpha}(\varpi_E)
=(-1)^m \cdot t_{\lan \varpi \ran}(\vmod{A}{[g]})\cdot t_{\lan \varpi \ran}(\vmod{A^{\ast}}{[g]})\).
\end{itemize}
On the other hand, by Proposition \ref{prop:toral-invariant} \eqref{item:sign-symram}, we have
\[
(\eps{\bfG^{\ast}}{\alpha} \cdot \eps{\bfG}{\alpha})(\gamma)=(-1)^{mv_E(\gamma)}
\]
for \(\gamma \in \bfS(F)=E^{\times}\).

By Proposition \ref{prop:symp-mods}, the two modules \(\vmod{A^{\ast}}{[g]}\) and \(\vmod{A}{[g]}\) are isomorphic.
This implies that \(\zeta_{\Tam, \alpha}|_{E^{\times}}\) is trivial on \(\calO_E^{\times}\) and \(\zeta_{\Tam, \alpha}(\varpi_E)=(-1)^m\).
Thus \eqref{eq:symram-comp} follows.

\appendix
\section{Comparison of the two constructions of supercuspidal representations of inner forms of \(\GL_n\)} \label{sec:app}
In this appendix, we sketch a proof of Proposition \ref{prop:BH-KY}.
As the case of \(\bfG^{\ast}=\GL_n\) is treated in \cite[Appendix A]{MR4206603}, we focus on \(\bfG=\underline{A}^{\times}\).
In fact, the proof for \(\bfG\) is mostly the same as that for \(\bfG^{\ast}\).
Thus, we largely follow \cite[Appendix A]{MR4206603}, while trying to highlight the points where some extra care is required.

Let \((\bfS, \xi)\) be a tame elliptic regular pair of \(\bfG\) and \((E/F, \xi)\) be an admissible pair of degree \(n\) such that the \(\bfG(F)\)-conjugacy class of \((\bfS, \xi)\) and the \(F\)-isomorphism class of \((E/F, \xi)\) correspond under the bijections in Section \ref{sec:ter-and-adm-pairs}.
Our goal is to prove that \({}_{\bfG}\pi_{(E/F, \xi)}^{\BH}\) is isomorphic to \({}_{\bfG}\pi_{(\bfS, \xi)}^{\KY}\).

First we review Bushnell--Henniart's construction of \({}_{\bfG}\pi_{(E/F, \xi)}^{\BH}\)
in \cite{MR2848585} in some details.
The parametrization 
\[
\{\text{admissible pairs of degree \(n\)}\}/{\text{\(F\)-isom.}} \xrightarrow{1:1} \calA_m^{\et}(D); \quad (E/F, \xi) \mapsto {}_{\bfG}\pi_{(E/F, \xi)}^{\BH}
\]
is a generalization of the parametrization of \(\calA_n^{\et}(F)\) in \cite{MR2138141}.
However, due to ``some novel technical difficulties" (\cite[page 472]{MR2848585}), the map constructed in \cite{MR2848585}
is in fact backward;
the bijection
\[
\calA_m^{\et}(D) \xrightarrow{1:1} \{\text{admissible pairs of degree \(n\)}\}/{\text{\(F\)-isom.}}
\]
is defined in \cite[4 and 6.1 Parametrization Theorem]{MR2848585}.
Here we essentially work out the inverse map, making explicit the various subgroups and the character \(\theta\) in the intermediate steps.

\begin{description}
\item[Factorization of \(\xi\)]
fix a uniformizer \(\varpi_F \in F\).
We take a factorization \(\xi=\xi_{\mathrm{t}} \cdot \xi_{\w}\) satisfying the following:
\begin{itemize}
\item The character \(\xi_{\mathrm{t}}\) is tamely ramified. In other words, \(\xi_{\w}|_{U_E^1}=\xi|_{U_E^1}\).
\item We have \(\varpi_F \in \Ker \xi_{\w}\). In particular, \(\xi_{\w}\) has finite order.
\item The character \(\xi_{\w}\) has \(p\)-power order. In particular, \(\mu_E \subset \Ker \xi_{\w}\).
\end{itemize}
It is easy to check that 
such a factorization uniquely exists.

In the following, using \(\xi_{\mathrm{t}}, \xi_{\w}\),
we are going to construct irreducible smooth representations \(\Lambda_{\mathrm{t}}, \Lambda_{\w}\) of an open subgroup \(\bfJ \subset \bfG(F)\).
For the definitions of \(\bfJ\) and related open subgroups as well as the definition of a character \(\theta\), we follow \cite[\S 3.4, \S 3.5]{MR3505131} which are more explicit than \cite{MR2848585}.

\item[Definition of open subgroups]
to proceed further, we need to fix an \(F\)-embedding \(E \hookrightarrow A\).
We take an \(F\)-embedding that induces the inclusion \(\bfS \simeq \Res_{E/F} \Gm \hookrightarrow \underline{A}^{\times}=\bfG\).
While the resulting representation \({}_{\bfG}\pi_{(E/F, \xi)}^{\BH} \) will be independent of the choice of an \(F\)-embedding (up to isomorphism), we make this choice for later convenience.

Recall from Section \ref{sec:vmod} various objects attached to
the situation. 
The admissible pair \((E/F, \xi)\) and the \(F\)-embedding \(E \hookrightarrow A\) define a sequence of subfields and integers:
\[
F=E_t \subsetneq E_{t-1} \subsetneq \cdots \subsetneq E_0 \subset E_{-1}=E, \quad 
a_t \geq a_{t-1} > \cdots > a_1 > a_0 > a_{-1}=0.
\]
Moreover, together with the \(F\)-embedding \(E \hookrightarrow A\), it further induces \(E_k\)-subalgebras \(A_k \subset A\) and \(\calO_E\)-orders \(\frakA_{k} \subset A_k\) for \(-1\leq k\leq t\) as well as compact open subgroups of \(A^{\times}\):
\begin{align*}
J^1&:=U_{\frakA_0}^1U_{\frakA_1}^{\frac{1}{2}a_0e(\frakA_1/\calO_E)} \cdots U_{\frakA_{t-1}}^{\frac{1}{2}a_{t-2}e(\frakA_{t-1}/\calO_E)}U_{\frakA}^{\frac{1}{2}a_{t-1}e(\frakA/\calO_E)}, \\
H^1&:=U_{\frakA_0}^1U_{\frakA_1}^{\frac{1}{2}a_0e(\frakA_1/\calO_E)+} \cdots U_{\frakA_{t-1}}^{\frac{1}{2}a_{t-2}e(\frakA_{t-1}/\calO_E)+}U_{\frakA}^{\frac{1}{2}a_{t-1}e(\frakA/\calO_E)+}.
\end{align*}
We furthermore need two more open subgroups:
\[
J^0:=U_{\frakA_0}J^1, \quad \bfJ:=E^{\times}J^0=E_0^{\times}J^0.
\]
Here, the last equality holds because
\(E/E_0\) is unramified by an axiom of an admissible pair.
To ease the exposition we also fix a concrete expression for \(A_0\) and \(\frakA_0\).
As \(r(\frakA_0)\) divides \(e(E/E_0)=1\) by \eqref{eq:r(frakA)},
we see that \(\frakA_0 \subset A_0\) is a maximal \(\calO_{E_0}\)-order.
Abbreviating \(m(A_0)\) in \eqref{eq:m_k} to \(l\),
we see that there exists an isomorphism
\[
A_0 \simeq M_l(D_0)
\]
of \(E_0\)-algebras inducing an isomorphism between \(\frakA_0\) and \(M_l(\calO_{D_0})\).
From now on, we fix such an isomorphism.
Then we have
\[
\frakK_{\frakA_0}=D_0^{\times}U_{\frakA_0}.
\]

Note that all these objects depend only on \(\xi|_{U_E^1}=\xi_{\w}|_{U_E^1}\) (and \(E \hookrightarrow A\)).

\item[Definition of an irreducible representation \(\Lambda_{\mathrm{t}}\) of \(\bfJ\)]
from \(\xi_{\mathrm{t}}\) we construct an irreducible smooth representation \(\Lambda_{\mathrm{t}}\) of \(\bfJ\) trivial on \(J^1\) as follows.
Note that we have
\[
\bfJ/J^1 \simeq (\bfJ \cap A_0^{\times})/(J^1 \cap A_0^{\times}), \quad
\bfJ \cap A_0^{\times}=E^{\times}U_{\frakA_0}=E_0^{\times}U_{\frakA_0}, \quad
J^1 \cap A_0^{\times}=U_{\frakA_0}
\]
and in particular \(E_0^{\times}\) is central in \(\bfJ \cap A_0^{\times}\).
As \(\xi_{\mathrm{t}}\) is tame, \(\xi_{\mathrm{t}}|_{U_E}\) is the inflation of a character \(\overline{\xi}_{\mathrm{t}}\) of \(k_E^{\times}\).
By the Green parametrization \cite{MR72878}, \(\overline{\xi}_{\mathrm{t}}\) yields an irreducible cuspidal representation \(\overline{\lambda}\) of \(\GL_l(k_{D_0}) \simeq U_{\frakA_0}/U_{\frakA_0}^1 \simeq J^0/J^1\).
Let \(\lambda\) be the inflation of \(\overline{\lambda}\) to \(J^0\).
We define an extension \(\Lambda_{\mathrm{t}}\) of \(\lambda\) to \(\bfJ=E_0^{\times}J^0\) by requiring that \(\Lambda|_{E_0^{\times}}\) is \(\xi_{\mathrm{t}}|_{E_0^{\times}}\)-isotypic.

For this part, see \cite[\S 4.3, \S 4.2, \S 2.4]{MR2848585}, especially the paragraph after \S 4.3 Lemma 2.
\begin{rem}
An axiom of an admissible pair implies that \(\overline{\xi}_{\mathrm{t}}\) is \(k_E/k_{E_0}\)-regular (in the sense recalled in \cite[\S 2.4 Remark 1]{MR2848585}).
While the Green parametrization also works more generally for a \(k_E/k_{D_0}\)-regular character and produces an irreducible cuspidal representation of \(\GL_l(k_{D_0})\),
the stronger property of \(k_E/k_{E_0}\)-regularity eventually ensures \(\delta({}_{\bfG}\pi_{(E/F, \xi)}^{\BH})=n\),
or equivalently the normalizer of \(\lambda\) in \(A_0^{\times}\) being \(\bfJ \cap A_0^{\times}=E_0^{\times}U_{\frakA_0}\) (rather than a larger subgroup in the normalizer \(\frakK_{\frakA_0}=D_0^{\times}U_{\frakA_0}\) of \(\frakA_0\) in \(A_0^{\times}\)).
See \cite[\S 2.4, \S 2.6]{MR2848585} for the definition of the parametric degree \(\delta\).
\end{rem}

\item[Definition of a character \(\theta\) of \(H^1\)]
from \(\xi\) we construct a character \(\theta\) of \(H^1\) as follows.
As will be evident, this construction only depends on \(\xi|_{U_E^1}=\xi_{\w}|_{U_E^1}\).

Recall from Section \ref{sec:classical-Howe-factorization} that we have a factorization of \(\xi\) in terms of characters \(\xi_k\) of \(E_k^{\times}\):
\[
\xi=\xi_{-1} \cdot (\xi_0 \circ \Nr_{E/E_0}) \cdots (\xi_t \circ \Nr_{E/E_t}).
\]
Among other properties, these characters satisfy the following:
\begin{itemize}
\item For \(k=-1\), the character \(\xi_{-1}\) is tamely ramified.
\item For \(0\leq k\leq t-1\), the \(E\)-level of \(\xi_k \circ \Nr_{E/E_k}\) is \(a_k\).
\end{itemize}
We use \(\xi_k\) for \(0\leq k\leq t\) to define \(\theta\).
We fix a character \(\psi_F\) of \(F\) which is trivial on \(\frakp_F\) but not on \(\calO_F\).
Let \(c_k \in \frakp_E^{-a_k} \cap E_k\) be an element such that
\[
\xi_k (1+x)=\psi_F(\Tr_{E_k/F}(c_kx))
\]
for \(x \in \frakp_E^{\frac{a_k}{2}+} \cap E_k\).
Note that this condition determines the coset \(c_k+(\frakp_E^{-\frac{a_k}{2}} \cap E_k) \) uniquely.
We define a character \(\psi_{c_k}\) of \(U_{\frakA_t}^{\frac{1}{2}a_k e(\frakA/\calO_E)+}\) by
\[
\psi_{c_k}(1+x):= \psi_F(\Trd_{A_t/F}(c_kx))
\]
for \(x\in \frakP_{\frakA_t}^{\frac{1}{2}a_k e(\frakA_t/\calO_E)+}\).
With \(\psi_{c_k}\), we define characters \(\theta_k\) of
\[
U_{\frakA_k}^{\frac{1}{2}a_{k-1}e(\frakA_k/\calO_E)+}U_{\frakA_{k+1}}^{\frac{1}{2}a_ke(\frakA_{k+1}/\calO_E)+} \cdots U_{\frakA_t}^{\frac{1}{2}a_{t-1}e(\frakA_t/\calO_E)+}
\]
for \(0\leq k\leq t\) inductively from \(k=t\) as follows.
We first define \(\theta_t\) by
\[
\theta_t:=\xi_t \circ \Nrd_{A/F} |_{U_{\frakA_t}^{\frac{1}{2}a_{t-1}e(\frakA_t/\calO_E)+}}.
\]
Then inductively, \(\theta_k\) is defined using \(\theta_{k+1}\) by
\[
\theta_k:=
\begin{cases}
\psi_{c_k} \theta_{k+1} \text{ on \(U_{\frakA_{k+1}}^{\frac{1}{2}a_ke(\frakA_{k+1}/\calO_E)+} \cdots U_{\frakA_t}^{\frac{1}{2}a_{t-1}e(\frakA_t/\calO_E)+}\)}, \\
(\xi_k \circ \Nrd_k) \cdot (\xi_{k+1} \circ \Nrd_{k+1}) \cdots (\xi_t \circ \Nrd_t) \text{ on \(U_{\frakA_k}^{\frac{1}{2}a_{k-1}e(\frakA_{k}/\calO_E)+}\)},
\end{cases}
\]
where
we write \(\Nrd_k \colon A_k^{\times} \to E_k^{\times}\) for the reduced norm map.
Finally, we put \(\theta:=\theta_0\).
We can easily check that \(\theta\) does not depend on the choice of \((\xi_k)_{k=-1}^t\) and only depends on \(\xi|_{U_E^1}=\xi_{\w}|_{U_E^1}\).

Note that we have
\begin{align*}
\theta|_{U_{\frakA_0}^1}
&=(\xi_0 \circ \Nrd_0) \cdot (\xi_1 \circ \Nrd_1) \cdots (\xi_t \circ \Nrd_t)|_{U_{\frakA_0}^1} \\
&=(\xi_0 \cdot (\xi_1 \circ \Nr_{E_0/E_1}) \cdots (\xi_t \circ \Nr_{E_0/E_t})) \circ \Nrd_0|_{U_{\frakA_0}^1}.
\end{align*}
Thus from \(\theta\) we can recover \(\xi_0 \cdot (\xi_1 \circ \Nr_{E_0/E_1}) \cdots (\xi_t \circ \Nr_{E_0/E_t})|_{U_{E_0}^1}\)
and hence also
\[
\xi_{\w}|_{U_E^1}=\xi|_{U_E^1}=(\xi_0 \cdot (\xi_1 \circ \Nr_{E_0/E_1}) \cdots (\xi_t \circ \Nr_{E_0/E_t})) \circ \Nr_{E/E_0}|_{U_E^1}.
\]

\begin{rem}
This part may look significantly different from the corresponding part in \cite{MR2848585}.
In the second paragraph and Lemma 1 of \cite[\S 4.3]{MR2848585},
one defines \(\xi_{\w}\) from \(\theta\) as follows:
\begin{itemize}
\item The restriction \(\theta |_{U_{\frakA_0}^1}\) factors through \(\Nrd_0\), namely, there exists a unique character \(\xi'\) of \(U_{E_0}^1\) such that
\[
\theta|_{U_{\frakA_0}^1}=\xi' \circ \Nrd_0.
\]
\item There exists a unique character \(\xi_{\w}\) of \(E^{\times}\) such that
\begin{itemize}
\item \(\xi_{\w}|_{U_E^1}=\xi' \circ \Nr_{E/E_0}\),
\item \(\varpi_F \in \Ker \xi_{\w}\), and 
\item \(\xi_{\w}\) has \(p\)-power order.
\end{itemize}
\end{itemize}
Note that the notation is not entirely consistent; for example, \(A_0\) in this paper is written as \(B\) in \cite{MR2848585}
and \(\xi'\) in the above is denoted by \(\xi_0\) in \cite{MR2848585}.

Here we follow \cite[\S 3.4]{MR3505131} and use the fact that \(\theta\) is a \emph{simple character} (namely \(\theta\) lies in \(\calC (\frakA, \beta, \psi_F)\); cf.\ \cite[\S 2.5]{MR2848585}, \cite{MR2081220}) to define the character \(\theta\) explicitly from \(\xi_{\w}\).
There are some minor typos in \cite[page 362]{MR3505131} and we corrected them in the above (see \cite[page 1731 (i)]{MR3509939} for the case \(\bfG^{\ast}=\GL_n\)).
\end{rem}

\item[Definition of an irreducible representation \(\eta\) of \(J^1\)]
there exists a unique irreducible smooth representation \(\eta\) of \(J^1\) containing \(\theta\). We only remark that the proof is based on the fact that \(J^1/\Ker \theta\)  is a Heisenberg \(p\)-group and \(H^1/\Ker \theta\) is its center.

For this part, see \cite[(2.5.1) (2)]{MR2848585}
and the paragraph after 4.3 Lemma 1.

\item[Definition of an irreducible representation \(\Lambda_{\w}\) of \(\bfJ\)]
recall that we fixed a uniformizer \(\varpi_F\) of \(F\).
We extend \(\eta\) to an irreducible smooth representation \(\Lambda_{\w}\) of \(\bfJ\) characterized by the following conditions:
\begin{description}
\item[(Ext1)] The restriction \(\Lambda_{\w}|_{J^1}\) is isomorphic to \(\eta\).
\item[(Ext2)] The restriction \(\Lambda_{\w}|_{J^0}\) is intertwined by \(A_0^{\times}\), that is, for any \(a \in A_0^{\times}\), we have
\[
\Hom_{J^0\cap J^{0, a}}(\Lambda_{\w}|_{J^0 \cap J^{0, a}}, \Lambda_{\w}^a|_{J^0 \cap J^{0, a}}) \neq 0,
\]
where \(J^{0, a}=a^{-1}J^0a\) is the conjugate and \(\Lambda_{\w}^a\) is the representation of \(J^{0, a}\) obtained as the conjugation of \(\Lambda_{\w}\).
\item[(Ext3)] We have \(\varpi_F \in \Ker \Lambda_{\w}\). In particular, \(\det \Lambda_{\w}\) has finite order.
\item[(Ext4)] The determinant character \(\det \Lambda_{\w}\) has \(p\)-order.
\end{description}

Note that \(\Lambda_{\w}\) depends on \(\xi|_{U_E^1}=\xi_{\w}|_{U_E^1}\) and the choice of \(\varpi_F\).
An extension of \(\eta\) to an irreducible smooth representation of \(\bfJ\) satisfying (Ext1) and (Ext2) is called a wide extension of \(\eta\) in \cite{MR2848585}.
The existence and the uniqueness of a wide extension satisfying (Ext3) and (Ext4) are proved in \cite[\S 4.3 Lemma 2]{MR2848585}.

\begin{rem} \label{rem:wide-extension}
In fact, (Ext2) is not necessary for the characterization of the representation \(\Lambda_{\w}\). This is proved in the same way as \cite[Remark A.1]{MR4206603}.
In the proof of Proposition \ref{prop:wild-part} below, we will only verify (Ext1), (Ext3) and (Ext4) to show that a certain representation appearing in the construction of \({}_{\bfG}\pi_{(\bfS, \xi)}^{\KY}\) is isomorphic to \(\Lambda_{\w}\).
\end{rem}

\item[Definition of \({}_{\bfG}\pi_{(E/F, \xi)}^{\BH}\)]
we put \(\Lambda_{\xi}:=\Lambda_{\mathrm{t}} \otimes \Lambda_{\w}\) and finally define \({}_{\bfG}\pi_{(E/F, \xi)}^{\BH}:=\cInd_{\bfJ}^{\bfG(F)}\Lambda_{\xi}\).
We can easily check that \(\Lambda_{\xi}\) is independent of the choice of \(\varpi_F \in F\) and that \({}_{\bfG}\pi_{(E/F, \xi)}^{\BH}\) is independent of the choice of an \(F\)-embedding \(E \hookrightarrow A\).
\end{description}

On the other hand, Kaletha's construction starts with
taking a regular Yu-datum corresponding to \((\bfS, \xi)\) by \cite[Proposition 3.7.8]{MR4013740} (recalled in Section \ref{sec:rsc-rep})
and then Yu's construction \cite{MR1824988} yields an open subgroup \(K\) of \(\bfG(F)\) and an irreducible smooth representation \(\kappa\) of \(K\) which give rise to a supercuspidal representation \({}_{\bfG}\pi_{(\bfS, \xi)}^{\KY}:=\cInd_K^{\bfG(F)} \kappa\).
As in the proof of the surjectivity of \cite[Proposition 3.7.8]{MR4013740}, we use a Howe factorization of \((\bfS, \xi)\) as in Section \ref{sec:Howe-factorization} to construct a regular Yu-datum as above.
We write \(\phi_k \colon \bfG^k(F) \to \bbC^{\times}\) \((k=-1, \dots, t)\) and \(\Psi=(\bfG^0 \subsetneq \bfG^1 \subsetneq \cdots \subsetneq \bfG^t, \pi_{-1}, (\phi_0, \dots, \phi_t))\) for the Howe factorization and the associated regular Yu-datum.
For convenience, we take the Howe factorization \((\phi_k)_{k=-1}^t\) 
which induces \((\xi_k)_{k=-1}^t\) appearing in the construction of \({}_{\bfG}\pi_{(E/F, \xi)}^{\BH}\) in the way explained in Section \ref{sec:classical-Howe-factorization}.
In particular, we have the following:
\begin{itemize}
\item \(\bfG^k\) is as in Section \ref{sec:Howe-factorization},
\item \(\phi_k=\xi_k \circ \Nrd_k\) for \(k=-1, \dots, t\) with \(\xi_k\) and \(\Nrd_k\) as in the definition of \(\theta\),
\item \(\xi=\prod_{k=-1}^t \phi_k|_{\bfS(F)}\), or equivalently \(\xi=\prod_{k=-1}^t(\xi_k \circ \Nr_{E/E_k})\), and
\item \(\pi_{-1}:=\pi_{(\bfS, \phi_{-1})}^{\bfG^0}\) for the irreducible depth-zero supercuspidal representation of \(\bfG^0(F)\) associated to \((\bfS, \phi_{-1})\) (denoted by \(\pi_{(\bfS, \phi_{-1})}\) in \cite[\S 3.4]{MR4013740}).
\end{itemize}
According to \cite[\S 3.4]{MR2431732}, we may express \(\kappa\) as a tensor product \(\bigotimes_{k=-1}^t \kappa_k\) of irreducible smooth representations, where each \(\kappa_k\) only depends on \(\phi_k\) for \(0\leq k\leq t\) and \(\kappa_{-1}\) only depends on \(\pi_{-1}\).
As this expression is more convenient for our argument, below we often cite \cite{MR2431732}.

In the following, we will see that \(K\supset \bfJ\) holds and moreover show that if we take a suitable regular Yu-datum \(\underline{\Psi}\), 
which is \(\bfG\)-equivalent to \(\Psi\), 
then 
\(\kappa_{-1}\) (resp.\ \(\bigotimes_{k=0}^{t}\kappa_k\)) is related to 
\(\Lambda_{\mathrm{t}}\) (resp.\ \(\Lambda_{\w}\)) in a simple way.
The desired isomorphism \({}_{\bfG}\pi_{(\bfS, \xi)}^{\KY} \simeq {}_{\bfG}\pi_{(E/F, \xi)}^{\BH}\) will follow immediately from this.

\begin{rem} \label{rem:reduced-extended}
Here, strictly speaking, a regular Yu-datum \(\Psi\) defines \((K, \kappa)\) and related objects only up to conjugacy (of course this does not affect the isomorphism class of the resulting supercuspidal representation \(\cInd_K^{\bfG(F)} \kappa\)).
This is because a choice of an expression 
\[
\pi_{-1} \simeq \cInd_{\bfG^0(F)_{[x]}}^{\bfG^0(F)} \rho
\]
is involved in the definition and the pair \((\bfG^0(F)_{[x]}, \rho)\) is determined from \(\pi_{-1}\) only up to \(\bfG^0(F)\)-conjugacy (cf.\ the first paragraph of page 151 of \cite{MR2431732}). 
However, we implicitly make a choice naturally attached to our situation; namely we choose \(x\) to be the point in the extended Bruhat--Tits building of \(\bfG^0\) over \(F\) associated to \(\bfS \subset \bfG^0\) (cf.\ Section \ref{sec:epsilon-character})
and \(\rho\) to be the irreducible smooth representation of \(\bfG^0(F)_{[x]}\) canonically obtained from the construction of \(\pi_{(\bfS, \phi_{-1})}^{\bfG^0}\) from \((\bfS, \phi_{-1})\) (see the proof of Proposition \ref{prop:tame-part} below).
\end{rem}

In \cite[page 591]{MR1824988} (and also \cite[page 52]{MR2431732}) the open subgroup \(K \subset \bfG(F)\) is defined in terms of the Moy--Prasad subgroups as
\[
K:=K^t:=\bfG^0(F)_{[x]}\bfG^1(F)_{x, s_0} \cdots \bfG^t(F)_{x, s_{t-1}},
\]
where
\begin{itemize}
\item \([x] \in \calB_{\red}(\bfG^0, F)\) is the point in the reduced Bruhat--Tits building associated to \(\bfS \subset \bfG\) (cf.\ Section \ref{sec:epsilon-character}),
\item \(x \in \calB(\bfG^0, F)\) is a point mapping to \([x] \in \calB_{\red}(\bfG^0, F)\), which is also identified with the image in \(\calB(\bfG^k, F)\) for \(0\leq k\leq t\), and
\item \(s_k:=\frac{r_k}{2}=\frac{\depth(\phi_k)}{2}\) for \(k=0, \dots, t-1\).
\end{itemize}

Recalling that we arranged the embedding \(E \hookrightarrow A\) to be compatible with the inclusion \(\bfS \hookrightarrow \bfG\),
we can compare various subgroups appearing in the construction of \({}_{\bfG}\pi_{(E/F, \xi)}^{\BH}\) and the Moy--Prasad subgroups.
This is done in \cite[\S 6]{2112.12367}, using \cite{MR1888474}.
In particular, we have the following equalities:
\begin{align*}
\bfG^0(F)_{[x]}&=\frakK_{\frakA_0}=D_0^{\times}U_{\frakA_0}, \\
\bfG^k(F)_{x, s}&=U_{\frakA_k}^{se(\frakA_k/\calO_F)}, \\
\bfG^k(F)_{x, s+}&=U_{\frakA_k}^{se(\frakA_k/\calO_F)+}
\end{align*}
for any \(0\leq k\leq t\) and \(s\in \bbR_{\geq 0}\).
Thus we have \(K \supset \bfJ\) (this also follows from \cite[Proposition 10.2 (2), (3)]{2112.12367}).
Similarly, we find that the compact open subgroup
\[
K_{+}^t:=\bfG^0(F)_{x, 0+}\bfG^1(F)_{x, s_0+} \cdots \bfG^t(F)_{x, s_{t-1}+}
\]
defined in \cite[page 591]{MR1824988} (and \cite[page 52]{MR2431732})
is nothing but \(H^1\) (this also appears as \cite[Proposition 10.2 (1)]{2112.12367}).

The conditions required of the factorization \(\xi=\xi_{\mathrm{t}}\xi_{\w}\) are not exactly consistent with those imposed on a Howe factorization in Definition \ref{def:Howe-facto} (\cite[Definition 3.6.2]{MR4013740}), so we slightly modify the latter factorization in the following proposition to eventually obtain a \(\bfG\)-equivalent regular Yu-datum.

\begin{prop}\label{prop:refactorization}
There exists a regular Yu-datum \(\underline{\Psi}=(\bfG^0 \subsetneq \bfG^1 \subsetneq \cdots \subsetneq \bfG^t, \underline{\pi}_{-1}, (\underline{\phi}_0, \dots, \underline{\phi}_t))\) which is \(\bfG\)-equivalent to \(\Psi\) and satisfies the following properties:
\begin{itemize}
\item[(i)] 
\(\underline{\pi}_{-1} \simeq \pi_{(\bfS, \xi_{\mathrm{t}})}^{\bfG^0}\),
\item[(ii)] 
\(\prod_{k=0}^{t} \underline{\phi}_k|_{\bfS(F)}=\xi_{\w}\),
\item[(iii)] 
\(\underline{\phi}_k|_{\bfG^k(F)_{x, 0+}}=\phi_k|_{\bfG^k(F)_{x, 0+}}\) for \(k =0, \dots, t\), and
\item[(iv)] 
\(\underline{\phi}_k\) has finite \(p\)-power order for \(k =0, \dots, t\).
\end{itemize}
\end{prop}
\begin{proof}
The proof is similar to that of \cite[Proposition A.3]{MR4206603}.
Here we only sketch the construction of \(\underline{\Psi}\).
For this, we need to define \(\underline{\pi}_{-1}\) and \(\underline{\phi}_{k}\) for \(k =0, \dots, t\).
As for \(\underline{\pi}_{-1}\), we simply put
\[
\underline{\phi}_{-1}:=\xi_{\mathrm{t}}, \quad \underline{\pi}_{-1}:=\pi_{(\bfS, \underline{\phi}_{-1})}^{\bfG^0}.
\]
To define \(\underline{\phi}_{k}\) for \(k =0, \dots, t\), we take auxiliary characters \(\underline{\xi}_k\) of \(E_k^{\times}\) satisfying the following conditions:
\begin{itemize}
\item \(\underline{\xi}_k|_{U_{E_k}^1}=\xi_k|_{U_{E_k}^1}\),
\item \(\varpi_F \in \Ker (\underline{\xi}_k \circ \Nr_{E/E_k})\),
\item \(\underline{\xi}_k\) has finite \(p\)-power order, and
\item if \(\phi_t=\mathbbm{1}\), then \(\underline{\xi}_t=\mathbbm{1}\).
\end{itemize}
Then we define
\[
\underline{\phi}_k:=\underline{\xi}_k \circ \Nrd_k.
\]
Note that the above conditions do not specify the characters \(\underline{\xi}_k\) uniquely, but this ambiguity does not matter for the argument.
\end{proof}

Let \(\underline{\Psi}\) be a regular Yu-datum as in Proposition \ref{prop:refactorization} and 
\(\kappa_k\) be the irreducible smooth representation of \(K\) constructed from \(\underline{\Psi}\) for \(k=-1, \dots, t\).
Here, again as in Remark \ref{rem:reduced-extended}, we implicitly
take the same \(x\) as before and the irreducible smooth representation \(\underline{\rho}\) of \(\bfG^0(F)_{[x]}\) canonically obtained from the construction of 
\(\underline{\pi}_{-1}=\pi_{(\bfS, \xi_{\mathrm{t}})}^{\bfG^0}\).

\begin{prop} \label{prop:tame-part}
With the choice of a regular Yu-datum \(\underline{\Psi}=(\bfG^0 \subsetneq \bfG^1 \subsetneq \cdots \subsetneq \bfG^t, \underline{\pi}_{-1}, (\underline{\phi}_0, \dots, \underline{\phi}_t))\) as in Proposition \ref{prop:refactorization},
we have \(\kappa_{-1}\simeq \cInd_{\bfJ}^K \Lambda_{\mathrm{t}}\).
\end{prop}
\begin{proof}
The proof is similar to that of \cite[Proposition A.4]{MR4206603}
even though in loc.\ cit.\ we have \(K=\bfJ\) and thus \(\cInd_{\bfJ}^K \Lambda_{\mathrm{t}}\) is simply \(\Lambda_{\mathrm{t}}\).
It amounts to directly comparing the construction of \(\Lambda_{\mathrm{t}}\) from \(\xi_{\mathrm{t}}\) and that of \(\kappa_{-1}\) from \(\underline{\pi}_{-1}=\pi_{(\bfS, \xi_{\mathrm{t}})}^{\bfG^0}\) (for details of the latter, see \cite[\S 3.4]{MR4013740} and \cite[page 67]{MR2431732}).

Let us briefly summarize the latter construction
and sketch the comparison.
For this, it may be helpful to keep in mind the following diagram:
\[
\begin{tikzcd}
\bfS(F)\bfG^0(F)_{x, 0}=E_0^{\times}U_{\frakA_0} \arrow[r, hookrightarrow] \arrow[d, hookrightarrow]& K^0=\bfG^0(F)_{[x]}=D_0^{\times}U_{\frakA_0} \arrow[d, hookrightarrow] \\
\bfJ=E_0^{\times}J^0 \arrow[r, hookrightarrow] & K=D_0^{\times}J^0.
\end{tikzcd}
\]
As described in \cite[page 50, page 67]{MR2431732}, one takes an expression
\[
\underline{\pi}_{-1} \simeq \cInd_{\bfG^0(F)_{[x]}}^{\bfG^0(F)} \underline{\rho}
\]
as in \cite[page 590, D4]{MR1824988}
and performs the inflation operation \(\kappa_{-1} =\inf_{K^0}^K(\underline{\rho})\) from \(K^0=\bfG^0(F)_{[x]} \subset \bfG^0(F)\) to \(K \subset \bfG(F)\).
In our situation the construction of \(\underline{\pi}_{-1}=\pi_{(\bfS, \xi_{\mathrm{t}})}^{\bfG^0}\) provides a natural choice for \(\underline{\rho}\) as follows.
We recall from \cite[page 1095]{MR4013740} that \(\pi_{(\bfS, \xi_{\mathrm{t}})}^{\bfG^0}\) is defined as
\[
\pi_{(\bfS, \xi_{\mathrm{t}})}^{\bfG^0}:=\cInd_{\bfS(F)\bfG^0(F)_{x, 0}}
^{\bfG^0(F)}\kappa_{(\bfS, \xi_{\mathrm{t}})}^{\bfG^0}
\]
for a certain representation \(\kappa_{(\bfS, \xi_{\mathrm{t}})}^{\bfG^0}\) associated to \((\bfS, \xi_{\mathrm{t}})\).
Thus, we set
\[
\underline{\rho}:=\cInd_{\bfS(F)\bfG^0(F)_{x, 0}}^{\bfG^0(F)_{[x]}}\kappa_{(\bfS, \xi_{\mathrm{t}})}^{\bfG^0}.
\]
Now, examining the construction of \(\kappa_{(\bfS, \xi_{\mathrm{t}})}^{\bfG^0}\) in \cite[page 1095]{MR4013740}
shows that \(\Lambda_{\mathrm{t}}\) is isomorphic to the representation inflated from \(\kappa_{(\bfS, \xi_{\mathrm{t}})}^{\bfG^0}\) by using the isomorphism
\(\bfS(F)\bfG^0(F)_{x, 0}/\bfG^0(F)_{x, 0+}=E_0^{\times}U_{\frakA_0}/U_{\frakA_0}^1 \simeq E_0^{\times}J^0/J^1\)
and the fact that \(\kappa_{(\bfS, \xi_{\mathrm{t}})}^{\bfG^0}\) is trivial on \(\bfG^0(F)_{x, 0+}=U_{\frakA_0}^1\).
Then we easily conclude that \(\kappa_{-1}\simeq \cInd_{\bfJ}^K \Lambda_{\mathrm{t}}\), as desired.

Note that in the definition of \(\kappa_{(\bfS, \xi_{\mathrm{t}})}^{\bfG^0}\),
one associates to a character of \(k_E^{\times}\) an irreducible cuspidal representation of \(\GL_l(k_{D_0})\) by using the Deligne--Lusztig induction \cite{MR0393266} associated to \(\Res_{k_E/k_F}\Gm \subset \Res_{k_{D_0}/k_F} \GL_{l, k_{D_0}}\).
By \cite[Corollaire 3.2]{MR1709085}, it is naturally identified with the Deligne--Lusztig induction associated to \(\Res_{k_E/k_{D_0}}\Gm \subset \GL_{l, k_{D_0}}\), which yields the Green parametrization of \(\GL_l(k_{D_0})\), as is well-known.
\end{proof}

\begin{prop} \label{prop:wild-part}
With the choice of a regular Yu-datum \(\underline{\Psi}=(\bfG^0 \subsetneq \bfG^1 \subsetneq \cdots \subsetneq \bfG^t, \underline{\pi}_{-1}, (\underline{\phi}_0, \dots, \underline{\phi}_t))\) as in Proposition \ref{prop:refactorization}, we have \(\bigotimes_{k=0}^{t} \kappa_k |_{\bfJ}\simeq \Lambda_{\w}\).
\end{prop}
\begin{proof}
The proof is similar to that of \cite[Proposition A.5]{MR4206603}
even though in loc.\ cit.\ we have \(\bfJ=K\) and thus \(\bigotimes_{k=0}^{t} \kappa_k |_{\bfJ}\) is simply \(\bigotimes_{k=0}^{t} \kappa_k\) .
Here we only outline the argument.

We need to check that the representation \(\bigotimes_{k=0}^{t} \kappa_k|_{\bfJ}\) of \(\bfJ\) satisfies the characterizing conditions of \(\Lambda_{\w}\).
By Remark \ref{rem:wide-extension}, it suffices to do so for the conditions (Ext1), (Ext3) and (Ext4).

Let us show that \(\bigotimes_{k=0}^t \kappa_k|_{\bfJ}\) satisfies (Ext1), that is,
\[
\bigotimes_{k=0}^t \kappa_k|_{J^1} \simeq \eta.
\]
As the first step in constructing \(\kappa_k\), a character \(\underline{\hat{\phi}}_k\) of \(K_{+}^t\) is constructed from \(\underline{\phi}_k\) in \cite[page 591]{MR1824988}.
The proof of \cite[Proposition 10.4]{2112.12367}, together with \cite{MR1888474}, gives the following expression of \(\underline{\hat{\phi}}_k\):
\[
\underline{\hat{\phi}}_k=
\begin{cases}
\psi_{\underline{c}_k} &\text{ on \(U_{\frakA_{k+1}}^{\frac{1}{2}a_ke(\frakA_{k+1}/\calO_E)+}U_{\frakA_{k+2}}^{\frac{1}{2}a_{k+1}e(\frakA_{k+2}/\calO_E)+} \cdots U_{\frakA_t}^{\frac{1}{2}a_{t-1}e(\frakA_t/\calO_E)+}\)} \\
\underline{\xi}_k \circ \Nrd_k &\text{ on \(U_{\frakA_0}^{\frac{1}{2}a_{-1}e(\frakA_0/\calO_E)+}U_{\frakA_1}^{\frac{1}{2}a_0e(\frakA_1/\calO_E)+} \cdots U_{\frakA_k}^{\frac{1}{2}a_{k-1}e(\frakA_{k}/\calO_E)+}\)},
\end{cases}
\]
where
\begin{itemize}
\item \(\underline{\xi}_k\) is as in the proof of Proposition \ref{prop:refactorization},
\item \(\underline{c}_k\) is defined in the same way as \(c_k\) with \(\underline{\xi}_k\) in place of \(\xi_k\), namely an element in \(\frakp_E^{-a_k} \cap E_k\) such that
\[
\underline{\xi}_k(1+x)=\psi_F(\Tr_{E_k/F} (\underline{c}_kx))
\]
for \(x\in \frakp_E^{\frac{a_k}{2}+} \cap E_k\), and
\item \(\psi_{\underline{c}_k}\) is defined in the same way as \(\psi_{c_k}\) in the definition of \(\theta\) with \(\underline{c}_k\) in place of \(c_k\).
\end{itemize}
By (iii) of Proposition \ref{prop:refactorization}, we have \(\underline{\xi}_k|_{U_{E_k}^1}=\xi_k|_{U_{E_k}^1}\) and thus
\(\psi_{\underline{c}_k}=\psi_{c_k}\).
Therefore, we obtain the equality
\[
\prod_{k=0}^t \underline{\hat{\phi}}_k=\theta
\]
of characters of \(K_{+}^t=H^1\).
As recalled earlier, \(\eta\) is defined to be the unique irreducible smooth representation of \(J^1\) containing the character \(\theta\) of \(H^1\).
Since \(\bigotimes_{k=0}^t \kappa_k|_{K_{+}^t}\) does contain \(\prod_{k=0}^t \underline{\hat{\phi}}_k\) by \cite[Lemma 3.27]{MR2431732}, it is enough to show \(\dim \bigotimes_{k=0}^t \kappa_k=\dim \eta\).
This is done in the proof of \cite[Proposition 10.5]{2112.12367}.
Thus, \(\bigotimes_{k=0}^t \kappa_k|_{\bfJ}\) indeed satisfies (Ext1).

The condition (Ext3) is easy to verify.
Writing \(Z(\bfG)\) for the center of \(\bfG\), we use the fact that \(\bigotimes_{k=0}^t \kappa_k|_{Z(\bfG)(F)}\) is \(\prod_{k=0}^t \underline{\phi}_k|_{Z(\bfG)(F)}\)-isotypic, which is easy to check, and the equality \(\prod_{k=0}^t \underline{\phi}_k|_{Z(\bfG)(F)}=\xi_{\w}|_{Z(\bfG)(F)}\), which is immediate from (ii) of Proposition \ref{prop:refactorization}.

To verify that \(\bigotimes_{k=0}^t \kappa_k|_{\bfJ}\) satisfies (Ext4),
we prove a slightly stronger assertion:
the character \(\det \kappa_k\) of \(K\) has \(p\)-power order for \(0\leq k\leq t\).
For this, we examine the definition of \(\kappa_k\) in \cite[\S 3.4]{MR2431732} and
use (iv) of Proposition \ref{prop:refactorization}.
The key step of the proof is to show that \(\det \omega_k\) has \(p\)-power order, where \(\omega_k\) is a certain representation used in the definition of \(\kappa_k\) and defined by means of the Heisenberg--Weil construction.
 \end{proof}
Now we prove that \({}_{\bfG}\pi_{(\bfS, \xi)}^{\KY} \simeq {}_{\bfG}\pi_{(E/F, \xi)}^{\BH}\).
By Proposition \ref{prop:tame-part}, ``the projection formula" and Proposition \ref{prop:wild-part}, we see
\begin{align*}
\kappa &= \bigotimes_{k=-1}^t \kappa_k
\simeq (\cInd_{\bfJ}^K \Lambda_{\mathrm{t}}) \otimes \bigotimes_{k=0}^t \kappa_k
\simeq \cInd_{\bfJ}^K \Big(\Lambda_{\mathrm{t}} \otimes \bigotimes_{k=0}^t \kappa_k|_{\bfJ}\Big) \\
& \simeq \cInd_{\bfJ}^K (\Lambda_{\mathrm{t}} \otimes \Lambda_{\w})
= \cInd_{\bfJ}^K \Lambda_{\xi}.
\end{align*}
Thus, we conclude:
\[
{}_{\bfG}\pi_{(\bfS, \xi)}^{\KY}=\pi_{\Psi}
\simeq \pi_{\underline{\Psi}} =\cInd_K^{\bfG(F)} \kappa
\simeq \cInd_{\bfJ}^{\bfG(F)} \Lambda_{\xi}={}_{\bfG}\pi_{(E/F, \xi)}^{\BH}.
\]

\bibliographystyle{my_amsalpha}
\bibliography{myrefs}

\end{document}